\documentclass[12pt]{amsart}
\usepackage{amsmath,amsfonts,amsthm,amscd, amssymb,mathrsfs}
\usepackage[all,cmtip]{xy}
\usepackage{amssymb}
\usepackage{graphicx}
\usepackage{amsmath}
\usepackage{latexsym}
\usepackage{amscd}
\usepackage{mathrsfs}
\usepackage{enumitem} 
\usepackage{tikz-cd}
\usepackage{tikz}
\usepackage{circuitikz}
\usepackage{braket}
\usepackage[margin=1.0in]{geometry}
\usepackage{fix-cm}
\usepackage{bm}
\usepackage{anyfontsize}
\usepackage{color}
\usepackage{marginnote}
\usepackage{verbatim}
\usepackage{imakeidx}
\makeindex[title=Index of Notation]

\usepackage{hyperref}

\allowdisplaybreaks 

\usepackage[backend=biber,maxnames=99,style=alphabetic]{biblatex} 
\newtheorem{theorem}{Theorem}[section]
\newtheorem{corollary}[theorem]{Corollary}
\newtheorem{example}[theorem]{Example}
\newtheorem{lemma}[theorem]{Lemma}
\newtheorem{proposition}[theorem]{Proposition}
\theoremstyle{definition}
\newtheorem{definition}[theorem]{Definition}
\newtheorem{remark}[theorem]{Remark}
\newcommand{\N}{\mathbb{N}}
\newcommand{\R}{\mathbb{R}}
\newcommand{\Q}{\mathbb{Q}}
\newcommand{\C}{\mathbb{C}}
\newcommand{\Z}{\mathbb{Z}}
\newcommand{\W}{\mathcal{W}}

\DeclareMathOperator{\diam}{diam}

\DeclareMathOperator{\ric}{Ric}
\DeclareMathOperator{\inj}{inj}
\DeclareMathOperator{\vol}{vol}
\title{Long-Time Behavior of the $\alpha$-Ricci Flow on the 2-Torus}
\author{Yifan Guo}
\author{Xiaokang Wang}
\address{Courant Institute, New York University, 251 Mercer St, New York, NY
10012, USA}
\email{yg3991@nyu.edu}
\address{Department of Mathematics, University of California, Irvine, CA 92697, USA}
\email{xiaokaw@uci.edu}
\begin{document}

\begin{abstract}
	We study the $\alpha$-Ricci flow $(g(t),\alpha(t))$, which is a Ricci flow coupled with a harmonic 1-form heat flow, on the 2-torus $M$. We prove global existence for this flow as well as a Type III curvature estimate. The long-time behavior of the $\alpha$-Ricci flow depends on the cohomology class $[\alpha] \in H^1(M,\R)$, which is invariant under the flow. If $[\alpha]$ is rational, then as $t\to\infty$, $(M,t^{-1}g(t))$ Gromov--Hausdorff converges to $S^1$ with length $\ell([\alpha])/\sqrt 2$ where $\ell([\alpha])$ is the primitive length of $[\alpha]$.  On the universal cover, the rescaled flow $(s^{-1}\tilde g(st),\tilde\alpha(st))$ smoothly converges to an expander $(g^*(t),\alpha^*(t))$ on $\R^2$ as $s\to \infty$ where $\alpha^*(t)$ is nonzero and constant in spacetime and $g^*(t)$ is flat and expanding in the $\alpha^*$ direction.
    If $[\alpha]$ is irrational, then there is a sequence $s_i\to\infty$ such that 
    $(M,s_i^{-1}g(s_it))$ Gromov--Hausdorff converges to a point. Moreover,  
    on the universal cover, $(s_i^{-1}\tilde g(s_it),\tilde\alpha(s_it))$ smoothly converges to $(g^*(t),\alpha^*(t))$.
\end{abstract}

\maketitle
\section{Introduction}

\subsection{Introduction to \texorpdfstring{$\alpha$}{alpha}-Ricci flow}\label{intro}

The $\alpha$-Ricci flow $(M,g(t),\alpha(t))$ is a one-parameter family of metrics $g(t)$ and closed 1-forms $\alpha(t)$ on $M$ satisfying the following equations
\begin{equation}\label{alpha_Ricci_flow}
    \left\{\begin{aligned}
	\frac{\partial g}{\partial t}=&-2\ric+\frac{1}{2}\alpha\otimes\alpha\\
	\frac{\partial \alpha}{\partial t}=&-dd^*_g\alpha
\end{aligned}\right.
\end{equation}
where $d^*_g\alpha=-\operatorname{tr}_g(\nabla\alpha)$.
The $\alpha$-Ricci flow is a coupled system of the heat equation and the Ricci flow equation, introduced in \cite{guo2025alpha}, and shares many similarities with the Ricci flow. We note that since $\alpha$ evolves by an exact form, the cohomology class $[\alpha]\in H^1(M,\R)$ is invariant under the flow. The origin of the $\alpha$-Ricci flow comes from the case when $\alpha(t)=du(t)$ is exact. In this case, the $\alpha$-Ricci flow reduces to the List flow \cite{list2008evolution}. 
The $\alpha$-Ricci flow is closely related to the Ricci-Harmonic map flow \cite{muller2012ricci}, which couples the Ricci flow with the harmonic map heat flow (cf. Remark \ref{r:connection.ricci.harmonic}), the Ricci--Yang--Mills flow
\cite{Streets2010}, which couples the Ricci flow with the Yang--Mills flow,
and the generalized Ricci flow
\cite{GarciaFernandezStreets2021,OliynykSuneetaWoolgar2006,Streets2008},
which couples the Ricci flow with the harmonic 3-form heat flow.

The $\alpha$-Ricci flow has a deep relation with the Ricci flow: Let $\tilde M$ be a cover of $M$ such that $\alpha(t)=du(t)$ on $\tilde M$ and $(g(t),\alpha(t))$ is a solution to \eqref{alpha_Ricci_flow} on $M$. Then $(\tilde M,\bar g(t)=\tilde g(t)+e^{u(t)}g_{S^1})$ satisfies the gauged-Ricci flow equation (cf. \cite{guo2025alpha}). For basic results on the $\alpha$-Ricci flow, together with a surgery construction in three dimension, we refer to \cite{guo2025alpha}.

When $\alpha$ is exact (i.e., when $\alpha=du$), the $\alpha$-Ricci flow on surfaces of arbitrary genus was studied by List \cite{list2008evolution}, together with Lott and Sesum \cite{lott2014ricci}. On the sphere $S^2$, the solution develops a finite-time singularity and shrinks to a round point with constant $u$ \cite{list2008evolution}. On the torus, List flow exists for all time and converges to a flat metric with constant $u$ \cite{lott2014ricci}. On surfaces of genus at least 2,  Lott and Sesum \cite{lott2014ricci} showed that on the thick part, the rescaled metric converges to a hyperbolic metric and $|du|$ converges to~$0$.

Since $H^1(S^2)=0$, the $\alpha$-Ricci flow on the sphere $S^2$ reduces to the List flow. In this paper, we focus on the long-time behavior of the $\alpha$-Ricci flow \eqref{alpha_Ricci_flow} on the 2-torus $\mathbb T^2$ with $0\neq[\alpha]$, leaving the higher-genus case for future study.
\subsection{Estimates for geometric quantities on $\mathbb{T}^2$}

Throughout this article, we write $M=\mathbb{T}^2$ and let $(g(t),\alpha(t))$ be a solution to the $\alpha$-Ricci flow \eqref{alpha_Ricci_flow}.

We first derive estimates for the following geometric quantities. They are generalizations of the corresponding estimates for the List flow obtained in \cite{lott2014ricci}.
\begin{theorem}\label{t:2d_estimates}
    Let $(g(t),\alpha(t))$ be a solution of $\alpha$-Ricci flow on the 2-torus $M=\mathbb{T}^2$ with initial metric $g(0)$ and closed 1-form $\alpha(0)$. Then 
    there exists a constant $C=C(g(0),\alpha(0))$ such that the following hold
    \begin{itemize}
        \item The flow exists on $[0,\infty)$;
        \item The volume $V(t)$ satisfies the bound: $V(0)\le V(t) \le V(0)\sqrt{1+\frac{Ct}{2}}$;
        \item The $\alpha$-norm satisfies $|\alpha(t)|_{g(t)}^2\le \frac{C}{1+Ct/2}$;
        \item Type III estimate: $|R(x,t)|\le \frac{C}{1+t}$;
        \item The systole growth lower bound:
        \begin{align*}
            L_\sigma(t)^2\ge \frac{[\alpha](\sigma)^2}{2}t+L_\sigma(0)^2,
        \end{align*}
        where $\sigma\in H_1(M,\Z)$, $L_\sigma(t)=\inf\{L(\gamma,g(t)):\gamma\in \sigma\}$, and $[\alpha](\sigma):=\int_\sigma\alpha$;
        \item The diameter estimate: $\operatorname{diam}(M,g(t)) \le C(\sqrt{t}+1)$.
    \end{itemize}
\end{theorem}

\begin{remark}
    Note that in the arguments, we use the Gauss--Bonnet theorem in an essential way. Therefore, the estimates are specific to the 2-torus.
\end{remark}

To understand the long-time behavior, a blow-down argument is needed. We consider the blow-down flow $(g_s(t),\alpha_s(t)):=~(s^{-1}g(st), \alpha(st))$. For any fixed $s>0$, $(g_s(t),\alpha_s(t))$ is also a solution to the $\alpha$-Ricci flow \eqref{alpha_Ricci_flow}. By Theorem~\ref{t:2d_estimates}, we have that for any $s\ge1$
\begin{itemize}
    \item  volume collapsing: $V(M,g_s(1)) = O(s^{-1/2})$;
    \item curvature uniformly bounded: $|R(g_s(1))|\leq C$;
    \item diameter uniformly bounded: $\operatorname{diam}(M,g_s(1))\leq C$.
\end{itemize}

As $t\to\infty$, $(M,t^{-1}g(t))$ collapses with bounded curvature and diameter. The collapsing theory of Cheeger-Fukaya-Gromov \cite{cheeger1992nilpotent} can be applied. Using the classification of the bounded collapsing limit of $\mathbb{T}^2$ by \cite{katz2020torus}, we obtain subsequential Gromov--Hausdorff convergence of $(M,t^{-1}g(t))$ to $S^1$ or a point. We also show that along the collapsing sequence, the local fundamental groups inject into the fundamental group $\pi_1(M)$, which implies that the universal cover is noncollapsing and has subsequential limits (cf. Corollary~\ref{c:convergence_CFG}).
Our main goal is to relate these limits---both the Gromov--Hausdorff limit of $(M, t^{-1}g(t))$ and the geometric limits of the universal cover---to the cohomology class $[\alpha]$.

\subsection{Long-time behavior on $\mathbb{T}^2$}

We aim to show that the long-time behavior of the $\alpha$-Ricci flow $(M,g(t),\alpha(t))$ depends on the rationality of the class $[\alpha]$, which remains unchanged under the flow.

On the 2-torus $\mathbb{T}^2$, let $0\neq [\alpha]\in H^1(M,\mathbb{R})$. Consider the $\mathbb{R}$-subgroup, $\mathbb{R}\cdot[\alpha]$, generated by $[\alpha]$ in $H^1(M,\mathbb{R})$. We know that $\mathbb{R}\cdot[\alpha]\cap H^1(M,\mathbb{Z})$ is a proper subgroup of $H^1(M,\mathbb{Z})\cong\Z\times\Z$.
\begin{definition}\label{d:rationality}
     We say that $[\alpha]$ is \textit{rational} if $\mathbb{R}\cdot[\alpha]\cap H^1(M,\mathbb{Z})\cong \mathbb{Z}$, and $[\alpha]$ is \textit{irrational} if $\mathbb{R}\cdot[\alpha]\cap H^1(M,\mathbb{Z}) = \{0\}$.
\end{definition}
 \begin{definition}\label{d:prmitive length}
     If $[\alpha]$ is rational, there exists a unique generator $v$ of $\mathbb{R}\cdot[\alpha]\cap H^1(M,\mathbb{Z})$ such that $[\alpha]=\ell v$ for a unique constant $\ell >0$. We call $\ell=:\ell([\alpha]) $ the \textit{primitive length} of $[\alpha]$. If $[\alpha]$ is irrational, we define the primitive length $\ell([\alpha])=0$ in view of the fact that $\mathbb{R}\cdot[\alpha]\cap H^1(M,\mathbb{Z}) = \{0\}$. 
 \end{definition}
  \begin{remark}\label{r:remark_on_the_rationality}
     When $[\alpha] = 0$, i.e. $\alpha$ is exact, we also define $\ell([\alpha])=0$. Note that we don't say that $[\alpha]$ is irrational in this case since Definition~\ref{d:rationality} requires $[\alpha]\neq0$ (cf. Remark~\ref{r:lott.sesum}).
 \end{remark}

\begin{remark}\label{r:connection.ricci.harmonic}
    If $[\alpha]$ is rational, then there is a map $h(t):M\to S^1$, such that $\alpha(t)=\ell([\alpha])\nabla h(t)=\ell([\alpha])h(t)^* d\theta$ where $d\theta$ is the volume form of $S^1$ under the standard metric. In this case, the $\alpha$-Ricci flow reduces to 
    \begin{align*}
        	\frac{\partial g}{\partial t}=&-2\ric+c\nabla h\otimes\nabla h\\
	\frac{\partial h}{\partial t}=&\tau_{g(t)}h.
    \end{align*}
    where $c=\frac{1}{2}\ell([\alpha])^2$ which is the Ricci-Harmonic map flow introduced in \cite{muller2012ricci}.
\end{remark}
 We introduce notion describing the convergence of the flow.

\begin{definition}[Pointed $C^\infty$ (Cheeger--Gromov) convergence]\label{d:pointed.cinfty.convergence}
Let $\{(M_i, g_i(t),\alpha_i(t), p_i)\}_{i=1}^\infty$ be a sequence of families of metrics and closed 1-forms defined on pointed manifolds $(M_i,p_i)$ for $t \in (0,\infty)$. We say that
\[
(M_i, g_i(t),\alpha_i(t), p_i) \xrightarrow{pt\ C^\infty} (M_\infty, g_\infty(t),\alpha_\infty(t), p_\infty)
\]
if there exist:

\begin{enumerate}[label=(\roman*)]
    \item An exhaustion $\{U_j\}_{j=1}^\infty$ of $M_\infty$ with $p_\infty \in U_j$;
    
    \item Open subsets $V_{i,j} \subset M_i$ with $p_i \in V_{i,j}$ and diffeomorphisms
    \[
    \phi_{i,j} : U_j \to V_{i,j}, \qquad \phi_{i,j}(p_\infty)=p_i;
    \]
    
    \item For all $j$,
    \[
    \phi_{i,j}^* g_i(t) \to g_\infty(t), \qquad \phi_{i,j}^* \alpha_i(t) \to \alpha_\infty(t)
    \]
    smoothly on $U_j \times [j^{-1},j]$.
\end{enumerate}
\end{definition}

\begin{definition}[Equivariant pointed $C^\infty$ convergence]\label{d:equivariant.pointed.cinfty.convergence}
In addition to the above, suppose that $G_i$ acts isometrically on $(M_i,g_i(1))$ and preserves $\alpha_i(1)$. For $l>0$, write
\[
G_i(l):=\{\gamma\in G_i : d_{g_i(1)}(\gamma p_i,p_i)\le l\}.
\]

We say that
\[
(M_i, g_i(1),\alpha_i(1),G_i, p_i) \xrightarrow{pt\ eq\ C^\infty} (M_\infty, g_\infty(1),\alpha_\infty(1),G_\infty, p_\infty)
\]
if:

\begin{enumerate}[label=(\roman*)]
    \item $(M_i, g_i(t),\alpha_i(t), p_i) \xrightarrow{pt\ C^\infty} (M_\infty, g_\infty(t),\alpha_\infty(t), p_\infty)$ with maps $\phi_{i,j}$ as above;

    \item $G_\infty \subset \mathrm{Isom}(M_\infty,g_\infty(1))$ is a closed subgroup preserving $\alpha_\infty(1)$;

    \item For every $j$ and all sufficiently large $i$, there exist maps
    \[
    \varphi_{i,j} : G_i(j) \to G_\infty(j+1), \qquad \psi_{i,j} : G_\infty(j) \to G_i(j+1),
    \]
    such that:
    
    \begin{enumerate}[label=(\alph*)]
        \item For every $\gamma \in G_i(j)$ and $x \in U_j$ with $\gamma(\phi_{i,j}(x)) \in V_{i,j}$,
        \[
        d_{g_\infty(1)}\big(\phi_{i,j}^{-1}(\gamma(\phi_{i,j}(x))), \varphi_{i,j}(\gamma)(x)\big) \to 0;
        \]
        
        \item For every $\eta \in G_\infty(j)$ and $x \in U_j$ with $\psi_{i,j}(\eta)(\phi_{i,j}(x)) \in V_{i,j}$,
        \[
        d_{g_\infty(1)}\big(\phi_{i,j}^{-1}(\psi_{i,j}(\eta)(\phi_{i,j}(x))), \eta(x)\big) \to 0;
        \]
    \end{enumerate}
    as $i \to \infty$, uniformly over the indicated group elements and points $x\in U_j$ satisfying the respective conditions.
\end{enumerate}

By \cite[Theorem 2-1]{fukaya1986theory}, if $(M_i, g_i(1),\alpha_i(1), p_i) \xrightarrow{pt\ eq\ C^\infty} (M_\infty, g_\infty(1),\alpha_\infty(1), p_\infty)$, then we have $(M_i/G_i, d_{g_i(1)})\to (M_\infty/G_\infty,d_{g_\infty(1)})$ in Gromov--Hausdorff topology with the following commutative diagram\begin{center}
\begin{tikzcd}
  & (M_i, g_i(1),\alpha_i(1), G_i, p_i) 
  \arrow[r, "pt\ eq\ C^\infty"] 
  \arrow[d, "\pi_{i}"]
   & (M_\infty, g_\infty(1),\alpha_\infty(1),G_\infty, p_\infty) 
   \arrow[d, "\pi_\infty"]\\
  & (M_i/G_i, d_{g_i(1)})
  \arrow[r, "GH"]
   & (M_\infty/G_\infty,d_{g_\infty(1)}).
\end{tikzcd}
\end{center}
\end{definition}

\noindent \textbf{Notation}: To simplify our notation, we write $(M,g_s(1))\xrightarrow{GH} (X,d)$ if $(M,g_s(1))$ converges to $(X,d)$ in Gromov--Hausdorff sense as $s\to\infty$.
Moreover, we write $(\tilde g_s(t),\tilde \alpha_s(t))\xrightarrow{C^\infty}(\tilde g_\infty(t),\tilde \alpha_\infty(t))$ 
to denote the pointed Cheeger-Gromov convergence on the universal cover.
If we replace $s$ by $s_i$, then the above convergences are along the sequence $s_i\to\infty$.

Two expanders on $\R^2$ model the long-time behavior of the
$\alpha$-Ricci flow on the universal cover. The first is the
Gaussian Ricci expander $(g_{\mathrm{Euc}},0)$, with vanishing
$1$-form. The second is
\begin{equation}\label{eq:flat_expander_model}
    (g^*(t),\alpha^*(t))
    =(dv^2+t\,d\theta^2,\sqrt{2}\,d\theta),
\end{equation}
where $(\theta,v)$ are coordinates on $\R^2$. Under the warped-product interpretation
(cf. Section~\ref{intro}), taking the additional fiber to be
$\R$ gives the Euclidean metric on $\R^3$ for the case $(g_{\mathrm{Euc}},0)$
and $\R\times \mathbb H^2$ in case $(g^*(t),\alpha^*(t))$,
with the hyperbolic factor having sectional curvature
$-1/(2t)$.

We write $(S^1,d([\alpha]))$ for the $S^1$ with metric $d([\alpha])$ whose total length is $\ell([\alpha])/\sqrt{2}$. The main result of this article is the following.

\begin{theorem}\label{t:main.GH_limit}
    Let $(g(t),\alpha(t))$ be a solution to the $\alpha$-Ricci flow on $\mathbb{T}^2$ defined for $[0,\infty)$, and suppose that $[\alpha]\neq 0$. Then we have the following.
    \begin{enumerate}[label=(\roman*)]
        \item If $[\alpha]$ is rational, then $(M,t^{-1}g(t))\xrightarrow{GH}(S^1,d([\alpha]))$ as $t\to\infty$. On $\tilde M$, we have $(\tilde g_s(t),\tilde \alpha_s(t))\xrightarrow{C^\infty}(g^*(t),\alpha^*(t))$ as $s\to\infty$, where $(g^*(t),\alpha^*(t))$ is as in \eqref{eq:flat_expander_model}. Moreover, we have the following equivariant convergence
        \begin{center}
\begin{tikzcd}
  & (\tilde{M},\tilde{g}_{s}(1), \tilde\alpha_{s}(1), G_{s}, p) 
  \arrow[r, "pt\ eq\ C^\infty"] 
  \arrow[d, "\pi"]
       & (\R^2,g^*(1), \alpha^*(1), G_\infty, p_\infty) 
   \arrow[d, "\pi_\infty"]\\
  & (M, g_{s}(1)) 
  \arrow[r, "GH"]
   & (S^1,d([\alpha]))
\end{tikzcd}
\end{center}
        where
        $G_s=\pi_1(M)\subset \operatorname{Isom}(\tilde g_s(1))$ and $G_\infty\cong \R\times\Z\subset \operatorname{Isom}( g^*(1))$ acts on $\R^2$ by translations generated by $\partial_v$ and $(\theta,v)\mapsto (\theta+\ell([\alpha])/\sqrt{2},v)$.
         \item If $[\alpha]$ is irrational, then there is a sequence $s_i\to\infty$ such that 
         $(M,s_i^{-1}g(s_i))\xrightarrow{GH}~\{\mathrm{pt}\}$ and 
         on the universal cover, $(\tilde g_{s_i}(t),\tilde \alpha_{s_i}(t))\xrightarrow{C^\infty}(g^*(t),\alpha^*(t))$.  Moreover, we have the following equivariant convergence
        \begin{center}
\begin{tikzcd}
  & (\tilde{M},\tilde{g}_{s_i}(1), \tilde\alpha_{s_i}(1), G_{i}, p_i) 
  \arrow[r, "pt\ eq\ C^\infty"] 
  \arrow[d, "\pi"]
   & (\R^2,g^*(1), \alpha^*(1), G_\infty, p_\infty) 
   \arrow[d, "\pi_\infty"]\\
  & (M, g_{s_i}(1)) 
  \arrow[r, "GH"]
   & \{\mathrm{pt}\}
\end{tikzcd}
\end{center}
        where $G_i=\pi_1(M)\subset \operatorname{Isom}(\tilde g_i(1))$ and $G_\infty\cong \R\times\R\subset \operatorname{Isom}( g^*(1))$ is the translation subgroup of the Euclidean isometry group.
     \end{enumerate}
\end{theorem}
\begin{remark}\label{r:lott.sesum}
    Lott and Sesum’s result \cite{lott2014ricci} treats the case $[\alpha]=0$, where $(M,g(t))$ converges to a flat $2$-torus and hence $(M,t^{-1}g(t))$ collapses to a point and $(\tilde g_s(t),\tilde \alpha_s(t))\xrightarrow{C^\infty}(g_{\mathrm{Euc}},0)$ as $s\to\infty$. In our case where $[\alpha]\neq 0$, the limit on the universal cover is $(g^*(t),\alpha^*(t))$, which is different from $(g_{\mathrm{Euc}},0)$.
\end{remark}
\begin{remark}
    We note that in case (i), the conclusion implies $\alpha^*(\partial_v)=0$ since $\alpha^*=\sqrt 2 d\theta$, which means $\alpha^*$ has no component in the fiber direction of the action $G_\infty$. This fact plays an important role in showing the full sequence convergence both in Gromov--Hausdorff sense and on the universal cover (cf. Theorem~\ref{no_fiber_component}).
\end{remark}
\begin{remark}
    We want to compare our result with the result of Lott \cite{lott2010dimensional,lott2026under} where he showed that compact 3-dimensional Type III Ricci flow has similar collapsing behavior determined by the Thurston type of the manifold. Our result can be viewed as a generalization to noncompact manifolds with symmetry determined by $[\alpha]$ in view of the warped product interpretation in Section~\ref{intro}. In our case, the collapsing behavior also depends on the topology, namely the cohomology class $[\alpha]$.
\end{remark}
\begin{remark}
Theorem~\ref{t:main.GH_limit} exhibits a new type of metric behavior in a geometric flow, despite the many formal similarities between the Ricci flow and \eqref{alpha_Ricci_flow}. In particular, the long-time behavior depends on the arithmetic property of the initial data. As $[\alpha]$ varies, the long-time behavior changes dramatically. We expect that analogous phenomena occur for higher-dimensional tori as well.
\end{remark}
\begin{figure}[!ht]
    \centering
    \resizebox{1\textwidth}{!}{\begin{circuitikz}
\tikzstyle{every node}=[font=\LARGE]
\node [font=\Huge] at (12.5,0.75) {};

\begin{scope}[shift={(-95,-15)}]
\draw [ line width=1pt ] (95.25,20) ellipse (8.25cm and 4.75cm);
\draw [line width=1pt, short] (92,19.75) .. controls (94.75,18.25) and (96.75,18.5) .. (99.25,19.75);
\draw [line width=1pt, short] (93.75,19) .. controls (95.25,20) and (95.75,19.75) .. (97.25,19);
\draw [ color={rgb,255:red,4; green,51; blue,255}, line width=1pt, dashed] (97,24.5) .. controls (94.25,24) and (94,20.25) .. (95.25,19.75);
\draw [ color={rgb,255:red,4; green,51; blue,255}, line width=1pt, short] (95.25,19.75) .. controls (96,21.5) and (96,23.5) .. (93,24.5);
\draw [ color={rgb,255:red,4; green,51; blue,255}, line width=1pt, dashed] (93,24.5) .. controls (91.25,22.5) and (92.25,20.25) .. (93.75,19);
\draw [ color={rgb,255:red,4; green,51; blue,255}, line width=1pt, short] (93.75,19) .. controls (92.75,21.5) and (90.75,22.75) .. (87.75,22);
\draw [ color={rgb,255:red,4; green,51; blue,255}, line width=1pt, dashed] (87.75,22) .. controls (89.5,17.25) and (92.25,17.75) .. (94.25,18.75);
\draw [ color={rgb,255:red,4; green,51; blue,255}, line width=1pt, short] (94.25,18.75) .. controls (92.25,19) and (91.5,17) .. (92.5,15.5);
\draw [ color={rgb,255:red,4; green,51; blue,255}, line width=1pt, dashed] (92.5,15.5) .. controls (97,15.75) and (96.25,16.75) .. (96.75,18.75);
\draw [ color={rgb,255:red,4; green,51; blue,255}, line width=1pt, short] (96.75,18.75) .. controls (96,16.75) and (97,15.75) .. (99,15.75);
\draw [ color={rgb,255:red,4; green,51; blue,255}, line width=1pt, dashed] (99,15.75) .. controls (100.75,17) and (100.25,19) .. (99,19.5);
\draw [ color={rgb,255:red,4; green,51; blue,255}, line width=1pt, short] (99,19.5) .. controls (101.25,19.5) and (102.5,20) .. (102.5,22.25);
\draw [ color={rgb,255:red,4; green,51; blue,255}, line width=1pt, dashed] (102.5,22.25) .. controls (98.75,22.75) and (97.25,22.25) .. (96.5,19.5);
\draw [ color={rgb,255:red,4; green,51; blue,255}, line width=1pt, short] (97,24.5) .. controls (98.5,23.5) and (99,21.25) .. (96.5,19.5);
\draw [->, >=Stealth] (104.5,20) -- (114.5,20)node[pos=0.5, fill=white]{Gromov-Hausdorff as t $\rightarrow$ $\infty$};
\draw [ color={rgb,255:red,4; green,51; blue,255} , line width=1pt ] (118.25,20) ellipse (2.75cm and 2.75cm);
\draw [ line width=1pt ] (95.5,8.25) ellipse (9cm and 5.75cm);
\draw [line width=1pt, short] (91.75,8) .. controls (94.5,6.75) and (96.5,6.75) .. (99.5,8);
\draw [line width=1pt, short] (93.5,7.5) .. controls (95.25,8.25) and (95.75,8.25) .. (97.25,7.5);
\draw [ color={rgb,255:red,4; green,51; blue,255}, line width=1pt, short] (93.5,13.75) .. controls (95.5,11.5) and (95.25,10.25) .. (95,8);
\draw [ color={rgb,255:red,4; green,51; blue,255}, line width=1pt, short] (94,13.75) .. controls (96,11.5) and (95.75,10) .. (95.5,8);
\draw [ color={rgb,255:red,4; green,51; blue,255}, line width=1pt, short] (94.5,14) .. controls (96.5,11.5) and (96.25,9.75) .. (96,8);
\draw [ color={rgb,255:red,4; green,51; blue,255}, line width=1pt, short] (95,14) .. controls (97,11.5) and (96.75,10) .. (96.5,7.75);
\draw [ color={rgb,255:red,4; green,51; blue,255}, line width=1pt, short] (93,13.75) .. controls (94.5,12) and (95,10.5) .. (94.5,8);
\draw [ color={rgb,255:red,4; green,51; blue,255}, line width=1pt, short] (93.75,7.25) .. controls (93,5.5) and (93,4.5) .. (93.5,2.75);
\draw [ color={rgb,255:red,4; green,51; blue,255}, line width=1pt, short] (94.5,7) .. controls (94,6) and (93.5,4.25) .. (94.25,2.5);
\draw [ color={rgb,255:red,4; green,51; blue,255}, line width=1pt, short] (95,7) .. controls (94.5,5.25) and (94.25,4.25) .. (94.75,2.5);
\draw [ color={rgb,255:red,4; green,51; blue,255}, line width=1pt, short] (95.5,7) .. controls (95,6.25) and (94.5,4) .. (95.25,2.5);
\draw [ color={rgb,255:red,4; green,51; blue,255}, line width=1pt, short] (95.75,7) .. controls (95.25,5.75) and (95.25,3.75) .. (95.75,2.5);
\draw [ color={rgb,255:red,4; green,51; blue,255}, line width=1pt, short] (96.25,7) .. controls (95.25,5) and (96,3.75) .. (96.25,2.5);
\draw [ color={rgb,255:red,4; green,51; blue,255}, line width=1pt, dashed] (93,13.75) .. controls (93.25,10.5) and (94,9.75) .. (95,8);
\draw [ color={rgb,255:red,4; green,51; blue,255}, line width=1pt, dashed] (93.5,14) .. controls (93.75,10.5) and (94.25,10.5) .. (95.5,8);
\draw [ color={rgb,255:red,4; green,51; blue,255}, line width=1pt, dashed] (94,14) .. controls (94.5,10.75) and (94.75,10.75) .. (96,8);
\draw [ color={rgb,255:red,4; green,51; blue,255}, line width=1pt, dashed] (94.5,14) .. controls (95,10.75) and (95.25,10.75) .. (96.5,8);
\node [font=\LARGE, color={rgb,255:red,4; green,51; blue,255}] at (89,8.25) {..........};
\node [font=\LARGE, color={rgb,255:red,4; green,51; blue,255}] at (101.5,8.25) {.........};
\draw [ color={rgb,255:red,4; green,51; blue,255}, line width=1pt, dashed] (94.5,7) .. controls (94.25,5) and (94.5,4.75) .. (93.5,2.75);
\draw [ color={rgb,255:red,4; green,51; blue,255}, line width=1pt, dashed] (95,7) .. controls (95,4.75) and (95,4.75) .. (94.25,2.75);
\draw [ color={rgb,255:red,4; green,51; blue,255}, line width=1pt, dashed] (95.5,7) .. controls (95.5,4.75) and (95.75,5) .. (94.75,2.75);
\draw [ color={rgb,255:red,4; green,51; blue,255}, line width=1pt, dashed] (95.75,7) .. controls (96,4.75) and (95.75,4.75) .. (95.25,2.75);
\draw [ color={rgb,255:red,4; green,51; blue,255}, line width=1pt, dashed] (96.25,7) .. controls (96.75,4.75) and (96.25,4.25) .. (95.75,2.5);
\draw [->, >=Stealth] (105,8) -- (115.25,8)node[pos=0.5, fill=white]{Gromov-Hausdorff as t $\rightarrow$ $\infty$};
\draw [ color={rgb,255:red,4; green,51; blue,255}, line width=2pt, short] (119.25,7.75) -- (119.25,7.75);
\draw [ color={rgb,255:red,4; green,51; blue,255} , fill={rgb,255:red,4; green,51; blue,255}] (118.25,7.75) circle (0.25cm);
\node [font=\Huge] at (104.5,24.5) {$[\alpha]$ is rational};
\node [font=\Huge] at (106,13.25) {$[\alpha]$ is irrational};
\end{scope}
\end{circuitikz}}
    \caption{Different collapsing behavior of $\alpha$-Ricci flow}
    \label{f:GH_convergence}
    \label{fig:GH limit}
\end{figure}
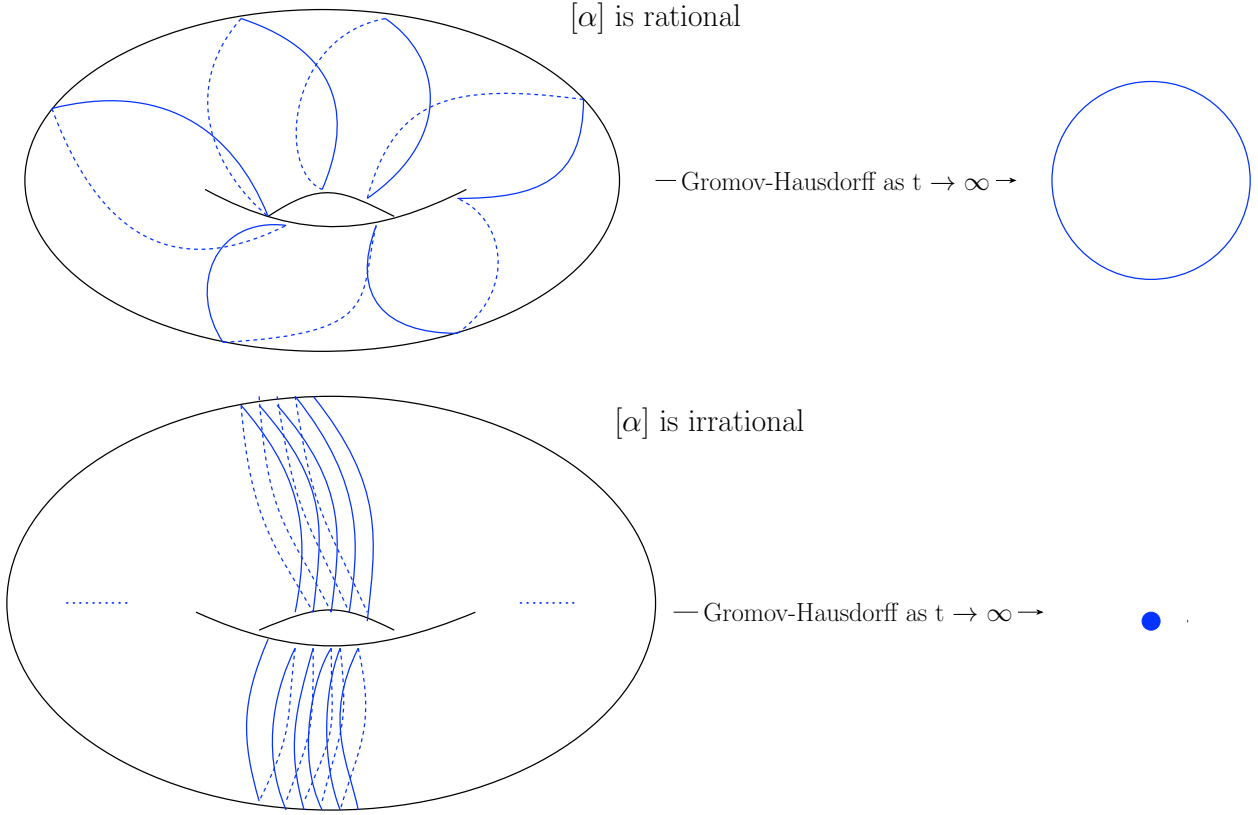

\begin{remark}
    The metric component of the $\alpha$-Ricci flow \eqref{alpha_Ricci_flow} is not conformal. When the initial metric is flat and the $1$-form is parallel, one can show that the conformal class $[g(t)] \in \mathbb{H}^2 / \operatorname{PSL}_2(\mathbb{Z})$ converges to $\left[\left(\frac{a}{b}, 0\right)\right] \in \partial \mathbb{H}^2 / \operatorname{PSL}_2(\mathbb{Z})$ where $a,b$ are the periods of $\alpha$ on a set of generators of $\pi_1(M)$ (cf.\ the proof of Theorem~\ref{model.case}).
\end{remark}

\begin{remark}
    In the irrational case, our method only gives a subsequence $s_i\to\infty$ such that the desired convergences hold. It would be desirable to show that the convergences are along the full sequence.
\end{remark}
Figure~\ref{f:GH_convergence} illustrates the Gromov--Hausdorff convergence in the model case of Theorem~\ref{model.case}, where the initial metric is flat and $\alpha$ is parallel. When $[\alpha]$ is rational,  the blue circle inside the torus represents a loop whose length converges to $\ell([\alpha])/\sqrt{2}$, while the metric in the transverse direction collapses. When $[\alpha]$ is irrational, the integral curves of $\ker\alpha$ are dense in $M$, and the distance along these curves shrinks under the blow-down. Consequently, the entire manifold converges to a point.

Finally, we have a similar result for $\alpha$-Ricci flow on the Klein bottle. This case is easier since all cohomology classes are rational.

\begin{definition}Let $K$ be the Klein bottle.  
Let $[\alpha]\in H^1(K,\R)\cong\R$. If $[\alpha]\neq 0$, we define $\ell([\alpha])$ to be the unique $\ell>0$ such that $[\alpha]=\ell v$, where $v\in H^1(K,\Z)$ is a generator of $H^1(K,\Z)\cong \Z$. If $[\alpha]=0$, we define $\ell([\alpha])=0$.
\end{definition}
\begin{corollary}\label{c:klein.bottle}
    Let $(g(t),\alpha(t))$ be a solution to the $\alpha$-Ricci flow on the Klein bottle $K$. Then the following hold.
    \begin{enumerate}[label=(\roman*)]
    \item If $[\alpha]\neq 0$, then $(K,t^{-1}g(t))\xrightarrow{GH} ~(S^1,d([\alpha]))$ as $t\to\infty$. On the universal cover $\tilde K$, we have $(\tilde g_s(t),\tilde \alpha_s(t))\xrightarrow{C^\infty}(g^*(t),\alpha^*(t))$ as $s\to\infty$. Moreover, we have the following equivariant convergence
        \begin{center}
\begin{tikzcd}
  & (\tilde{K},\tilde{g}_{s}(1), \tilde\alpha_{s}(1), G_{s}, p) 
  \arrow[r, "pt\ eq\ C^\infty"] 
  \arrow[d, "\pi"]
       & (\R^2,g^*(1), \alpha^*(1), G_\infty, p_\infty) 
   \arrow[d, "\pi_\infty"]\\
  & (K, g_{s}(1)) 
  \arrow[r, "GH"]
   & (S^1,d([\alpha]))
\end{tikzcd}
\end{center}
        where
        $G_s=\pi_1(K)\subset \operatorname{Isom}(\tilde g_s(1))$ and $G_\infty\cong \R\rtimes\Z\subset \operatorname{Isom}( g^*(1))$ is generated by translations in $\partial_v$ direction and the glide reflection $(\theta,v)\mapsto (\theta+\ell([\alpha])/\sqrt{2},-v)$.
    \item If $[\alpha]=0$, then
    \[
        (K,t^{-1}g(t))\xrightarrow{GH}\{\mathrm{pt}\}
        \quad\text{as }t\to\infty,
    \]
    and, on the universal cover,
    \[
        (\tilde g_s(t),\tilde\alpha_s(t))
        \xrightarrow{C^\infty}(g_{\mathrm{Euc}},0)
        \quad\text{as }s\to\infty.
    \]
    \end{enumerate}
\end{corollary}
\begin{remark}
    As in Remark~\ref{r:lott.sesum}, the case $[\alpha]=0$ follows directly from \cite{lott2014ricci}.
\end{remark}
\subsection{Sketch of the proof}

The proof of Theorem~\ref{t:main.GH_limit} follows the strategy of \cite{lott2010dimensional}. We first establish a compactness theorem for $\alpha$-Ricci flows in the Gromov--Hausdorff sense and on the universal cover using collapsing theory \cite{fukaya1988boundary,cheeger1992nilpotent} (Corollary~\ref{c:convergence_CFG}). While \cite{lott2010dimensional} uses groupoid language, we work directly on the universal cover taking advantage of the simple collapsing behavior of $\mathbb{T}^2$. 

We then study a version of the $\mathcal{W}_+$ functional for locally $\mathbb{R}^N$-invariant $\alpha$-Ricci flows over a compact base. Its monotonicity is conditional: it holds when $\alpha$ has no fiber component (Corollary~\ref{c:conditional.monotonicity}). Using the $\mathcal{W}_+$ functional, we classify locally $\mathbb{R}^N$-invariant $\alpha$-expanders on the trivial bundle $S^1\times\mathbb R^N$ with trivial linear holonomy and no fiber component (Corollary~\ref{1d.alpha.expander}) and show that these are the only blow-down limits (Proposition~\ref{p:base.S^1.implies.universal.cover.expander}). For invariant flows over $S^1$, this classification requires no convergence of the flat connection $A$; see Corollary~\ref{1d.alpha.expander} and Remark~\ref{r:connection.convergence}. The no-fiber condition when $[\alpha]$ is rational is established in Theorem~\ref{no_fiber_component} by combining collapsing theory with the systole growth bound from Theorem~\ref{t:2d_estimates}.

In relating the Gromov--Hausdorff limit to the rationality of $[\alpha]$, a key observation is that the rationality condition for $[\alpha]$ can be naturally recast in terms of the action of $\pi_1(M)$ on the level sets of $u$, where $du = \alpha$ on $\tilde{M}$ (Lemma~\ref{action.on.level.sets}). Remarkably, this action is invariant under the flow and persists to the limit under equivariant convergence on the universal cover (Theorem~\ref{t:GH_limit_vs_rationality}). With these ingredients in place, Theorem~\ref{t:main.GH_limit} follows.

\subsection{Organization of the paper}

The paper is organized as follows. In Section~\ref{s:preliminaries}, we discuss preliminaries concerning generators of $H_1(\mathbb{T}^2, \mathbb{Z})$ and the rationality of $[\alpha]$. In Section~\ref{s:model.solutions}, we study expanders and a model case where $g(0)$ is flat and $\alpha(0)$ is parallel. In Section~\ref{s:geometric quantity}, we prove Theorem~\ref{t:2d_estimates}. In Section~\ref{s:collapsing torus}, we apply Cheeger--Fukaya--Gromov theory to $\alpha$-Ricci flow to establish the compactness of $\alpha$-Ricci flows on the torus. In Section~\ref{s:condition monotonicity}, we establish the conditional monotonicity of the $\mathcal{W}_+$ functional. In Section~\ref{s:no.fiber}, we show that any blow-down limit has no fiber component if $[\alpha]$ is rational. In Section~\ref{s:GHlimit_vs_rationality}, we relate the Gromov--Hausdorff limits to the rationality of the cohomology class $[\alpha]$. In Section~\ref{s:proof.of.main.theorem.(i)}, we prove Theorem~\ref{t:main.GH_limit}(i) and Corollary~\ref{c:klein.bottle}. In Section~\ref{s:proof.of.main.theorem.(ii)}, we prove Theorem~\ref{t:main.GH_limit}(ii). Finally, in Appendix~\ref{s:dynamic stability}, we prove the stability of locally $\mathbb{R}^N$-invariant $\alpha$-Ricci flows over $S^1$ near the expander $(g^*(t), \alpha^*(t))$.
\subsection{Acknowledgement}The authors thank Jeff Streets for introducing them to this topic, for sharing notes arising from his discussions with Richard Bamler, and for many enlightening discussions. They are also grateful to Dan Knopf for generously explaining his work on the stability of invariant Ricci flows. The authors were deeply saddened to learn of his passing in 2025 and remember with gratitude his insight and kindness. Further thanks go to Kuan-Hui Lee and Jinxin Xue for valuable discussions, and to Anna Cusenza and Michael Womack for helpful conversations.

This material is based upon work supported by the National Science
Foundation under Grant No. DMS-1928930, while the first named author was in
residence at the Simons Laufer Mathematical Sciences Institute (formerly
MSRI) in Berkeley, California, during the Fall semester of 2024.

\subsection{AI disclosure}
The mathematical ideas and arguments were developed by the authors. ChatGPT (GPT-6 Astra) was used to polish the language and assist in checking the exposition and arguments. The authors reviewed all AI suggestions and take full responsibility for the manuscript.

\section{Preliminaries}\label{s:preliminaries}
\subsection{Generators of \texorpdfstring{$H_1(\mathbb{T}^2,\Z)$}{H1(T2,Z)} and an equivalent definition of Definition~\ref{d:rationality}}

We would like to give an equivalent definition of the rationality of $[\alpha]$ using a basis of $\pi_1(\mathbb{T}^2)$, which is easier to work with.

Let $\alpha$ be a closed 1-form on $\mathbb{T}^2$ such that $0\neq [\alpha]\in H^1(\mathbb{T}^2,\mathbb{R})$. We choose a set of generators $[\gamma_1],[\gamma_2]$ of $\pi_1(M)\cong\Z\times\Z$ and define $a=\int _{\gamma_1}\alpha$, $b=\int_{\gamma_2}\alpha$. 
 Since $\alpha$ in \eqref{alpha_Ricci_flow} evolves by an exact form, both $a$ and $b$ are invariant under the flow. 
We use the convention that for $a,b\in \R$, $a/b=0$ if $a=0$ and $b\neq0$ and $a/b=\infty$ if $a\neq0$ and $b=0$. Note that $a,b$ cannot be zero simultaneously. If $a/b\in \Q\cup\{\infty\}$, we define 
\begin{equation}\label{ell.a.b}
    \ell_{a,b}=\left\{
    \begin{aligned}
        \frac{|b|}{|q|}\quad&\text{ if }a,b\neq 0,\ \frac{a}{b}=\frac{p}{q} \text{ where }p,q\in \Z, \gcd(p,q)=1,\\
        |b|\quad&\text{ if } a=0,\\
        |a|\quad &\text{ if } b=0.
    \end{aligned}\right.
\end{equation}
Therefore, we have $a=\ell_{a,b}\operatorname{sign}(a)|p|$, $b=\ell_{a,b}\operatorname{sign}(b)|q|$, whenever $a,b\neq0$.
If $a/b\in \R\setminus\Q$, we define $\ell_{a,b}=0$.
\begin{proposition}\label{p:equivalent definition}
    Define $\alpha,[\gamma_1],[\gamma_2],a,b$ as above. We have the following.
    \begin{enumerate}[label=(\roman*)]
        \item $[\alpha]$ is rational  if and only if $a/b\in\mathbb{Q}\cup\{\infty\}$ and $[\alpha]$ is irrational if and only if $a/b\in \R\setminus\Q$;
        \item $\ell_{a,b} = \ell([\alpha])$, where $\ell([\alpha])$ is defined in Definition~\ref{d:prmitive length}.
    \end{enumerate}
    In particular, for a given $[\alpha]$, the property that $a/b\in\mathbb{Q}\cup\{\infty\}$ as well as the value of $\ell_{a,b}$ are independent of the choice of $[\gamma_1],[\gamma_2]\in \pi_1(\mathbb{T}^2)$.
\end{proposition}

\begin{proof}
   
    To prove (i), let $[\gamma_1]^*,[\gamma_2]^*\in H^1(M,\mathbb{Z})$ be the dual basis of $[\gamma_1],[\gamma_2]\in H_1(M,\Z)$. 
    Suppose $[\alpha]$ is rational. Then there exists $c>0$ such that:
    \begin{align*}
        c[\alpha] = m[\gamma_1]^* +n[\gamma_2]^*,
    \end{align*}
    for some integers $m,n\in\mathbb{Z}$. Thus, we have $a = m/c$ and $b = n/c$. Hence $a/b = m/n\in \mathbb{Q}\cup\{\infty\}$.  
    
    Now suppose $a/b\in\mathbb{Q}\cup\{\infty\}$. If $a,b\neq0$, then there exist $p,q\in\mathbb{Z}\setminus\{0\}$ such that $\frac{a}{b} = \frac{p}{q}$, with $\operatorname{gcd}(p,q) = 1$ and $p/a = q/b = c$, for some constant $c$. Thus, $c[\alpha] = p[\gamma_1]^*+q[\gamma_2]^*$, which shows that $[\alpha]$ is rational. If $a = 0$ or $b = 0$, then we can choose $c = b^{-1}$ or $a^{-1}$ to see that $[\alpha]$ is rational.
    
    To prove (ii), by definition of $a,b,[\gamma_1]^*,[\gamma_2]^*$, we have $[\alpha] = a[\gamma_1]^* +b[\gamma_2]^*$. If $a/b\in \R\setminus\Q$, then by definition $\ell([\alpha])=\ell_{a,b}=0$. If $a/b\in\mathbb{Q}\cup\{\infty\}$ and $a,b\neq0$, there exist $p,q\in\mathbb{Z}\setminus\{0\}$ such that $\frac{a}{b} = \frac{p}{q}$ with $\operatorname{gcd}(p,q) = 1$. We have \[(\ell_{a,b})^{-1}[\alpha]=\operatorname{sign}(a)|p|[\gamma_1]^*+\operatorname{sign}(b)|q|[\gamma_2]^*\in H^1(M,\Z)\] and\begin{align*}
        \R\cdot[\alpha]\cap H^1(M,\Z)=&\{c[\alpha]:ca\in \Z,cb\in \Z\}\\
        =&\{c'(\ell_{a,b})^{-1}[\alpha]:c'\operatorname{sign}(a)|p|\in \Z,c'\operatorname{sign}(b)|q|\in \Z \}\\
        =&\{c'(\ell_{a,b})^{-1}[\alpha]:c'\in\Z\},
    \end{align*} where we substitute $\ell_{a,b}c=c'$ in the second equality and use $\gcd(p,q)=1$ in the third equality. Thus $\R\cdot[\alpha]\cap H^1(M,\Z)=\Z\cdot [(\ell_{a,b})^{-1}\alpha]$ and $[\alpha]=\ell_{a,b}[(\ell_{a,b})^{-1}\alpha]$. If $a = 0$ or $b = 0$, then $\R\cdot[\alpha]\cap H^1(M,\Z)=\Z\cdot[\gamma_2]^*$ or $\Z\cdot[\gamma_1]^*$ and $[\alpha]= \ell_{a,b}(\operatorname{sign}(b) [\gamma_2]^*)$ or $ \ell_{a,b}(\operatorname{sign}(a) [\gamma_1]^*)$. In any case, $\ell([\alpha])=\ell_{a,b}$. 
\end{proof}

Now, we prove a lemma that distinguishes the rationality of the class $[\alpha]$ based on the action of $\pi_1(\mathbb T^2)$ on $\R^2$, which plays a crucial role in analyzing the Gromov--Hausdorff limit.
\begin{lemma}\label{action.on.level.sets}
    Let $(M,g)$ be a Riemannian 2-torus, and $[\gamma_1],[\gamma_2]$ be a set of generators of $\pi_1(M)$. We write $\int_{\gamma_1}\alpha=a$, $\int_{\gamma_2}\alpha=b$. We view $[\gamma_1],[\gamma_2]$ as deck transformations which act isometrically on the universal cover $(\tilde M,\tilde g)$. On the universal cover, the lift of $\alpha$ can be written as $\tilde \alpha=du$ for some function $u$. Then for any $x\in \tilde M$  we have \[u([\gamma_1] (x))-u(x)=a\text{ and }u([\gamma_2] (x))-u(x)=b.\] 
    This is saying $\pi_1(M)$ acts on the set $\{u^{-1}(r):r\in \R\}$ of level sets of $u$ which is generated by $u^{-1}(r)\mapsto u^{-1}(r+a)$ and $u^{-1}(r)\mapsto u^{-1}(r+b)$. In particular,
    \begin{itemize}
        \item If $[\alpha]$ is rational, then the action of $\pi_1(M)$ on $\{u^{-1}(r):r\in \R\}$ is equivalent to the discrete $\Z$ action on $\{u^{-1}(r):r\in \R\}$ generated by $u^{-1}(r)\mapsto u^{-1}(r+\ell([\alpha]))$ for any $r\in \R$. In other words, $\{[\alpha](\sigma):\sigma\in \pi_1(M)\}=\Z \cdot \ell([\alpha])$;
        \item If $[\alpha]$ is irrational, then the action of $\pi_1(M)$ on $\{u^{-1}(r):r\in \R\}$ has dense orbits.
    \end{itemize} 
\end{lemma}
\begin{proof}
    Let $\tilde\gamma_1$ be a lift of $\gamma_1$ to $\tilde M$ such that $\tilde\gamma_1(0)=x$ and $\tilde\gamma_1(1)=[\gamma_1](x)$. Then \[u([\gamma_1](x))-u(x)=\int _{\tilde\gamma_1} du(\tilde\gamma_1')=\int_{\gamma_1}\alpha=a.\] In particular, $[\gamma_1]$ maps $u^{-1}(r)$ to $u^{-1}(r+a)$ for any $r\in\R$. The case for $[\gamma_2]$ is similar. 

    If $[\alpha]$ is rational, by an algebraic fact, the action of $\pi_1(M)$ on $\{u^{-1}(r):r\in \R\}$ is a $\Z$-action generated by $u^{-1}(r)\mapsto u^{-1}(r+\ell_{a,b})$ and $\ell_{a,b}=\ell([\alpha])$ by Proposition~\ref{p:equivalent definition}. 
    
    If $[\alpha]$ is irrational, by an algebraic fact, then the action has dense orbits.
\end{proof}

\section{Model solutions}\label{s:model.solutions}
In this section, we examine the behavior of model solutions. These solutions have flat metrics and constant 1-forms which will be asymptotic models for long-time behavior.  We also prove the main theorem in the special case when the initial data is flat and $\alpha$ is constant.
\subsection{Example \texorpdfstring{$\alpha$}{alpha}-expander solutions}
We first define  the \textit{$\alpha$-Ricci expander}, which is a self-similar expanding solution of the $\alpha$-Ricci flow.

\begin{definition}
    We say that $(M,g, \alpha,V)$ is an \textit{$\alpha$-Ricci expander}, if there exists a smooth vector field $V$ such that:
    \begin{align*}
        \operatorname{Ric} - \frac{1}{4}\alpha\otimes\alpha + \frac{1}{2}\mathcal{L}_V g &= -\frac{1}{2}g,\\
        d_g^*\alpha + i_{V}\alpha &= 0.
    \end{align*}
     If $V = \nabla f$ for some $f\in C^\infty(M)$, then we call $(M,g, \alpha,f)$ a \textit{gradient $\alpha$-Ricci expander}.
\end{definition}

\begin{example}Any Ricci expander is also an $\alpha$-Ricci expander. 
An example relevant to this paper is given by 
$
(M, g, \alpha, f) = (\mathbb{R}^N, g_{\mathrm{Euc}}, 0, -\tfrac{1}{4}|v|^2),$
where $v$ denotes the standard coordinates on $\mathbb{R}^N$.
\end{example}
\begin{example}[cf. \eqref{eq:flat_expander_model}]\label{example.noncompact invariant expander}
Consider $(M,g,\alpha,f) = (\mathbb{R}^N\times\R, g_{\mathbb{R}^N}+ d\theta^2, \sqrt{2}d\theta,-\frac{1}{4}|v|^2)$ where $v$ denotes the standard coordinates on $\R^N$ and $\theta$ denotes the standard coordinate on $\R$. One can check that this is a gradient $\alpha$-Ricci expander. We write $\phi_t(\theta,v) = (\theta,t^{-\frac{1}{2}}v)$ for the one parameter family of diffeomorphisms generated by $\frac{1}{t}\nabla f$. Then \begin{align*}
    g^*(t) &=t\phi_t^*g= g_{\mathbb{R}^N} + td\theta^2\\
    \alpha^*(t) &=\phi_t^*\alpha= \sqrt{2}d\theta\equiv \alpha^*
\end{align*} is an $\alpha$-Ricci flow on $(0,\infty)$.
Since $\alpha^*(t)$ is spatially constant and $g^*(t)$ is flat, one can view 
$(g^*(t), \alpha^*(t))$ as an $\alpha$-Ricci flow on any quotient of 
$\mathbb{R}^{N+1}$, for example, on $S^1\times\mathbb{R}^N$.
The case relevant to this paper is $N=1$.
\end{example}
\begin{remark}
When converted back to the usual Ricci flow by taking a warped product with $\R$ (cf. section \ref{intro}), the expander $(g^*(t), \alpha^*(t))$ is isometric to $\R^{N}\times \mathbb{H}^2$.
\end{remark}
\begin{remark}
    The behavior of $(g^*(t),\alpha^*(t))$ is that the base metric on $S^1$ expands while the fiber metric on $\R^N$ remains fixed. This resembles the expander on the Sol-geometry in the 3D-Ricci flow \cite[Section 3.3.2]{lott2007long}. This also indicates that with the presence of $\alpha$, the long-time behavior of the $\alpha$-Ricci flow is determined not only by the curvature of $g$, which is flat for all $t$, but also by $\alpha$.
\end{remark}
The following lemma shows that the $\alpha$-Ricci flow on $\mathbb{R}^2$ with $\mathbb{R}^2$-symmetry must be an $\alpha$-expander.
\begin{lemma}\label{l:R2.invariant.solution}
    Let $(\tilde g(t), \tilde \alpha(t))$ be a solution to the $\alpha$-Ricci flow on $\mathbb{R}^2 \times (0,\infty)$ 
which is invariant under translations of $\mathbb{R}^2$. Then either $(\tilde g(t), \tilde \alpha(t))$ 
is isometric to $(g^*(t+c), \alpha^*(t+c))$ for some $c \ge 0$, or it is the static solution 
$(g_{\mathrm{Euc}}, 0)$.
\end{lemma}
\begin{proof}
    Since $(\tilde g(t), \tilde \alpha(t))$ is translation invariant, we have $\tilde g(t)$ is flat for $t\in (0,\infty)$ and $\tilde\alpha(t)$ is parallel with respect to $\tilde g(t)$. By \eqref{alpha_Ricci_flow}, $\tilde \alpha(t)\equiv\tilde\alpha$ is constant in time. If $\tilde\alpha=0$, then integrating \eqref{alpha_Ricci_flow} shows that $\tilde g(t)$ is isometric to $g_{\mathrm{Euc}}$. If $\tilde\alpha\neq0$, we choose coordinates $(\theta,v)$ such that $\tilde g(1)(\partial_v,\partial_\theta)=0$, $\tilde\alpha=\sqrt{2}d\theta$ and $\tilde g(1)=dv^2+\frac{c_0}{2}\tilde\alpha\otimes\tilde\alpha$ for some $c_0>0$. Integrating \eqref{alpha_Ricci_flow}, we get $\tilde  g(t)=dv^2+\frac{t-1+c_0}{2}\tilde\alpha\otimes \tilde \alpha$. Since $\tilde g(t)$ is a metric for $t\in (0,\infty)$, we have $c=c_0-1\ge0$. Thus $\tilde  g(t)=dv^2+\frac{t+c}{2}\tilde\alpha\otimes \tilde \alpha=g^*(t+c)$ and $\tilde \alpha(t)=\alpha^*(t+c)$.
\end{proof}
With this classification in hand, we now show that the blow-down limit of any $\R^2$-invariant solution with nonzero $\alpha$ must converge to the  $\alpha$-expander in Example \ref{example.noncompact invariant expander}.
\begin{proposition}\label{p:canonical_alpha_expander_attractor}
     Let $(\tilde g(t), \tilde \alpha(t))$ be a solution to the $\alpha$-Ricci flow on $\mathbb{R}^2 \times (0,\infty)$ 
that is invariant under translations of $\mathbb{R}^2$ and $\tilde{\alpha}$ is nonzero. Then $(\tilde{g}_s(t),\tilde{\alpha}_s(t))\xrightarrow{C^\infty}(g^*(t),\alpha^*(t))$.
\end{proposition}
\begin{proof}
By Lemma~\ref{l:R2.invariant.solution}, there is $c\ge 0$ such that in some coordinates $(\theta,v)$ of $\R^2$ we have \[(\tilde{g}(t),\tilde{\alpha}(t))=(g^*(t+c),\alpha^*(t+c))=(dv^2+\frac{t+c}{2}\alpha^*\otimes \alpha^*,\alpha^*=\sqrt 2d\theta).\] Consider the blow-down solution \[\tilde{g}_s(t) = \frac{1}{s}\tilde{g}(st)=\frac{1}{s}dv^2 + \frac{(st+c)}{2s}\alpha^*\otimes\alpha^*,\quad\tilde{\alpha}_s(t) \equiv \alpha^*.\] Let $\psi_s(\theta,v)=(\theta,s^{\frac{1}{2}}v)$. Then as $s\to\infty$, we have
\[\psi_s^*\tilde g_s(t)=dv^2 + \frac{(st+c)}{2s}\alpha^*\otimes\alpha^*\to g^*(t),\quad \psi_s^*\tilde \alpha_s(t)\equiv\alpha^*.\]This finishes the proof.
\end{proof}

\subsection{Explicit dynamics of model solutions}

In this subsection, we present an elementary proof of Theorem~\ref{t:main.GH_limit} under the additional assumption that $g(0)$ is flat and $\alpha(0)$ is nonzero and parallel with respect to $g(0)$. Under these initial conditions, the solution to the $\alpha$-Ricci flow can be written down explicitly.
The main theorem of this subsection is as follows.
\begin{theorem}[Model case]\label{model.case}
	Let $(g(t),\alpha(t))$ be a solution to the $\alpha$-Ricci flow on $\mathbb{T}^2$ with $g(0)$ flat and $\alpha(0)$ nonzero and parallel with respect to $g(0)$. Then we have the following.
	\begin{enumerate}[label=(\roman*)]
		\item If $[\alpha]$ is rational, then $(M,t^{-1}g(t))\xrightarrow{GH}(S^1,d([\alpha]))$ as $t\to\infty$. Moreover, we have $d_{GH}((M,t^{-1}g(t)),(S^1,d([\alpha])))=O(t^{-1/2})$.
		\item If $[\alpha]$ is irrational, then  $(M,t^{-1}g(t))\xrightarrow{GH}\{\mathrm{pt}\}$ as $t\to\infty$.
	\end{enumerate}
    In either case, on the universal cover, we have $(\tilde g_s(t),\tilde \alpha_s(t))\xrightarrow{C^\infty}(g^*(t),\alpha^*(t))$ as $s\to\infty$.
\end{theorem}

Since $[\alpha(0)]\neq 0$, by Proposition~\ref{p:equivalent definition}, we choose a set of generators $[\gamma_1]$, $[\gamma_2]$ for $\pi_1(M)$ and define $a=\int_{\gamma_1}\alpha$ and $b=\int _{\gamma_2}\alpha$. Moreover, we choose the orientation of $[\gamma_1]$, $[\gamma_2]$ such that $a,b\ge0$.
We choose coordinates $(x,y)$ on the universal cover $\tilde M\cong\R^2$ so that $M\cong\R^2/\langle[\gamma_1],[\gamma_2]\rangle$ where $[\gamma_1]$ corresponds to $(x,y)\mapsto (x+1,y)$ and $[\gamma_2]$ corresponds to $(x,y)\mapsto (x,y+1)$. Let $\pi:\tilde M\to M$ be the covering map. Then we have $\tilde\alpha(0)=adx+bdy$. The solution to the $\alpha$-Ricci flow on $\tilde M$ is
\begin{equation}\label{model.solution.universal.cover}
    \left\{\begin{aligned}
	\tilde g(t)=&\frac{1}{2}a^2tdx^2+abtdxdy+\frac{1}{2}b^2tdy^2+g(0)\\
	\tilde \alpha(t)=&adx+bdy.
\end{aligned}\right.
\end{equation}
Now, the blow-down limit of the universal cover follows from  Proposition~\ref{p:canonical_alpha_expander_attractor}, since \eqref{model.solution.universal.cover} is translation invariant with nonzero $\tilde\alpha$.

To prove Gromov--Hausdorff convergence, we choose coordinates in which $t^{-1}\tilde g(t)$ converges to the Euclidean metric and track the lattice generators. Set $\rho=\sqrt{a^2+b^2}$, $\xi=(ax+by)/\rho$, and $\eta=(-bx+ay)/\rho$, and write $g(0)=A_0d\xi^2+2B_0d\xi d\eta+C_0d\eta^2$. In the coordinates $\theta=\rho\xi/\sqrt2$ and $v=\sqrt{C_0/t}\,\eta$, equation~\eqref{model.solution.universal.cover} gives
\[
\frac{\tilde g(t)}{t}
=\left(1+\frac{2A_0}{\rho^2t}\right)d\theta^2
+\frac{2\sqrt2B_0}{\rho\sqrt{C_0t}}d\theta dv+dv^2,
\qquad \tilde\alpha=\sqrt2\,d\theta.
\]
Thus in this coordinates, we have $|(t^{-1}\tilde g(t),\tilde\alpha)-(d\theta^2+dv^2,\sqrt2\,d\theta)|\le Ct^{-1/2}$. The lattice generators become
\begin{align}
 e_1(t)&=\frac{a}{\sqrt2}\partial_\theta-\frac{b\sqrt{C_0}}{\rho\sqrt t}\partial_v,\label{est1}\\
 e_2(t)&=\frac{b}{\sqrt2}\partial_\theta+\frac{a\sqrt{C_0}}{\rho\sqrt t}\partial_v.\label{est2}
\end{align}
They converge to $a\partial_\theta/\sqrt2$ and $b\partial_\theta/\sqrt2$, respectively, at rate $O(t^{-1/2})$.

\begin{figure}[h]
    \centering
    \resizebox{0.7\textwidth}{!}{\begin{circuitikz}
\tikzstyle{every node}=[font=\fontsize{19.9pt}{25.9pt}\selectfont]
\draw [ color={rgb,255:red,10; green,0; blue,248}, draw opacity=1, line width=0.8pt, -{Stealth[scale=1.5]}, ] (5.375,10.125) -- (5.375,12.375);
\draw [ color={rgb,255:red,10; green,0; blue,248}, draw opacity=1, line width=0.8pt, -{Stealth[scale=1.5]}, ] (5.375,10.125) -- (8,10.125);
\draw [ color={rgb,255:red,10; green,0; blue,248}, draw opacity=1, line width=0.8pt, dashed] (5.375,12.375) -- (8,12.375);
\draw [ color={rgb,255:red,10; green,0; blue,248}, draw opacity=1, line width=0.8pt, dashed] (8,10.125) -- (8,12.375);
\draw [-{Stealth[scale=1.5]}, ] (10.125,11.875) .. controls (10.5,12.75) and (11.5,12.75) .. (12,11.875) ;
\draw [dashed] (8,13.625) -- (4.625,9.125)node[font=\fontsize{18.2pt}{23.7pt}\selectfont, pos=0,left, fill={rgb,255:red,255; green,255; blue,255}, fill opacity=1, text opacity=1, inner xsep=0.080cm, inner ysep=0.085cm, rounded corners=0.020cm]{$(a,b)$};
\draw [ color={rgb,255:red,10; green,0; blue,248}, draw opacity=1, line width=0.8pt, -{Stealth[scale=1.5]}, ] (14.25,9.375) -- (15.25,11.125);
\draw [ color={rgb,255:red,10; green,0; blue,248}, draw opacity=1, line width=0.8pt, -{Stealth[scale=1.5]}, ] (14.25,9.375) -- (16.375,11.875);
\draw [dashed] (17.25,13.375) -- (13.875,8.875);
\node [font=\fontsize{15.2pt}{23.7pt}\selectfont, fill={rgb,255:red,255; green,255; blue,255}, fill opacity=1, text opacity=1, inner xsep=0.080cm, inner ysep=0.085cm, rounded corners=0.020cm, anchor=west] at (2.5,15.375) {$\frac{g(t)}{t}=(adx+bdy)^2/2+\frac{g(0)}{t}$};
\draw [ color={rgb,255:red,10; green,0; blue,248}, draw opacity=1, line width=0.8pt, dashed] (15.25,11.125) -- (17.375,13.625);
\draw [ color={rgb,255:red,10; green,0; blue,248}, draw opacity=1, line width=0.8pt, dashed] (16.375,11.875) -- (17.375,13.625);
\node [font=\fontsize{15.2pt}{23.7pt}\selectfont, fill={rgb,255:red,255; green,255; blue,255}, fill opacity=1, text opacity=1, inner xsep=0.080cm, inner ysep=0.085cm, rounded corners=0.020cm, anchor=west] at (13,15.5) {$\frac{g(t)}{t}\to d\theta^2+dv^2$};
\node [font=\fontsize{19.9pt}{25.9pt}\selectfont, fill={rgb,255:red,255; green,255; blue,255}, fill opacity=1, text opacity=1, inner xsep=0.080cm, inner ysep=0.085cm, rounded corners=0.020cm, anchor=west] at (12.5,10.625) {$e_2(t)$};
\node [font=\fontsize{19.9pt}{25.9pt}\selectfont, fill={rgb,255:red,255; green,255; blue,255}, fill opacity=1, text opacity=1, inner xsep=0.080cm, inner ysep=0.085cm, rounded corners=0.020cm, anchor=east] at (17.375,9.875) {$e_1(t)$};
\node [font=\fontsize{18.2pt}{23.7pt}\selectfont, fill={rgb,255:red,255; green,255; blue,255}, fill opacity=1, text opacity=1, inner xsep=0.080cm, inner ysep=0.085cm, rounded corners=0.020cm, anchor=west] at (15.375,13.625) {$(a,b)$};
\node [font=\fontsize{18.2pt}{23.7pt}\selectfont, fill={rgb,255:red,255; green,255; blue,255}, fill opacity=1, text opacity=1, inner xsep=0.080cm, inner ysep=0.085cm, rounded corners=0.020cm, anchor=west] at (4.25,11.25) {$e_2$};
\node [font=\fontsize{18.2pt}{23.7pt}\selectfont, fill={rgb,255:red,255; green,255; blue,255}, fill opacity=1, text opacity=1, inner xsep=0.080cm, inner ysep=0.085cm, rounded corners=0.020cm] at (6.75,9.625) {$e_1$};
\end{circuitikz}}
    \caption{Change of coordinates normalizing the metric, evolving the generators}
\end{figure}
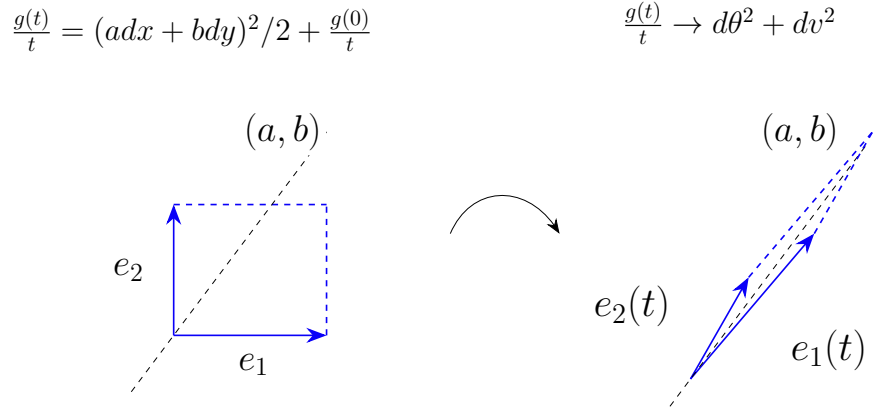

For $p_1,p_2\in \tilde M\cong\C$, we define $d_t(p_1,p_2):= d_{(M,t^{-1}g(t))}(\pi(p_1),\pi(p_2))$. Then we have 
\[d_t(p_1,p_2)=\min_{m_1,m_2\in\Z}d_{t^{-1}\tilde g(t)}(p_1,p_2+m_1e_1(t)+m_2e_2(t)).\]
We have the following Lemma showing that the fundamental domain of $(M,t^{-1}g(t))$ is contained in a neighborhood of a line segment determined by the periods of $\alpha$.
\begin{lemma}\label{gcd}
	Suppose there exist $\ell>0$, $r,s\in \Z$ such that $ra+sb=\ell$. Then $(M,t^{-1}g(t))$ is contained in a  $Ct^{-1/2}$-neighborhood of $\pi(S_\ell)$ where $C=C(\ell,r,s,a,b,g(0))$ and
	\[S_\ell=\{\frac{c}{\sqrt 2}\partial_\theta:0\le c\le \ell\}\subset \tilde M\cong\R^2.\]
\end{lemma}

\begin{proof}
	We consider the fundamental domain $F_t$ of the lattice generated by $e_1(t)$, $e_2(t)$:
	\[F_t=\{c_1e_1(t)+c_2e_2(t):0\le c_1,c_2\le 1\}.\] 
    It suffices to show that for any $p\in F_t$, $d_t(p,S_\ell):=\inf_{0\le c\le \ell}d_t(p,\frac{c}{\sqrt 2}\partial_\theta)\le Ct^{-1/2}$.
	Any point $p\in F_t$ can be written as $p=\frac{x}{\sqrt 2}\partial_\theta+\frac{y}{\sqrt 2}\partial_v$ for some $x,y\in\R$. By \eqref{est1} and \eqref{est2}, we have
	\[0\le x\le a+b,\quad |y|=O(t^{-1/2}).\]
	We write $x=\lfloor x/\ell\rfloor \ell+\{x/\ell\}\ell$ where $\lfloor x/\ell\rfloor$ is the integer part of $x/\ell$ and $\{x/\ell\}$ is the fractional part of $x/\ell$.
	Then we have
    \begin{align*}
	d_t\left(\frac{x}{\sqrt 2}\partial_\theta+\frac{y}{\sqrt 2}\partial_v,S_\ell\right)
	\le&d_t\left(\frac{x}{\sqrt 2}\partial_\theta+\frac{y}{\sqrt 2}\partial_v,
	\frac{\{x/\ell\}\ell}{\sqrt 2}\partial_\theta\right)\\
    \le&d_{t^{-1}\tilde g(t)}\left(
    \frac{x}{\sqrt 2}\partial_\theta+\frac{y}{\sqrt 2}\partial_v
    ,\frac{\{x/\ell\}\ell}{\sqrt 2}\partial_\theta
    +r\left\lfloor\frac{x}{\ell}\right\rfloor e_1(t)
    +s\left\lfloor\frac{x}{\ell}\right\rfloor e_2(t)\right)\\
    \le&(1+Ct^{-1/2})\left|
    \frac{x}{\sqrt 2}\partial_\theta+\frac{y}{\sqrt 2}\partial_v
    -\frac{\{x/\ell\}\ell}{\sqrt 2}\partial_\theta
    -r\left\lfloor\frac{x}{\ell}\right\rfloor e_1(t)
    -s\left\lfloor\frac{x}{\ell}\right\rfloor e_2(t)
    \right|\\
	\le&
	(1+Ct^{-1/2})\left|\left\lfloor \frac{x}{\ell} \right\rfloor\right|
	\left|
	\frac{\ell}{\sqrt{2}}\partial_\theta
	-r\frac{a}{\sqrt{2}}\partial_\theta
	-s\frac{b}{\sqrt{2}}\partial_\theta
	\right|
	+(1+Ct^{-1/2})\left|\frac{y}{\sqrt{2}}\partial_v\right|\\
    &\quad+(1+Ct^{-1/2})
    \left|\left\lfloor \frac{x}{\ell} \right\rfloor\right|
    \left(
    |r|\left|e_1(t)-\frac{a}{\sqrt{2}}\partial_\theta\right|
    +|s|\left|e_2(t)-\frac{b}{\sqrt{2}}\partial_\theta\right|
    \right)\\
	\le&Ct^{-1/2}.
\end{align*}
\end{proof}

We will need the following special case of Kronecker's approximation theorem (see, e.g., \cite[Theorem 7.7]{apostol1990modular}).
\begin{lemma}\label{dense}
  If $\lambda \in \mathbb{R} \setminus \mathbb{Q}$, then for any $y \in [0, 1]$ and any $\delta>0$, there exist $r, s \in \mathbb{Z}$ such that
  \[
    |y - r\lambda - s| < \delta.
  \]
\end{lemma}
We are now ready to prove Theorem~\ref{model.case}.
\begin{proof}[Proof of Theorem~\ref{model.case}]
	First, we show (i). By Proposition~\ref{p:equivalent definition}, $[\alpha]$ is rational if and only if $\frac{a}{b}\in \Q\cup\{\infty\}$. If $a,b>0$ and $\frac{a}{b}=\frac{p}{q}$ where $\gcd (p,q)=1$, then there exist $r,s\in\Z$ such that $r p+s q=1$ and hence $ra+sb=\frac{b}{q}=\ell_{a,b}$. If $a=0$ and $b>0$, then $0a+1b=b= \ell_{a,b}$. If $b=0$ and $a>0$, then $1a+0b=a= \ell_{a,b}$. In any case, we can apply Lemma~\ref{gcd} with $\ell=\ell_{a,b}$ and $r,s,a,b$ as above to get 
	\begin{equation}\label{epsilon}
		(M,t^{-1}g(t))\text{ is contained in a } Ct^{-1/2} \text{ neighborhood of }\pi(S_\ell).
	\end{equation}
    Now, we show that  there is $C=C(\ell,r,s,a,b,g(0))$ such that for any $0\le x,y\le 1$,
\begin{equation}\label{e:metric.on.S.ell}
	\left|d_t\left(
	\frac{x\ell}{\sqrt{2}}\partial_\theta,
	\frac{y\ell}{\sqrt{2}}\partial_\theta
	\right)
	-
	\frac{\ell}{\sqrt{2}}\rho(x,y)\right|\le Ct^{-1/2}.
\end{equation}
	where $\rho$ is the standard metric on $S^1=\R/\Z$ of length 1. 

Let $|\cdot|$ denote the standard Euclidean norm in the $(\theta,v)$ coordinates. For $p_1,p_2\in \R^2$ we define
\begin{align*}
d_E\left(
p_1,
p_2
\right)\
=
\min_{m_1,m_2\in\mathbb Z}
\left|
p_1-p_2-m_1e_1(t)-m_2e_2(t)
\right|.
\end{align*}
Then by \eqref{est1} and \eqref{est2},
\begin{align*}
d_E\left(
\frac{x\ell}{\sqrt{2}}\partial_\theta,
\frac{y\ell}{\sqrt{2}}\partial_\theta
\right)=\min_{m_1,m_2\in \Z}\left|
\frac{(x-y)\ell-m_1a-m_2b}{\sqrt{2}}\partial_\theta
+\frac{(m_1b-m_2a)\sqrt {C_0}}{\rho\sqrt t}\partial_v
\right|.
\end{align*}
Since $\partial_\theta$ and $\partial_v$ are orthogonal, for any $m_1,m_2\in\mathbb Z$, we have
$$
\left|
\frac{(x-y)\ell-m_1a-m_2b}{\sqrt{2}}\partial_\theta
+\frac{(m_1b-m_2a)\sqrt {C_0}}{\rho\sqrt t}\partial_v
\right|
\ge
\left|
\frac{(x-y)\ell-m_1a-m_2b}{\sqrt{2}}
\right|.
$$
Hence
$$
d_E\left(
\frac{x\ell}{\sqrt{2}}\partial_\theta,
\frac{y\ell}{\sqrt{2}}\partial_\theta
\right)
\ge
\frac{1}{\sqrt{2}}
\min_{m_1,m_2\in\mathbb Z}
\left|(x-y)\ell-m_1a-m_2b\right|=\frac{\ell}{\sqrt 2}\rho(x,y).
$$
On the other hand, let $n_1,n_2\in\mathbb Z$ be such that
$$
\left|(x-y)\ell-n_1a-n_2b\right|
=
\min_{m_1,m_2\in\mathbb Z}
\left|(x-y)\ell-m_1a-m_2b\right|=\ell\rho(x,y).
$$
Such $(n_1,n_2)$ can be chosen from $(0,0)$ or $\pm (r,s)$.
Then
\begin{align*}
d_E\left(
\frac{x\ell}{\sqrt{2}}\partial_\theta,
\frac{y\ell}{\sqrt{2}}\partial_\theta
\right)
&\le
\left|
\frac{(x-y)\ell-n_1a-n_2b}{\sqrt{2}}\partial_\theta
+\frac{(n_1b-n_2a)\sqrt {C_0}}{\rho\sqrt t}\partial_v
\right|\\
&\le
\frac{\ell}{\sqrt{2}}\rho(x,y)
+Ct^{-1/2}.
\end{align*}
Therefore,
$$
	\left|d_E\left(
	\frac{x\ell}{\sqrt{2}}\partial_\theta,
	\frac{y\ell}{\sqrt{2}}\partial_\theta
	\right)
	-
	\frac{\ell}{\sqrt{2}}\rho(x,y)\right|\le Ct^{-1/2}.
$$
Since $\left|\frac{\tilde g(t)}{t}-(d\theta^2+dv^2)\right|\le Ct^{-1/2}$, the corresponding distance functions converge uniformly on the relevant fundamental domains. Hence
$$
\left|d_t\left(
\frac{x\ell}{\sqrt{2}}\partial_\theta,
\frac{y\ell}{\sqrt{2}}\partial_\theta
\right)
-
\frac{\ell}{\sqrt{2}}\rho(x,y)\right|\le Ct^{-1/2}.
$$

	Thus, we see that $(\pi(S_\ell),t^{-1}g(t))$ converges to the metric space $(S^1,d([\alpha]))$. Combining (\ref{epsilon}) and \eqref{e:metric.on.S.ell}, we get \[d_{GH}((M,t^{-1}g(t)), (S^1,d([\alpha])))=O(t^{-1/2}).\]
	
	Next, we show (ii). Since $\frac{a}{b}\in \R\setminus\Q$, we have $b>0$. For any $0<\epsilon\le2b$ we apply Lemma~\ref{dense} with $\lambda=\frac{a}{b}$, $y=\frac{\epsilon}{2b}$, $\delta=\frac{\epsilon}{4b}$ to obtain $r,s\in \Z$ so that $|\frac{\epsilon}{2b}-r\frac{a}{b}-s|<\frac{\epsilon}{4b}$. Hence $|\frac{\epsilon}{2} -ra-sb|<\epsilon/4$. We apply Lemma~\ref{gcd} with $\ell=ra+sb$ where $\frac{1}{4} \epsilon\le \ell\le \frac{3}{4}\epsilon$. Then we get $(M,t^{-1}g(t))$ is contained in a $C(\epsilon)t^{-1/2}$-neighborhood of $\pi(S_\ell) \subset\pi(S_{3\epsilon/4})$. Since $\diam(\pi(S_{3\epsilon/4}))\le \sqrt{1+\frac{2A_0}{\rho^2t}}\frac{3\epsilon}{4\sqrt 2}$, for any $\epsilon>0$ there exists $t_0=t_0(\epsilon)>0$ so that for $t>t_0$, $\diam(M,t^{-1}g(t)) \le \epsilon$. Thus we have $(M,t^{-1}g(t))$ converges to a point in the Gromov--Hausdorff topology.
\end{proof}

\section{Geometric quantity estimates for the \texorpdfstring{$\alpha$}{alpha}-Ricci flow on \texorpdfstring{$\mathbb{T}^2$}{T2}}\label{s:geometric quantity}
To prove Theorem~\ref{t:main.GH_limit} with general initial data, we present several estimates for solutions to the $\alpha$-Ricci flow on the 2-torus, which will be used later. When $\alpha$ is exact, these estimates were previously obtained in \cite{list2008evolution} and \cite{lott2014ricci}. Since most of the proofs do not rely on the exactness of $\alpha$, the arguments for a general $\alpha$ follow the same idea. Consequently, we only sketch the proofs here.
\begin{theorem}[$\alpha$-norm estimate \cite{guo2025alpha}]\label{alpha.norm.estimate}
    Let $(M^n,g(t),\alpha(t))$ be a solution to the $\alpha$-Ricci flow. Then
    \[\frac{d}{dt}|\alpha|^2=\Delta |\alpha|^2-2|\nabla\alpha|^2-\frac{1}{2}|\alpha|^4.\]
    If $M$ is closed, then with $c_0=\sup_{M \times \{0\}} |\alpha|^2$ we have
    \[
\sup_{M \times \{t\}} |\alpha|^2 \le \ \frac{c_0}{1+\tfrac{c_0t}{2}}.\]
\end{theorem}

\begin{proof}The proof follows the  arguments in \cite[Lemma 3.2, Lemma 5.6]{list2008evolution} and \cite[(2.15), (2.19)]{lott2014ricci}.
    We recall the Bochner formula for a closed 1-form:
    \begin{align*}
        \frac{1}{2}\Delta |\alpha|^2=&-\langle \Delta_d\alpha,\alpha\rangle+|\nabla\alpha|^2+\ric(\alpha,\alpha)\\
        =&-\langle dd^*_g\alpha,\alpha\rangle+|\nabla\alpha|^2+\ric(\alpha,\alpha).
    \end{align*}
    Consequently, we have
    \begin{align*}
	\frac{d}{dt}|\alpha|^2=&-2\langle dd^*_g\alpha,\alpha\rangle+2\ric(\alpha,\alpha)-\frac{1}{2}|\alpha|^4\\
	=&\Delta|\alpha|^2-2|\nabla\alpha|^2-\frac{1}{2}|\alpha|^4.
\end{align*}
    Applying the maximum principle when $M$ is closed gives
    \[
\sup_{M \times \{t\}} |\alpha|^2 \le \ \frac{c_0}{1+\tfrac{c_0t}{2}}.\]
\end{proof}

\begin{theorem}
    \label{scalar.curvature.monotonicity}
    Let $(M^n,g(t),\alpha(t))$ be a solution to the $\alpha$-Ricci flow on $[0,T]$. We define $R^{ \alpha}=R-\frac{1}{4}|\alpha|^2$ and $\ric^{\alpha}=\ric -\frac{1}{4}\alpha\otimes \alpha$. Then we have
    \[
\frac{\partial  R^{ \alpha}}{\partial t} = \Delta  R^{ \alpha} + 2|{\mathrm{Ric}}^{\alpha}|^2+\frac{1}{2}(d^*_g\alpha)^2.
\]
    If $M$ is closed, then for any $t\in (0,T]$
\[
\inf_{M\times\{t\}} R^{\alpha}\ge -\frac{n}{2t}.
\]
\end{theorem}

\begin{proof}
The evolution equation for $ R^{ \alpha}$ on $M$ is proved in \cite{guo2025alpha}, see also \cite[(2.16)]{lott2014ricci}. If $M$ is closed, then using the Cauchy--Schwarz inequality $|\ric^\alpha|^2\ge \frac{1}{n}(R^\alpha)^2$ and the maximum principle, we get $
\inf_{M\times\{t\}} R^{\alpha}\ge -\frac{n}{2t}$.
\end{proof}

We now discuss some estimates for the 2-dimensional torus. Similar results can be derived for surfaces of genus  at least 2 as in \cite{lott2014ricci}.

\begin{theorem}[Volume estimate]\label{volume.estimate}
	Let $(g(t),\alpha(t))$ be a solution to the $\alpha$-Ricci flow on $M=\mathbb{T}^2$ and set $c_0=\sup_{M\times\{0\}} |\alpha|^2$. Letting $V(t)=\vol(M,g(t))$, we have
	\[V(0)\le V(t) \le V(0)\sqrt{1+\frac{c_0t}{2}}.\]    
\end{theorem}

\begin{proof}
    By a direct computation and the Gauss-Bonnet theorem, we have
    \[V'(t)=\int (-R+\frac{1}{4}|\alpha|^2)dV_{g(t)}=\frac{1}{4}\int|\alpha|^2dV_{g(t)}.\]
    Applying Theorem~\ref{alpha.norm.estimate}, we get
    \[0\le V'(t)\le \frac{c_0}{4+2c_0t}V(t).\]
    Integration gives the claimed bounds.
\end{proof}

\begin{theorem}[Type III estimate]\label{type.iii.estimate}
    Let $(g(t),\alpha(t))$ be a solution to the $\alpha$-Ricci flow on $M=\mathbb{T}^2$. Then the flow exists for all $t\in [0,\infty)$ and there exists $C>0$ such that
	\[|R(x,t)|\le \frac{C}{1+t}.\]    
\end{theorem}

\begin{proof}
    The proof closely follows that of \cite[Proposition 2.10]{lott2014ricci}, using Theorem~\ref{alpha.norm.estimate} in place of \cite[(2.15), (2.19)]{lott2014ricci}. We provide a brief sketch here.

    We first show that there is no finite-time singularity. Suppose, for the sake of contradiction, that there is $T<\infty$ such that the flow is singular at $T$. Then by \cite{guo2025alpha}, we have $\lim_{t\to T^-}\sup_{M\times\{t\}}|R|=\infty$. We can choose a sequence $(x_i,t_i)\in M\times [0,T)$ such that $Q_i:=|R(x_i,t_i)|=\sup _{M\times[0,t_i]}|R(x,t)|\to\infty$. Define the rescaled solutions $g_i(t)=Q_i g(t_i+Q_i^{-1}t)$ and $\alpha_i(t)=\alpha(t_i+Q_i^{-1}t)$. By \cite{guo2025alpha}, $(M\times[-Q_it_i,0],(x_i,0),g_i(t),\alpha_i(t))$ converges in the pointed Cheeger-Gromov topology to $(M_\infty\times(-\infty,0],(m_\infty,0),g_\infty(t),\alpha_\infty(t))$, where $(M_\infty,g_{\infty}(t))$ is a $\kappa$-solution in the sense of \cite[\S 11.2]{perelman2002entropy} and $\alpha_\infty=0$. Hence the limit must be $(S^2,g_{S^2})$. This contradicts the assumption that $M=\mathbb{T}^2$.
    
    It remains to show that $|R(x,t)|\le Ct^{-1}$, for $t\geq 1$.
    Suppose otherwise. Then \[\limsup_{t\to\infty}\sup_{M\times\{t\}}t|R|=\infty.\] Choose a sequence $\{T_i\}_{i=1}^\infty$ with $\lim_{i\to\infty}T_i=\infty$, and let $(x_i,t_i)\in M\times[0,T_i]$ be such that 
    \[t_i(T_i-t_i)|R(x_i,t_i)|=\sup_{M\times [0,T_i]}t(T_i-t)|R(x,t)|.\]
    Setting $Q_i=|R(x_i,t_i)|$, we consider the rescaled solutions
    \[g_i(t)=Q_i g(t_i+Q_i^{-1}t),\quad \alpha_i(t)=\alpha(t_i+Q_i^{-1}t).\]
    Then $Q_it_i\to\infty$, and $(g_i(t),\alpha_i(t))$ satisfies the $\alpha$-Ricci flow on a time interval $(A_i,\Omega_i)$, where $\lim_{i\to\infty}A_i=-\infty$ and $\lim_{i\to\infty}\Omega_i=\infty$. This can be seen as in \cite[Chapter 8, 5.1]{chow2004ricci}.
    Since $|R_i(x_i,0)|=1$, we can redefine $A_i,\Omega_i$ such that there exists a sequence $\{\gamma_i\}_{i=1}^\infty$, with $\lim_{i\to\infty}\gamma_i=1$ satisfying
    \[\sup_{M\times [A_i,\Omega_i]}|R_i(x,t)|\le \gamma_i.\]
    Following the idea of the proof in \cite[Lemma 2.11]{lott2014ricci}, and applying Theorem~\ref{alpha.norm.estimate},
    there is an $\epsilon>0$ such that $\inj _{(M,g_i(0))}(x_i)\ge \epsilon$ for all $i\in\N$.
    Taking the pointed limit of $(M\times[A_i,\Omega_i],(x_i,0),g_i(t),\alpha_i(t))$, we obtain a limit $(M_\infty\times\R,~(m_\infty,0),g_\infty(t),\alpha_\infty(t))$. By Theorem~\ref{alpha.norm.estimate},
    \[
    \sup_M|\alpha_i(t)|_{g_i(t)}^2
    =Q_i^{-1}\sup_M|\alpha(t_i+Q_i^{-1}t)|_{g(t_i+Q_i^{-1}t)}^2
    \le\frac{2}{Q_it_i+t}\longrightarrow0
    \]
    uniformly on compact time intervals, so $\alpha_\infty=0$. Theorem~\ref{scalar.curvature.monotonicity} also gives $R_i(t)\ge-\frac{1}{Q_it_i+t}$, hence $R_\infty\ge0$. Moreover, we have
    \[|R_i(x_i,0)|=1.\]
    In conclusion, $(M_\infty,g_\infty(t))$ is a complete eternal solution to the Ricci flow with uniformly bounded curvature. Since the spacetime maximum of $R$ is achieved at some point, \cite[Theorem 9.4]{chow2006hamilton} implies that $(M_\infty,g_\infty(t))$ is the cigar soliton. 
    In particular 
    \[\int_{M_\infty}R_{g_\infty(0)}dV_{g_\infty(0)}=4\pi.\]
    For sufficiently large $i$, there exists a subset $S_i\subset M$ such that
    \[\int_{S_i}R_{g(t_i)}dV_{g(t_i)}\ge 3\pi.\]
    On the other hand, by Theorem \ref{scalar.curvature.monotonicity}
    \[\int_{M\setminus S_i}R_{g(t_i)}dV_{g(t_i)}\ge -\frac{1}{t_i} \vol_{g(t_i)}(M\setminus S_i)\ge -\frac{V(t_i)}{t_i}\ge -\frac{C(1+\sqrt{t_i})}{t_i}\to0.\]
    Thus, for sufficiently large $i$, we obtain
    \[\int_MR_{g(t_i)}dV_{g(t_i)}\ge 2\pi,\]
    which contradicts the fact that $\chi(M)= 0$.
\end{proof}

\begin{theorem}[Systolic inequality]\label{growth.of.length.of.shortest.curve}
    Let $(g(t),\alpha(t))$ be a solution to the $\alpha$-Ricci flow on $M=\mathbb{T}^2$.
	For a homology class $\sigma\in H_1(M,\Z)$, let $[\alpha](\sigma):=\int_\sigma\alpha$. Define $L_\sigma(t):=\inf\{L(\gamma,g(t)):[\gamma]\in \sigma\}$ and
    $L(t)=\inf \{L_\sigma(t):0\neq \sigma\in H_1(M,\Z)\}$. Then the Dini derivatives of $L_\sigma(t)$ satisfy
    \begin{equation}\label{D^+L_sigma.upper.bound}
        D^+L_\sigma(t):=\limsup_{\Delta t\to0+}\frac{L_\sigma(t+\Delta t)-L_\sigma(t)}{\Delta t}\le\int_0^L-\frac{1}{2}R+\frac{1}{4}(\alpha(\gamma'))^2ds
    \end{equation}
    and
    \[D_-L_\sigma(t):=\liminf_{\Delta t\to0+}\frac{L_\sigma(t)-L_\sigma(t-\Delta t)}{\Delta t}\ge\int_0^L-\frac{1}{2}R+\frac{1}{4}(\alpha(\gamma'))^2ds.\]
    Furthermore, $L_\sigma$ satisfies the following growth bound:
    \begin{equation}\label{growth.L_sigma}
        L_\sigma(t)^2\ge \frac{[\alpha](\sigma)^2}{2}t+L_\sigma(0)^2.
    \end{equation}
    In particular, $L(t)$ is nondecreasing in $t$.
\end{theorem}
\begin{proof}
    Let $\gamma$ be a closed geodesic representing the homology class $\sigma$ such that $L(\gamma,g(t))=L_\sigma(t)=:L$. Bounding $L_{\sigma}(t+\Delta t)$ by the length of $\gamma$ at time $t+\Delta t$, we obtain
    \[\limsup_{\Delta t\to0+}\frac{L_\sigma(t+\Delta t)-L_\sigma(t)}{\Delta t}\le\frac{1}{2}\int_0^L\frac{\partial g}{\partial t}(\gamma',\gamma')ds=\int_0^L-\frac{1}{2}R+\frac{1}{4}(\alpha(\gamma'))^2ds.\]
    Similarly, bounding $L_{\sigma}(t-\Delta t)$ by the length of $\gamma$ at time $t-\Delta t$ gives
    \[D_-L_\sigma(t):=\liminf_{\Delta t\to0+}\frac{L_\sigma(t)-L_\sigma(t-\Delta t)}{\Delta t}\ge\int_0^L-\frac{1}{2}R+\frac{1}{4}(\alpha(\gamma'))^2ds.\]
    Applying the Cauchy-Schwarz inequality gives
	\[|[\alpha](\sigma)|=\left|\int_\gamma\alpha\right|=\left|\int_0^{L} \alpha(\gamma')ds\right|\le L_\sigma(t)^{1/2}(\int \alpha(\gamma')^2ds)^{1/2}.\]
    Because $\gamma$ is a stable closed geodesic, it follows that
    \[\int_{\gamma}R(\gamma(s))ds\le 0.\] Combining the above, we deduce that
	\[D_-L_\sigma(t)\ge \int_0^L-\frac{1}{2}R+\frac{1}{4}(\alpha(\gamma'))^2ds\ge \frac{[\alpha](\sigma)^2}{4}L_\sigma(t)^{-1}.\]
    Solving this differential inequality yields
	\[L_\sigma(t)^2\ge\frac{[\alpha](\sigma)^2}{2}t+L_\sigma(0)^2.\]
    In particular, $L_\sigma(t)$ is non-decreasing in $t$. Since $L(t)=\inf_{\sigma\neq0}L_\sigma(t)$, we conclude that $L(t)$ is also non-decreasing.
\end{proof}
\begin{theorem}[Diameter bound]\label{diameter.bound}
	Let $(g(t),\alpha(t))$ be a solution to the $\alpha$-Ricci flow on $M=\mathbb{T}^2$. Then the diameter satisfies
	\[\diam(M,g(t)) \le C(\sqrt{t}+1).\]
\end{theorem}

\begin{proof}
    Combining Theorem~\ref{volume.estimate} and Theorem~\ref{type.iii.estimate} yields
	\[\lim_{t\to\infty}V(t)\sup_{x\in M} |R(x,t)|=0.\]
	It then follows from \cite[Claim 2.15]{lott2014ricci} that \[\diam(M,g(t))\le CL(t)^{-1}V(t)\le CL(0)^{-1}V(t)\le C(\sqrt{t}+1).\]
\end{proof}

\begin{proof}[Proof of Theorem~\ref{t:2d_estimates}]
    The desired conclusion follows directly from combining Theorem~\ref{volume.estimate}, Theorem~\ref{type.iii.estimate}, and Theorem~\ref{diameter.bound}.
\end{proof}

\section{Collapsing of the 2D torus}\label{s:collapsing torus}
It follows from Theorem~\ref{volume.estimate}, Theorem~\ref{type.iii.estimate}, and Theorem~\ref{diameter.bound} that for each fixed \(t\), the sequence of manifolds \(\{(M,g_s(t))\}_{s\in(0,\infty)}\) collapses with bounded curvature and diameter. Consequently, the injectivity radius of \((M,g_s(t))\) tends to zero as \(s\to\infty\). Therefore, the standard compactness theorem for the \(\alpha\)-Ricci flow established in \cite{guo2025alpha} is not directly applicable. To extract a subsequential limit as \(s\to\infty\), one is naturally led to consider the Cheeger--Fukaya--Gromov collapsing theory \cite{cheeger1992nilpotent}. 

In the context of the Ricci flow, Lott \cite[Theorem 5.12]{lott2007long} proved a compactness theorem without an injectivity radius lower bound, in which the limits are Ricci flows on Riemannian groupoids. An analogous compactness theorem for the \(\alpha\)-Ricci flow can be established in a manner parallel to \cite{lott2007long}. For the purposes of the present work, however, it is more convenient to work on the universal cover. Readers interested in the groupoid perspective are referred to \cite{guo2025alpha}.
\subsection{Collapsing of the time slices}
First, we consider the convergence of the sequence $\{(M,g_{s_i}(t))\}_{i=1}^\infty$ for a fixed $t$.  Without loss of generality, we set $t=1$. 
We have the following developing theorem:
\begin{theorem}\label{develop}
Let $(M_i,g_i)$ be a sequence of $2$-tori  satisfying, for each $i$,
\begin{itemize}
    \item[(a)] $|\operatorname{Rm}(g_i)|\le K$,
    \item[(b)] $\vol (M_i,g_i)\to0$,
    \item[(c)] $\operatorname{diam}(M_i,g_i)\leq 1$.
\end{itemize}
Assume that there exists a subsequence, still denoted by $(M_i, g_i)$, such that
\begin{align*}
   \begin{tikzcd}[ampersand replacement=\&]
         \& (M_i,g_i) \arrow[r, "\text{GH}"] 
         \& (X, d)
\end{tikzcd}.
\end{align*}
Then:
\begin{enumerate}[label=(\roman*)]
    \item $(X,d)$ is either $S^1$ with diameter bounded by $1$ or a point $\{\mathrm{pt}\}$.
    \item The universal cover is noncollapsing, in particular, up to a subsequence, we have
    \begin{align*}
        \begin{tikzcd}[ampersand replacement=\&]
         \& (\tilde{M_i},\tilde{g_i}, p_i) \arrow[r, "\text{pt } C^{1,\alpha}"] 
         \& (\mathbb{R}^2, g_\infty, p)
    \end{tikzcd}
    \end{align*}
\end{enumerate}
for some $C^{1,\alpha}$ metric $g_\infty$ on $\R^2$.
\end{theorem}

\begin{proof}
To prove (i), we apply
\cite[Proposition 11.5]{fukaya1990hausdorff} to see that $(X,d)$
can only be a circle, an interval, or a point.
By \cite[Theorem 4.1]{katz2020torus}, the interval case is ruled out.
The diameter estimate follows directly from
Gromov--Hausdorff convergence.

To prove (ii), we first establish a controlled contraction property
on the universal covers. Fix $\epsilon\in(0,1/2)$.
We claim that 
for all sufficiently large $i$, every loop
$\gamma$ in $\tilde M_i$, based at a point $q$, admits a based
contraction $\{\gamma_s\}_{s\in[0,1]}$ satisfying $\gamma_0=\gamma$, $\gamma_1(t)=q$ for all $t\in [0,1]$ and
\begin{equation}\label{e:controlled.loop.contraction}
    L_{\tilde g_i}(\gamma_s)
    \le e^\epsilon L_{\tilde g_i}(\gamma)
    \qquad\text{for every }s\in[0,1].
\end{equation}

Suppose first that $X$ is a circle. By the \cite{cheeger1992nilpotent}, for sufficiently
large $i$ there is an invariant metric $h_i$ on the circle
fibration $M_i\to S^1$ satisfying $e^{-\epsilon}g_i\le h_i\le e^\epsilon g_i$.
After lifting to the universal cover, choosing arclength on the
lifted base, and removing the connection term, we obtain global
coordinates $(x,y)\in\mathbb R^2$ in which $\tilde h_i=dx^2+f_i(x)^2dy^2$.
Write $q=(x_0,y_0)$. First contract the $y$-coordinate of $\gamma$
linearly toward $y_0$, keeping its $x$-coordinate fixed. Then
contract the remaining $x$-coordinate linearly toward $x_0$.
Both stages fix $q$ and do not increase $\tilde h_i$-length.
Consequently,
\[
    L_{\tilde g_i}(\gamma_s)
    \le e^{\epsilon/2}L_{\tilde h_i}(\gamma_s)
    \le e^{\epsilon/2}L_{\tilde h_i}(\gamma)
    \le e^\epsilon L_{\tilde g_i}(\gamma),
\]
which proves \eqref{e:controlled.loop.contraction}.

If $X$ is a point, let $d_i=\operatorname{diam}(M_i,g_i)\to0$
and normalize the metrics by $\hat g_i=d_i^{-2}g_i$.
Then
\[
    \operatorname{diam}(M_i,\hat g_i)=1,
    \qquad
    |\operatorname{Rm}(\hat g_i)|\le Kd_i^2\longrightarrow0.
\]
After passing to a subsequence, $(M_i,\hat g_i)\xrightarrow{GH}(S^1,d)$ or $(M_i,\hat g_i)\xrightarrow{C^{1,\alpha}}(\mathbb{T}^2,g_\infty)$.
In the former case, the preceding invariant-metric approximation
applies. In the latter case, $g_\infty$ is flat, so $\hat g_i$ is bilipschitz to a fixed flat metric $g_\infty$ for sufficiently
large $i$. Thus $\hat g_i$ satisfies \eqref{e:controlled.loop.contraction}. Rescaling preserves the length ratio in
\eqref{e:controlled.loop.contraction}, and thus the claim is proved.

We first use \eqref{e:controlled.loop.contraction} to obtain a
uniform injectivity-radius bound. Choose $\delta>0$ such that $e^\epsilon\delta<\frac{\pi}{\sqrt K}$.
Suppose that $\operatorname{inj}_{\tilde g_i}(q)=r<\delta$
for some $q\in\tilde M_i$. The curvature bound implies that
there are no conjugate points at distance less than
$\pi/\sqrt K$. Thus Klingenberg's cut-locus lemma gives a
nonconstant geodesic loop based at $q$ of length $2r$.
By \eqref{e:controlled.loop.contraction}, it contracts through
based loops of length at most $2e^\epsilon r<\frac{2\pi}{\sqrt K},$
contradicting Klingenberg's long homotopy lemma.
Hence $\operatorname{inj}_{\tilde g_i}(q)\ge\delta$
    for every $q\in\tilde M_i$.

The universal covers are therefore noncollapsing, and standard
pointed compactness gives, after passing to a subsequence, $(\tilde M_i,\tilde g_i,p_i)
    \xrightarrow{pt\  C^{1,\alpha}}(M_\infty,g_\infty,p_\infty)$.

By \eqref{e:controlled.loop.contraction}, any loop in $M_\infty$
is approximated by loops in $\tilde M_i$ whose based contractions
remain in a uniformly bounded ball. Pulling these contractions
back through the convergence maps shows that the original loop
is null-homotopic. Hence $M_\infty$ is simply connected.

Minimizing geodesic rays in the universal covers pass to a ray
in $M_\infty$, so the limit is noncompact. Being also simply
connected, $M_\infty$ is diffeomorphic to $\mathbb R^2$.
This completes the proof.
\end{proof}

\subsection{Convergence of the flow on the universal cover}
As established in the preceding discussion, each time slice converges to $S^1$ or $\{\mathrm {pt}\}$ and the universal cover exhibits continuous symmetry. To capture the parabolic flow information, we need to establish convergence of the flows on the universal cover. Thanks to Theorem~\ref{develop}, we can work directly on the universal cover. We have the following corollary.
\begin{corollary}\label{c:convergence_CFG}
    Let $(g(t),\alpha(t))$ be a solution to the $\alpha$-Ricci flow on $M=\mathbb{T}^2$. For any sequence $s_i\to\infty$, there exists a subsequence, still denoted by $s_i$, and a limit space $(X,d)$ that is either $(S^1, d)$ or $\{\mathrm{pt}\}$, such that the following convergence holds. Choose $p \in M$ and a lift $\tilde{p} \in \tilde{M}\cong \R^2$ such that  $\pi(\tilde p)=p$, where $\pi:\tilde M\to M$ is the covering map. We have pointed $C^\infty$ convergence of the flow on $(0,\infty)$:
        \begin{align*}
        \begin{tikzcd}[ampersand replacement=\&]
         \& (\tilde M,\tilde g_{s_i}(t),\tilde\alpha_{s_i}(t), p) \arrow[r, "\text{pt } C^{\infty}"] 
         \& (\R^2, \tilde g_\infty(t),\tilde\alpha_\infty(t), p).
    \end{tikzcd}
    \end{align*}
     Moreover, the time $1$ slice has the following equivariant convergence diagram:
    \begin{center}
    \begin{tikzcd}
  & (\tilde M,\tilde{g}_{s_i}(1), \tilde\alpha_{s_i}(1),G_i,\tilde{p}) \arrow[r,  "\text{pt eq }C^\infty"] \arrow[d, "\pi"]
   & (\mathbb{R}^2,\tilde{g}_\infty(1),\tilde\alpha_\infty(1), G_\infty,\tilde{p}_\infty) \arrow[d,  "\pi_\infty"]\\
  & (M, g_{s_i}(1)) \arrow[r, "\text{GH}"]
   & (X,d),
\end{tikzcd}
\end{center}
where $G_i =\pi_1(M_i)\cong\mathbb{Z}\times\mathbb{Z}$ acts isometrically on $(\tilde M,\tilde{g}_{s_i}(1))$, $\tilde{g}_\infty(1)$ is some metric on $\mathbb{R}^2$, and $G_\infty\subset \operatorname{Isom}(\tilde{g}_\infty(t))$ preserves $\tilde \alpha_\infty(t)$  for all $t\in(0,\infty)$. Furthermore, $G_\infty\cong \R\times\Z$ if $(X,d)=(S^1,d)$, and $G_\infty\cong\R\times \R$ if $(X,d)=\{\mathrm {pt}\}$.
\end{corollary}
\begin{proof}
    By Theorem~\ref{alpha.norm.estimate} and Theorem~\ref{type.iii.estimate}, we have the estimates
\[
|\alpha_{s_i}(t)|_{g_{s_i}(t)}^2 \le \frac{c_0 s_i}{1 + \frac{c_0}{2} s_i t}, \qquad 
|R(g_{s_i}(t))| \le \frac{C s_i}{1+s_i t}.
\]
In particular, for any $T_0>0$, these yield uniform bounds on the curvature and on $|\alpha|$ over the time interval $[T_0,\infty)$. Therefore, the noncollapsing result in Theorem \ref{develop} and the compactness theorem in~\cite{guo2025alpha} apply and give the smooth pointed convergence on the universal cover.

    By Theorem~\ref{type.iii.estimate} and Gromov's compactness theorem \cite[Theorem 5.3]{gromov2007metric}, there exists a metric space $(X,d)$ such that $(M, g_{s_i}(1))\to (X,d)$ in the Gromov--Hausdorff topology. By Theorem~\ref{volume.estimate} and Theorem~\ref{develop}, we know that $(X,d)$ is either $(S^1, d)$ or a point $\{\mathrm{pt}\}$. The equivariant convergence and the fact that $G_\infty$ acts by isometries on $(\R^2,\tilde g_\infty(1))$ follow from \cite{fukaya1988boundary}.
Because $\tilde\alpha_{s_i}(1)$ is invariant under $G_i$, the equivariant smooth convergence on the universal cover implies that $\tilde \alpha_\infty(1)$ is invariant under $G_\infty$. 
    
    We would like to show that $G_\infty$ acts by isometries on $(\R^2,\tilde g_\infty(t))$ and preserves $\tilde\alpha_\infty(t)$ for any $t\in (0,\infty)$. Fix $\gamma_\infty\in G_\infty$. By equivariant convergence at
    time one, there exist $\gamma_i\in G_i$ converging locally uniformly
    to $\gamma_\infty$.
    Since the $\gamma_i$ are time-one isometries, their differentials
    are uniformly bounded on compact sets. The identity
    \[
    \gamma_i\circ\exp_p^{\tilde g_i(1)}
    =
    \exp_{\gamma_i(p)}^{\tilde g_i(1)}\circ d\gamma_i|_p
    \]
    and smooth convergence of the metrics upgrade this convergence
    to $C^\infty_{\mathrm{loc}}$. Passing to the limit in
    \[
    \gamma_i^*\tilde g_i(t)=\tilde g_i(t),
    \qquad
    \gamma_i^*\tilde\alpha_i(t)=\tilde\alpha_i(t)
    \]
    shows that $G_\infty$ preserves both $\tilde g_\infty(t)$ and
    $\tilde\alpha_\infty(t)$ for every $t>0$.
    
    We now show that $G_\infty$ is either $\R\times\Z$ or $\R\times\R$. For $\delta>0$ sufficiently small, consider the local fundamental group $\pi_1(M_i,x_i;\delta)$ of \cite[Definition 7.8]{fukaya1990hausdorff}. A globally null-homotopic based loop of length at most $2\delta$ lifts to a closed loop on $\tilde M_i$. By~\eqref{e:controlled.loop.contraction}, it contracts through based loops of length at most $2e^\epsilon\delta<4\delta$, so the projected contraction lies in $B(x_i,2\delta)$. Thus no globally trivial short loop remains locally nontrivial, and the local pseudogroup identifies with restrictions of the deck action of $\pi_1(M_i)$ on $\tilde M_i$. Let $H$ denote its limit pseudogroup.

    If $(X,d)=(S^1,d)$, then by \cite[Proposition 11.5]{fukaya1990hausdorff}, the action of $H$ is free since $X$ is a codimension one manifold. Because $G_\infty$ is locally isomorphic to $H$, the action of $G_\infty$ is free. Consequently, $\mathbb{R}^2$ is a principal $G_\infty$-bundle over $S^1$. A standard argument then shows $G_\infty \cong \mathbb{R}\times\mathbb{Z}$.
    
    If $(X,d)=\{\mathrm{pt}\}$, $G_\infty$ is abelian because it is the limit of abelian groups. Furthermore, since $X$ is a point, $G_\infty$ acts transitively on $\mathbb{R}^2$. Choosing a basepoint $p\in \mathbb{R}^2$, we have $\mathbb{R}^2\cong G_\infty/K$, where $K \subset G_\infty$ is the isotropy subgroup at $p$. For any $x\in \mathbb{R}^2$, choose $g\in G_\infty$ such that $g\cdot p=x$. Then for any $k\in K$, we have $k\cdot x=k\cdot(g\cdot p)=g\cdot(k\cdot p)=g\cdot p=x$. Thus, $k$ acts trivially on $\mathbb{R}^2$. It follows that $K=\{e\}$ is trivial, and therefore $G_\infty\cong \mathbb{R}\times \mathbb{R}$ as a Lie group.
\end{proof}


\begin{remark}\label{fiber bundle}
If $(X,d) = \{\mathrm{pt}\}$, then on the universal cover, $(\tilde{M}_\infty,\tilde{g}_\infty(t),\tilde\alpha_\infty(t))$ is $\mathbb{R}^2$-invariant. Hence, $(\tilde g_\infty(t),\tilde\alpha_\infty(t))$ is one of the solutions given in Lemma~\ref{l:R2.invariant.solution}.

If $(X,d) = (S^1, d)$, then $(\tilde{M}_\infty,\tilde{g}_\infty(t),\tilde\alpha_\infty(t))$ can be viewed as an $\alpha$-Ricci flow on an $\mathbb{R}$-bundle over $S^1$ i.e., we can write
\begin{align*}
\tilde g_\infty(t) =& u(\theta,t)d\theta^2 + G(\theta,t)(dv+A(\theta,t)d\theta)^2,\\
\tilde\alpha_\infty(t)=& a(\theta,t)d\theta+b(dv+A(\theta,t)d\theta)
\end{align*}
where $\partial_\theta$ and $\partial_v$ correspond to the $S^1$ and $\mathbb{R}$ direction respectively, and $u,A$ and $G$ are functions on $S^1\times(0,\infty)$. Using tools in Sections~\ref{s:condition monotonicity},~\ref{s:GHlimit_vs_rationality} and~\ref{s:dynamic stability},  we will show in Proposition~\ref{p:base.S^1.implies.universal.cover.expander} that $(\tilde g_\infty(t),\tilde \alpha_\infty(t))$ is the expander $(g^*(t),\alpha^*(t))$.
\end{remark}

\section{Conditional monotonicity of the \texorpdfstring{$\W_+$}{W+}-functional}\label{s:condition monotonicity}
\subsection{Locally \texorpdfstring{$\mathbb{R}^N$}{RN}-invariant \texorpdfstring{$\alpha$}{alpha}-Ricci flow}
Recall from Remark~\ref{fiber bundle} that if $(X,d)=(S^1,d)$, the blow-down limit of the $\alpha$-Ricci flow on the universal cover $\mathbb{R}^2$ descends to an $\mathbb{R}$-invariant $\alpha$-Ricci flow on an $\mathbb{R}$-bundle over $S^1$. 

In this section, following \cite[Section~4.1]{lott2010dimensional}, we consider a general twisted $\mathbb{R}^N$-bundle and the invariant metrics on it. Specifically, let $\mathbb{R}^N \to M \xrightarrow{\pi} B$ be a fiber bundle over a connected, oriented, compact base $B$, and let $\mathbb{R}^N \to E \xrightarrow{p} B$ be a flat vector bundle over $B$. Let $E \times_B M$ denote the fibered product $\coprod_{b \in B} (E_b \times M_b)$. Assume that there exists a smooth map $E \times_B M \to M$ such that over each base point $b \in B$, the fiber $E_b$ acts freely and transitively on $M_b$. Furthermore, assume that this action is compatible with the flat connection on $E$. By this, we mean that for any sufficiently small open set $U \subset B$ admitting a local trivialization $E|_U \cong U \times \mathbb{R}^N$, the preimage $\pi^{-1}(U)$ inherits a free $\mathbb{R}^N$-action, making it the total space of a principal $\mathbb{R}^N$-bundle over $U$. In this way, $M$ may be regarded as a principal $\mathbb{R}^N$-bundle over $B$ twisted by the flat vector bundle $E$. One can then define a connection $A$ on $M$ with the property that its restriction $A|_{\pi^{-1}(U)}$ is an $\mathbb{R}^N$-valued connection on the local principal bundle. We assume
throughout this section that $E$ admits a parallel positive
fiber density, or equivalently, that every holonomy
transformation $H$ satisfies $|\det H|=1$.

Let $s: U \to \pi^{-1}(U)$ be a local section over a coordinate chart $(U, \{x^\mu\}_{\mu=1}^n)$ of $B$, and let $\{e_i\}_{i=1}^N$ denote the standard basis for $\mathbb{R}^N$. On $\pi^{-1}(U)$, we adopt the local coordinates $(x^\mu, x^i)$ defined by $(x^\mu, x^i) \mapsto \exp\left(\sum_{i=1}^N x^i e_i\right) \cdot s(x^\mu)$, where $\{x^i\}_{i=1}^N$ are the fiber coordinates. Let $A^i(t) = A^i_\mu(t) dx^\mu$ denote the local connection 1-forms on $U$ pulled back via $s$. Suppose $\bar g(t)$ is a family of metrics and $\bar \alpha(t)$ is a family of closed 1-forms on $M$, both invariant under the action of $\mathbb{R}^N$. Then, in these local coordinates, they decompose as
\begin{equation}\label{invariant.flow}
\begin{aligned}
    \bar g(t) &= \sum_{i,j=1}^N G_{ij}(t)(dx^i+A^i(t))(dx^j+A^j(t))+\sum_{\mu,\nu=1}^n g_{\mu\nu}(t)dx^\mu dx^\nu, \\
    \bar\alpha(t) &= \sum_{\mu=1}^n a_\mu(t) dx^\mu + \sum_{i=1}^N b_i(t)(dx^i +A^i(t)),
\end{aligned}
\end{equation}
where $G_{ij}(t), A^i_\mu(t), g_{\mu\nu}(t), a_\mu(t)$, and $b_i(t)$ are smooth functions on $U \times (0,\infty)$. Following \cite[(4.6),(4.7)]{lott2007long}, we denote the components of the Ricci tensor of $(M,\bar g)$ by $\bar R_{ij}, \bar R_{i\mu}, \bar R_{\mu\nu}$, the scalar curvature of $(M,\bar g)$ by $\bar R$. We define the modified scalar curvature as $\bar R^{\bar \alpha}=\bar R-\frac{1}{4}|\bar \alpha|^2$. Let $F^i=dA^i$ denote the curvature 2-form of the connection $A^i$, whose components are $F_{\mu\nu}^i=\partial_\mu A_\nu^i-\partial_\nu A_\mu^i$.

We write $f_{,\mu_1\dots\mu_k}$ for the $\mu_1\dots\mu_k$-th coordinate derivatives of a function $f$ on $B$. Similarly, we denote by $T_{;\mu_1\dots\mu_k}$ the $\mu_1\dots\mu_k$-th covariant derivatives of a tensor $T$ on $B$, e.g., $a_{\mu;\nu}=a_{\mu,\nu}-a_{\gamma}\Gamma_{\mu\nu}^\gamma$, $F_{\mu\nu;\gamma}^i=F_{\mu\nu,\gamma}^i-F_{\delta\nu}^i\Gamma_{\mu\gamma}^\delta-F_{\mu\delta}^i\Gamma_{\nu\gamma}^\delta$, and $G_{ij;\mu\nu}=G_{ij,\mu\nu}-\Gamma_{\mu\nu}^\sigma G_{ij,\sigma}$, where $\Gamma_{\mu\nu}^\gamma$ denote the Christoffel symbols of $g_{\mu\nu}$ on $B$. Fix a point $b\in B$ at which we will be working. We choose the section $s$ such that $A^i(0,b)=0$. The Christoffel symbols of $\bar g(0)$ are then given by
\begin{equation}\label{christoffel.symbol}
\begin{aligned}
    \bar\Gamma_{ij}^k &= 0,\quad \bar\Gamma_{ij}^\mu=-\frac{1}{2}g^{\mu\nu}G_{ij,\nu},\quad \bar\Gamma_{i\mu}^j=\frac{1}{2}G^{jk}G_{ik,\mu},\\
    \bar\Gamma_{i\mu}^\nu &= \frac{1}{2}g^{\nu\gamma}G_{ij}F_{\mu\gamma}^j,\quad \bar \Gamma_{\mu\nu}^i=\frac{1}{2}( A_{\nu,\mu}^i+ A_{\mu,\nu}^i),\quad
    \bar \Gamma_{\mu\nu}^\gamma=\Gamma_{\mu\nu}^\gamma.
\end{aligned}
\end{equation}

We have the following reduction of the $\alpha$-Ricci flow system \eqref{alpha_Ricci_flow} for locally $\mathbb{R}^N$-invariant $(\bar g(t),\bar \alpha(t))$ as in \eqref{invariant.flow}.

\begin{proposition}\label{invariant.alpha.ricci.flow}
If $(\bar g(t),\bar \alpha(t))$ is of the form \eqref{invariant.flow} and is a solution to the $\alpha$-Ricci flow, then $g_{\mu\nu}, G_{ij}, A^i_\mu, a_{\mu}$, and $b_i$ satisfy the following equations:
\begin{equation}\label{e:invariant.alpha.ricci.flow}
\begin{aligned}
	\frac{\partial G_{ij}}{\partial t} &= g^{\mu\nu}G_{ij;\mu\nu}+\frac{1}{2}g^{\mu\nu}G^{kl}G_{kl,\mu}G_{ij,\nu}-g^{\mu\nu} G^{kl}G_{ik,\mu}G_{jl,\nu}-\frac{1}{2}g^{\mu\gamma}g^{\nu\delta}G_{ik}G_{jl}F_{\mu\nu}^kF_{\gamma\delta}^l+\frac{1}{2}b_ib_j, \\
	\frac{\partial A_\mu^i}{\partial t} &= -g^{\gamma\delta}F_{\mu\gamma;\delta}^i-g^{\gamma\delta}G^{ij}G_{jk,\gamma}F_{\mu\delta}^k-\frac{1}{2}g^{\gamma\delta}G^{kl}G_{kl,\gamma}F_{\mu\delta}^i+\frac{1}{2}a_\mu b_jG^{ij}, \\
	\frac{\partial g_{\mu\nu}}{\partial t} &= -2R_{\mu\nu}+G^{ij}G_{ij;\mu\nu}-\frac{1}{2}G^{ij}G_{jk,\mu}G^{kl}G_{il,\nu}+g^{\gamma\delta}G_{ij}F_{\mu\gamma}^iF_{\nu\delta}^j+\frac{1}{2}a_\mu a_\nu, \\
	\frac{\partial a_\mu}{\partial t} &= -(d d_{g,G}^*a)_{\mu}+b_ig^{\gamma\delta}F_{\mu\gamma;\delta}^i+b_ig^{\gamma\delta}G^{ij}G_{ik,\gamma}F_{\mu\delta}^k+\frac{1}{2}b_ig^{\gamma\delta}G^{kl}G_{kl,\gamma}F_{\mu\delta}^i-\frac{1}{2}a_\mu b_ib_jG^{ij}, \\
	\frac{\partial b_i}{\partial t} &= 0,
\end{aligned}
\end{equation}
where 
\[
d^*_{g,G}a = -g^{\mu\nu}a_{\mu;\nu}-\frac{1}{2}g^{\mu\nu}G^{ij}G_{ij,\nu}a_\mu.
\]
\end{proposition}

\begin{proof}
The metric component of the $\alpha$-Ricci flow reads
\begin{align*}
	&\frac{d}{dt}\Big(G_{ij}(dx^i+A^i)(dx^j+A^j)+g_{\mu\nu}dx^\mu dx^\nu\Big) \\
    =&\frac{d}{dt}\bar g = -2\overline{\mathrm{Ric}} + \frac{1}{2}\bar \alpha\otimes\bar\alpha \\ 
	=&-2\bar R_{ij}(dx^i+A^i)(dx^j+A^j)-4\bar R_{i\mu}(dx^i+A^i)dx^\mu-2\bar R_{\mu\nu}dx^\mu dx^\nu \\
	&+\frac{1}{2}(a_\mu dx^\mu+b_i (dx^i +A^i))(a_\nu dx^\nu+b_j (dx^j +A^j)).
\end{align*}
The evolution equations for $\frac{\partial G_{ij}}{\partial t}$, $\frac{\partial A_\mu^i}{\partial t}$, and $\frac{\partial g_{\mu\nu}}{\partial t}$ follow from expanding both sides and comparing coefficients. Note that this requires the formulas for $\bar R_{ij}, \bar R_{i\mu}, \bar R_{\mu\nu}$ given in \cite[(4.6)]{lott2007long}. Equivalently, one can directly add the components of $\frac{1}{2}\bar\alpha\otimes\bar \alpha$ to the system in \cite[(4.9)]{lott2007long}.

Since $\bar \alpha$ is closed, we have
\begin{align*}
0 &= d\big(a_\mu dx^\mu+b_i (dx^i +A^i)\big) = a_{\mu, \nu}dx^\nu \wedge dx^\mu + b_{i,\mu} dx^\mu \wedge (dx^i+A^i) + b_i dA^i \\
&= (a_{\nu,\mu} + b_{i,\mu} A^i_\nu + \frac{1}{2}b_i F^i_{\mu\nu}) dx^\mu \wedge dx^\nu + b_{i,\mu} dx^\mu \wedge dx^i.
\end{align*}
Setting the coefficient of $dx^\mu \wedge dx^i$ to zero, we obtain 
\begin{equation}\label{db=0}
    b_{i,\mu}=0.
\end{equation} 
Hence, each $b_i$ is a constant on $B$. Using \eqref{christoffel.symbol} and \eqref{db=0}, we compute
\begin{align*}
	&d^*_{\bar g}\big(a_\mu dx^\mu+b_i (dx^i +A^i)\big) \\
	=&-g^{\mu\nu}(a_{\mu,\nu}+b_iA_{\mu,\nu}^i)+g^{\mu\nu}(\Gamma_{\mu\nu}^\gamma a_\gamma+\bar\Gamma_{\mu\nu}^ib_i)+G^{ij}\bar\Gamma_{ij}^\mu a_{\mu} \\ 
    =&-g^{\mu\nu}(a_{\mu;\nu}+b_iA^i_{\mu,\nu})+g^{\mu\nu }A_{\mu,\nu}^ib_i-\frac{1}{2}g^{\mu\nu}G^{ij}G_{ij,\nu}a_\mu \\
	=&-g^{\mu\nu}a_{\mu;\nu}-\frac{1}{2}g^{\mu\nu}G^{ij}G_{ij,\nu}a_\mu \\
	=&d^*_{g,G}a.
\end{align*}

The $1$-form component of the $\alpha$-Ricci flow reads
\begin{align*}
    &\frac{d}{dt}(a_\mu+b_iA^i_\mu) dx^\mu+\frac{db_i}{dt}dx^i \\
    =&\frac{d}{dt}\big(a_\mu dx^\mu+b_i (dx^i +A^i)\big) \\
    =&-d d^*_{\bar g}\big(a_\mu dx^\mu+b_i (dx^i +A^i)\big) \\
    =&-(d d^*_{g,G}a)_{\mu}dx^\mu,
\end{align*}
where the last equality follows from the fact that $d^*_{g,G}a$ is a function on $B$, meaning $d d^*_{g,G}a$ has no fiber components. Finally, substituting the evolution equation for $\frac{\partial A_\mu^i}{\partial t}$ into the above yields the equations for $\frac{\partial a_\mu}{\partial t}$ and $\frac{\partial b_i}{\partial t}$.
\end{proof}

From \eqref{db=0} and $\frac{\partial b_i}{\partial t}=0$, we obtain the following corollary. 
\begin{corollary}\label{fiber.constant.in.spacetime}
The fiber components $b_i$ of $\bar \alpha$ are constant in spacetime.
\end{corollary}

Following \cite[(4.10)]{lott2007long}, we modify the system \eqref{e:invariant.alpha.ricci.flow} by a Lie derivative with respect to $-\nabla\log\sqrt{\det G}$, and additionally add an exact form to the evolution equation for $A_\mu^i$. This yields the gauged system:
\begin{equation}\label{gauged.invariant.alpha.ricci.flow}
\begin{aligned}
	\frac{\partial G_{ij}}{\partial t} &= g^{\mu\nu}G_{ij;\mu\nu}-g^{\mu\nu} G^{kl}G_{ik,\mu}G_{jl,\nu}-\frac{1}{2}g^{\mu\gamma}g^{\nu\delta}G_{ik}G_{jl}F_{\mu\nu}^kF_{\gamma\delta}^l+\frac{1}{2}b_ib_j, \\
	\frac{\partial A_\mu^i}{\partial t} &= -g^{\gamma\delta}F_{\mu\gamma;\delta}^i-g^{\gamma\delta}G^{ij}G_{jk,\gamma}F_{\mu\delta}^k+\frac{1}{2}a_\mu b_jG^{ij}, \\
	\frac{\partial g_{\mu\nu}}{\partial t} &= -2R_{\mu\nu}+\frac{1}{2}G^{ij}G_{jk,\mu}G^{kl}G_{il,\nu}+g^{\gamma\delta}G_{ij}F_{\mu\gamma}^iF_{\nu\delta}^j+\frac{1}{2}a_\mu a_\nu, \\
	\frac{\partial a_\mu}{\partial t} &= -(d d_{g}^*a)_{\mu}+b_ig^{\gamma\delta}F_{\mu\gamma;\delta}^i+b_ig^{\gamma\delta}G^{ij}G_{ik,\gamma}F_{\mu\delta}^k-\frac{1}{2}a_\mu b_ib_jG^{ij}, \\
	\frac{\partial b_i}{\partial t} &= 0.
\end{aligned}
\end{equation}
\subsection{$\mathcal{W}_+$-functional and Monotonicity}
Motivated by \cite{lott2010dimensional}, in order to understand the long-time behavior of the $\alpha$-Ricci flow under the Type III estimate, we require a modified version of the $\mathcal{W}_+$-functional and its corresponding monotonicity formula. For the definition and geometric significance of the usual $\mathcal{W}_+$-functional, we refer the reader to \cite{feldman2005entropy} for the standard Ricci flow case, and to \cite{lott2010dimensional} for the invariant Ricci flow case; here, we focus primarily on the explicit computations.

\begin{definition}\label{W plus functional}
Given $f \in C^\infty(B)$ and $t > 0$, we define
\begin{align*}
	&\W_+(G_{ij}, A^i_\mu, g_{\mu\nu}, a_\mu, b_i, f, t) \\
	=& \int_B \Big[ t \Big( |\nabla f|^2 + R - \frac{1}{4}g^{\mu\nu}a_\mu a_\nu - \frac{1}{4}G^{ij}b_i b_j - \frac{1}{4}g^{\mu\nu}G^{ij}G_{jk,\mu}G^{kl}G_{li,\nu} \\
	&\quad - \frac{1}{4}g^{\mu\gamma}g^{\nu\delta}G_{ij}F_{\mu\nu}^i F_{\gamma\delta}^j \Big) - f + n \Big] (4\pi t)^{-n/2} e^{-f} \, d\vol_B.
\end{align*}
\end{definition}

Before establishing the monotonicity of the $\mathcal{W}_+$-functional, we first derive some of its quantitative properties. The following lemma provides a lower bound for the modified scalar curvature, parallel to the classical estimate for the Ricci flow.

\begin{lemma}\label{l:scalar.curvature.monotonicity}
Let $(\bar g(t), \bar \alpha(t))$ be a locally $\mathbb{R}^N$-invariant $\alpha$-Ricci flow on a compact base $B$ of dimension $n$ defined for $t \in [0,T]$. Then, for all $t \in (0,T]$, we have 
\[
\bar R^{\bar\alpha}_{\min}(t) \ge -\frac{n+N}{2t}
\]
where $\bar R^{\bar\alpha}_{\min}(t)=\min_B (\bar R(t)-\frac{1}{4}|\bar\alpha(t)|^2)$.
\end{lemma}
\begin{proof}
Theorem \ref{scalar.curvature.monotonicity} gives
\[
\begin{aligned}
\frac{\partial \bar R^{\bar \alpha}}{\partial t}
&=\bar\Delta \bar R^{\bar \alpha}
+2\left|\overline{\operatorname{Ric}}
-\frac14\bar\alpha\otimes\bar\alpha\right|_{\bar g}^2
+\frac12|d_{\bar g}^*\bar\alpha|^2\ge \bar\Delta \bar R^{\bar \alpha}+\frac{2}{n+N}(\bar R^{\bar \alpha})^2.
\end{aligned}
\]
Since $\bar R^{\bar\alpha}$ descends to $B$, we have
\[
\bar\Delta\bar R^{\bar\alpha}
=\Delta\bar R^{\bar\alpha}
+\frac12g^{\mu\nu}G^{ij}G_{ij,\mu}\bar R^{\bar\alpha}_{,\nu}.
\] Since $B$ is
compact, the maximum principle therefore yields
\[
\min_B \bar R^{\bar \alpha}(t)\ge-\frac{n+N}{2t}.
\]
\end{proof}

\begin{lemma}\label{volume.ratio.lower.bound.by.W_+}
If $(4\pi t)^{-n/2}\int_B e^{-f}\, d\vol_B = 1$, then 
\begin{align*}
    \W_+(G_{ij}, A^i_\mu, g_{\mu\nu}, a_\mu, b_i, f, t) 
    &\ge t{\bar R^{\bar\alpha}}_{\min} + n + \frac{n}{2}\log(4\pi) - \log\big(t^{-n/2}\vol(B, g_{\mu\nu}(t))\big) \\
    &\ge \frac{n-N}{2} + \frac{n}{2}\log(4\pi) - \log\big(t^{-n/2}\vol(B, g_{\mu\nu}(t))\big).
\end{align*}
\end{lemma}

\begin{proof}
Using the expression for the integrand of $\mathcal{W}_+$, we have
\begin{align*}
    &\W_+(G_{ij}, A^i_\mu, g_{\mu\nu}, a_\mu, b_i, f, t) \\
	= &\int_B \Big[ t \Big( \big|\nabla f + \nabla\log\sqrt{\det(G_{ij})}\big|^2 + \bar R^{\bar \alpha} \Big) - f + n \Big] (4\pi t)^{-n/2}e^{-f}\, d\vol_B \\
    \ge& t{\bar R^{\bar\alpha}}_{\min} + n - \int_B f (4\pi t)^{-n/2}e^{-f}\, d\vol_B \\
    \ge& t{\bar R^{\bar\alpha}}_{\min} + n + \frac{n}{2}\log(4\pi) - \log\big(t^{-n/2}\vol(B, g_{\mu\nu}(t))\big),
\end{align*}
where the final inequality follows from applying Jensen's inequality to the probability measure $\frac{d\vol_B}{\vol(B, g_{\mu\nu}(t))}$ using the convex function $\phi(x) = x\log x$ evaluated at the density function $\vol(B, g_{\mu\nu}(t))(4\pi t)^{-n/2}\exp\big({-f}\big)$. Combining this estimate with Lemma~\ref{l:scalar.curvature.monotonicity} yields the second inequality in the statement.
\end{proof}

Now, we are ready to prove the monotonicity formula.
\begin{theorem}[Monotonicity formula]\label{t:evolution.of.w+.functional}
Under the flow equations (\ref{gauged.invariant.alpha.ricci.flow}) and
\begin{equation}\label{e:f_equation}
	\frac{\partial f}{\partial t}=-\Delta f+|\nabla f|^2-(R-\frac{1}{4}g^{\mu\nu}G^{ij}G_{jk,\mu}G^{kl}G_{il,\nu}-\frac{1}{2}g^{\mu\nu}g^{\gamma\delta}G_{ij}F_{\mu\gamma}^iF_{\nu\delta}^j-\frac{1}{4}g^{\mu\nu}a_\mu a_\nu+\frac{n}{2t})
\end{equation}
we have
\begin{align*}
	&\frac{d}{dt}\W_+(G_{ij},A^i_\mu,g_{\mu\nu},a_\mu,b_i,f,t)\\
	=&\frac{t}{2} \int_B|g^{\mu\nu}G_{ij;\mu\nu}-g^{\mu\nu}G^{kl}G_{ik,\mu}G_{jl,\nu}-\frac{1}{2}g^{\mu\gamma}g^{\nu\delta}G_{ik}G_{jl}F_{\mu\nu}^kF_{\gamma\delta}^l-g^{\mu\nu}G_{ij,\mu}f_{,\nu}+\frac{1}{2}b_ib_j|^2(4\pi t)^{-n/2}e^{-f}\, d\vol_B\\
	&+t\int_B|g^{\gamma \delta}F_{\mu\gamma;\delta}^i+g^{\gamma\delta}G^{ij}G_{jk,\gamma}F_{\mu\delta}^k-g^{\gamma\delta}f_{,\gamma}F_{\mu\delta}^i-\frac{1}{2}b_ka_\mu G^{ik}|^2(4\pi t)^{-n/2}e^{-f}\, d\vol_B\\
	&+2t\int_B|R_{\mu\nu}-\frac{1}{4}a_\mu a_\nu-\frac{1}{4}G^{ij}G_{jk,\mu}G^{kl}G_{li,\nu}-\frac{1}{2}g^{\gamma\delta}G_{ij}F_{\mu\gamma}^iF_{\nu\delta}^j+f_{;\mu\nu}+\frac{1}{2t}g_{\mu\nu}|^2(4\pi t)^{-n/2}e^{-f}\, d\vol_B\\
	&+\frac{t}{2}\int_B|d_g^*a+i_{\nabla f}a|^2(4\pi t)^{-n/2}e^{-f}\, d\vol_B+\frac{1}{4}\int_Bg^{\mu\gamma}g^{\nu\delta}G_{ij}F_{\mu\nu}^iF_{\gamma\delta}^j(4\pi t)^{-n/2}e^{-f}\, d\vol_B\\
	&-\frac{1}{4}\int _BG^{ij}b_ib_j(4\pi t)^{-n/2}e^{-f}\, d\vol_B.
    \end{align*}
\end{theorem}
\begin{proof}
We denote the time derivatives by $\dot{G}=\frac{d}{dt}G$, $\dot{A}=\frac{d}{dt}A$, $\dot{g}=\frac{d}{dt}g$, $\dot{a}=\frac{d}{dt}a$, and $\dot{f}=\frac{d}{dt}f$, and we write $u=(4\pi t)^{-n/2}e^{-f}$. Differentiating with respect to $t$ and integrating by parts to eliminate derivatives on the time-differentiated terms as in \cite[Proposition 5.3]{kleiner2008notes} and \cite[Proposition 4.64]{lott2010dimensional}, we obtain
\begin{equation}\label{e:compute.d/dt.W+}
    \begin{aligned}
	&\frac{d}{dt}\W_+(G_{ij},A_\mu^i,g_{\mu\nu},a_\mu,b_i,f,t)\\
	=&\frac{t}{2}\int_B \dot G_{kl}G^{ik}G^{jl}\left(g^{\mu\nu}G_{ij;\mu\nu}-g^{\mu\nu} G^{pq}G_{ip,\mu}G_{jq,\nu}-\frac{1}{2}g^{\mu\gamma}g^{\nu\delta}G_{ip}G_{jq}F_{\mu\nu}^pF_{\gamma\delta}^q\right.\\
    & \quad\left.-g^{\mu\nu}G_{ij,\mu}f_{,\nu}+\frac{1}{2}b_ib_j\right)u\, d\vol_B-\frac{t}{2}\int_B (\dot a_\nu+b_j\dot A^j_\nu) a_\mu g^{\mu\nu} u\, d\vol_B   \\
    &+t\int_B \dot A_\nu^j g^{\mu\nu}G_{ij}(-g^{\gamma\delta}F_{\mu\gamma;\delta}^i-g^{\gamma\delta}G^{im}G_{mk,\delta}F_{\mu\gamma}^k+g^{\gamma\delta}f_{,\delta}F_{\mu\gamma}^i+\frac{1}{2}a_\mu b_kG^{ik} )u\, d\vol_B\\
    &+t\int_B \dot g^{\mu\nu}\left(R_{\mu\nu}-\frac{1}{4}a_\mu a_\nu-\frac{1}{4}G_{jk,\mu}G^{ij}G_{il,\nu}G^{kl}-\frac{1}{2}g^{\gamma\delta}G_{ij}F_{\mu\gamma}^iF_{\nu\delta}^j+f_{;\mu\nu}\right)u\, d\vol_B\\
    &+\int_B(\frac{1}{2}g^{\mu\nu}\dot g_{\mu\nu}-\dot f-\frac{n}{2t})\left[t\left(2\Delta f-|\nabla f|^2+R-\frac{1}{4}g^{\mu\nu}a_\mu a_\nu-\frac{1}{4}G^{ij}b_ib_j\right.\right.\\
	& \quad\left.\left.-\frac{1}{4}g^{\mu\nu}G^{ij}G_{jk,\mu}G^{kl}G_{li,\nu}-\frac{1}{4}g^{\mu\gamma}g^{\nu\delta}G_{ij}F_{\mu\nu}^iF_{\gamma\delta}^j\right)-f+n\right]u\, d\vol_B\\
    &+\int_B \left(|\nabla f|^2+R-\frac{1}{4}g^{\mu\nu}a_\mu a_\nu-\frac{1}{4}G^{ij}b_ib_j-\frac{1}{4}g^{\mu\nu}G^{ij}G_{jk,\mu}G^{kl}G_{li,\nu}\right.\\
    & \quad \left.-\frac{1}{4}g^{\mu\gamma}g^{\nu\delta}G_{ij}F_{\mu\nu}^iF_{\gamma\delta}^j\right)u\, d\vol_B-\int_B \dot f u\, d\vol_B
    \end{aligned}
\end{equation}
Modifying the equations \eqref{gauged.invariant.alpha.ricci.flow} and \eqref{e:f_equation} by a Lie derivative with respect to $-\nabla f$,  we get
\begin{equation}\label{gauged.invariant.alpha.ricci.flow.with.f}
    \begin{aligned}
	\frac{\partial G_{ij}}{\partial t}=&g^{\mu\nu}G_{ij;\mu\nu}-g^{\mu\nu} G^{kl}G_{ik,\mu}G_{jl,\nu}-\frac{1}{2}g^{\mu\gamma}g^{\nu\delta}G_{ik}G_{jl}F_{\mu\nu}^kF_{\gamma\delta}^l+\frac{1}{2}b_ib_j-g^{\mu\nu}G_{ij,\mu}f_{,\nu}\\
	\frac{\partial A_\mu^i}{\partial t}=&-g^{\gamma\delta}F_{\mu\gamma;\delta}^i-g^{\gamma\delta}G^{ij}G_{jk,\gamma}F_{\mu\delta}^k+\frac{1}{2}a_\mu b_jG^{ij}+g^{\gamma\delta}f_{,\delta}F_{\mu\gamma}^i\\
	\frac{\partial g_{\mu\nu}}{\partial t}=&-2R_{\mu\nu}+\frac{1}{2}G^{ij}G_{jk,\mu}G^{kl}G_{il,\nu}+g^{\gamma\delta}G_{ij}F_{\mu\gamma}^iF_{\nu\delta}^j+\frac{1}{2}a_\mu a_\nu-2f_{;\mu\nu}\\
	\frac{\partial a_\mu}{\partial t}=&-(d(d_{g}^*a+i_{\nabla f}a))_{\mu}+b_ig^{\gamma\delta}F_{\mu\gamma;\delta}^i+b_ig^{\gamma\delta}G^{ij}G_{ik,\gamma}F_{\mu\delta}^k-\frac{1}{2}a_\mu b_ib_jG^{ij}-b_{i}g^{\gamma\delta}f_{,\delta}F_{\mu\gamma}^i\\
	\frac{\partial b_i}{\partial t}=&0\\
	\frac{\partial f}{\partial t}=&-\Delta f-(R-\frac{1}{4}g^{\mu\nu}G^{ij}G_{jk,\mu}G^{kl}G_{il,\nu}-\frac{1}{2}g^{\mu\nu}g^{\gamma\delta}G_{ij}F_{\mu\gamma}^iF_{\nu\delta}^j-\frac{1}{4}g^{\mu\nu}a_\mu a_\nu+\frac{n}{2t}).
\end{aligned}
\end{equation}
Since \eqref{gauged.invariant.alpha.ricci.flow} and \eqref{e:f_equation} differ from \eqref{gauged.invariant.alpha.ricci.flow.with.f} only by a Lie derivative term, and since $\W_+$ is invariant under pullback by diffeomorphisms of $B$, the evolution of $\mathcal{W}_+$ is the same. Notice that from \eqref{gauged.invariant.alpha.ricci.flow.with.f}, we have $\dot a+b_i\dot A^i=-d(d_{g}^*a+i_{\nabla f}a)$. Substituting \eqref{gauged.invariant.alpha.ricci.flow.with.f} into \eqref{e:compute.d/dt.W+} and integrating by parts yields the desired monotonicity formula.
\end{proof}
\begin{remark}
    The terms on the right-hand side of Theorem~\ref{t:evolution.of.w+.functional} are nonnegative, except for the last one, which depends on the fiber component of $\alpha$. We will show in Section~\ref{s:no.fiber} that for blow-down limits of an $\alpha$-Ricci flow on $\mathbb{T}^2$ when $[\alpha]$ is rational, the fiber component vanishes in the limit.
\end{remark}

\begin{definition}
    We say that a locally $\R^N$-invariant $\alpha$-Ricci flow $(\bar g(t), \bar\alpha(t))$ \textit{has no fiber component} if
    $b_i=0$ for all $i=1,\dots,N$, where $b_i$ are defined in \eqref{invariant.flow}. Note that by Corollary~\ref{fiber.constant.in.spacetime}, each $b_i$ is a constant in spacetime.
\end{definition}

\begin{corollary}[Conditional monotonicity]\label{c:conditional.monotonicity}
    Let $(M, \bar g(t), \bar\alpha(t))$ be a locally $\R^N$-invariant $\alpha$-Ricci flow defined on the time interval $(0,\infty)$. Suppose $\bar \alpha(t)$ has no fiber component and that $f$ is a solution to \eqref{e:f_equation} on $(0,\infty)$. Then $\frac{d}{dt}\W_+\ge0$.

    If, in addition, $\W_+$ is constant, then the following properties hold:
    
    (i) $F_{\mu\nu}^i=0$.
	
	(ii) $\det(G_{ij})$ is constant in spacetime.
    
    (iii) $G_{ij}$, $A_\mu^i$, $a_\mu$ are constant in time and $g_{\mu\nu}(t)=tg_{\mu\nu}(1)$.
	
	(iv) The following equations hold
    \begin{align*}
		g^{\mu\nu}G_{ij;\mu\nu}-g^{\mu\nu}G^{kl}G_{ik,\mu}G_{jl,\nu}=&0,\\
		R_{\mu\nu}-\frac{1}{4}a_\mu a_\nu-\frac{1}{4}G^{ij}G_{jk,\mu}G^{kl}G_{il,\nu}+\frac{1}{2t}g_{\mu\nu}=&0.
	\end{align*}
    
    (v) The $1$-form $a$ is harmonic on $B$:
\[
da=0,\qquad d_g^*a=0.
\]
\end{corollary}

\begin{proof}
    Since $\bar \alpha(t)$ has no fiber component, $\frac{d}{dt}\W_+\ge0$ follows from Theorem~\ref{t:evolution.of.w+.functional}.
    
    If $\W_+$ is constant, then by Theorem~\ref{t:evolution.of.w+.functional} we have 
    \begin{align}
        &F_{\mu\nu}^i=0,\label{W+.F.mu.nu.=0}\\
	    &g^{\mu\nu}G_{ij;\mu\nu}-g^{\mu\nu}G^{kl}G_{ik,\mu}G_{jl,\nu}-g^{\mu\nu}G_{ij,\mu}f_{,\nu}=0,\label{W+.Gij=0}\\
        &R_{\mu\nu}-\frac{1}{4}a_\mu a_\nu-\frac{1}{4}G^{ij}G_{jk,\mu}G^{kl}G_{li,\nu}-\frac{1}{2}g^{\gamma\delta}G_{ij}F_{\mu\gamma}^iF_{\nu\delta}^j+f_{;\mu\nu}+\frac{1}{2t}g_{\mu\nu}=0,\label{W+g.mu.nu.=0}\\
        &d_g^*a+i_{\nabla f}a=0.\label{W+d^*_fa=0}
	\end{align}
	Multiplying \eqref{W+.Gij=0} by $G^{ij}$ and summing over indices gives 
	\[\Delta\log \det (G_{ij})-\langle \nabla \log \det (G_{ij}),\nabla f\rangle=0.\]
	Integrating the left-hand side against $\log(\det (G_{ij}) )e^{-f}\, d\vol_B$ and applying the integration by parts formula, we get $\int_B|\nabla \log\det(G_{ij})|^2 e^{-f}\, d\vol_B=0$. Thus
    \begin{equation}\label{nabla.log.det.G=0}
	    G^{ij}G_{ij,\mu}=(\nabla\log\det (G_{ij}))_\mu=0.
	\end{equation} Hence $\det (G_{ij})$ is a constant in space. In particular
	\begin{equation}\label{Delta.log.det.G=0}
	    G^{ij}(g^{\mu\nu}G_{ij;\mu\nu}-g^{\mu\nu}G_{ik,\mu}G^{kl}G_{lj,\nu})=\Delta \log\det(G_{ij})=0.
	\end{equation}
    Plugging \eqref{W+.F.mu.nu.=0}, \eqref{W+d^*_fa=0}, \eqref{nabla.log.det.G=0} and \eqref{Delta.log.det.G=0} into \eqref{e:invariant.alpha.ricci.flow}, we have
    \begin{equation}\label{evolve.by.L.nabla.f}
        \frac{\partial G_{ij}}{\partial t}=g^{\mu\nu}G_{ij,\mu}f_{,\nu},\quad\frac{\partial A^i_\mu}{\partial t}=0,\quad \frac{\partial g_{\mu\nu}}{\partial t}=2f_{;\mu\nu}+\frac{1}{t}g_{\mu\nu},\quad\frac{\partial a_\mu}{\partial t}=(di_{\nabla f}a)_\mu.
    \end{equation}
	In particular
    $\frac{\partial}{\partial t}\log\det G_{ij}=G^{ij}\frac{\partial G_{ij}}{\partial t}=0$ by \eqref{nabla.log.det.G=0}. Hence $\det (G_{ij})$ is constant in spacetime. Moreover, from (\ref{evolve.by.L.nabla.f}) the flow evolves by a diffeomorphism generated by $\nabla f$ and a scaling of $g_{\mu\nu}$. Thus we have
	\[\frac{\partial \bar R^{\bar \alpha}}{\partial t}=\langle\nabla f,\nabla \bar R^{\bar \alpha}\rangle-\frac{\bar R^{\bar \alpha}}{t}.\]
Since $F=0$, $b=0$, and $\det G$ is spatially constant,
$\bar R_{i\mu}=0$,
$\bar\Delta\bar R^{\bar\alpha}
=\Delta\bar R^{\bar\alpha}$, and
$d_{\bar g}^*\bar\alpha=d_g^*a$.
The evolution equation consequently gives
\[
\frac{\partial\bar R^{\bar\alpha}}{\partial t}
=
\Delta\bar R^{\bar\alpha}
+2|\bar R_{ij}|^2
+2\left|\bar R_{\mu\nu}-\frac14a_\mu a_\nu\right|^2
+\frac12(d_g^*a)^2.
\]
Combining this with
\[
\frac{\partial\bar R^{\bar\alpha}}{\partial t}
=
\langle\nabla f,\nabla\bar R^{\bar\alpha}\rangle
-\frac{\bar R^{\bar\alpha}}{t}
\]
and completing the square yields
\begin{equation}\label{equation.for.R.alpha+n/2t}
\begin{aligned}
&\Delta\left(\bar R^{\bar\alpha}+\frac{n}{2t}\right)
+2|\bar R_{ij}|^2
+2\left|\bar R_{\mu\nu}-\frac14a_\mu a_\nu
+\frac{1}{2t}g_{\mu\nu}\right|^2\\
&\qquad
+\frac12(d_g^*a)^2
-\frac1t\left(\bar R^{\bar\alpha}+\frac{n}{2t}\right)
=
\left\langle\nabla f,
\nabla\left(\bar R^{\bar\alpha}+\frac{n}{2t}\right)
\right\rangle.
\end{aligned}
\end{equation}
At a negative minimum of
$\bar R^{\bar\alpha}+\frac{n}{2t}$, the left-hand side
would be strictly positive while the right-hand side vanishes.
Thus we must have $\min_B \bar R^{\bar \alpha}+\frac{n}{2t}\ge 0$.
	
    On the other hand, plugging (\ref{W+.F.mu.nu.=0}) and (\ref{Delta.log.det.G=0}) into \cite[(4.6),(4.7)]{lott2007long}, we have 
    \begin{align*}
        \bar R_{ij}=&-\frac{1}{2}g^{\mu\nu}G_{ij;\mu\nu}+\frac{1}{2}g^{\mu\nu}G^{kl}G_{ik,\mu}G_{jl,\nu},\quad\bar R_{i\mu}=0,\\
	\bar R_{\mu\nu}=&R_{\mu\nu}-\frac{1}{4}G^{ij}G_{jk,\mu}G^{kl}G_{il,\nu},\\
	\bar R=&R-\frac{1}{4}g^{\mu\nu}G^{ij}G_{jk,\mu}G^{kl}G_{il,\nu}.
    \end{align*}
	Tracing (\ref{W+g.mu.nu.=0}) by $g^{\mu\nu}$, we get 
    \begin{equation}\label{R.mu.f=0}
	    \bar R^{\bar \alpha}+\Delta f+\frac{n}{2t}=R-\frac{1}{4}|a|^2-\frac{1}{4}g^{\mu\nu}G^{ij}G_{jk,\mu}G^{kl}G_{il,\nu}+\Delta f+\frac{n}{2t}=0.
	\end{equation}
	Then by the divergence theorem, we have \[\int_B (\bar R^{\bar \alpha} +\frac{n}{2t})\, d\vol_B=-\int_B\Delta f\, d\vol_B=0.\]
	Since $\min_B\bar R^{\bar \alpha}+\frac{n}{2t}\ge0$, we have $\bar R^{\bar \alpha}+\frac{n}{2t}\equiv0$. Hence, by (\ref{R.mu.f=0}) and the fact that $B$ is closed, we conclude that $f$ is a constant. Due to (\ref{evolve.by.L.nabla.f}), we have (iii). By \eqref{W+d^*_fa=0}, we have (v). By (\ref{equation.for.R.alpha+n/2t}), we have $\bar R_{\mu\nu}-\frac{1}{4}a_\mu a_\nu+\frac{1}{2t}g_{\mu\nu}=\bar R_{ij}=0$ which gives the equations in (iv).
\end{proof}

\subsection{Asymptotic limits via \texorpdfstring{$\mathcal{W}_+$}{W+}-monotonicity}
In this subsection, we apply the conditional monotonicity of the $\mathcal{W}_+$-functional to study the asymptotic limits of locally $\R^N$-invariant $\alpha$-Ricci flows. We first introduce the following natural concept of convergence of locally $\R^N$-invariant $\alpha$-Ricci flow.
\begin{definition}\label{d:convergence.invariant.alpha.RF}
    Let $\{(M_i, g_i(t),\alpha_i(t), p_i)\}_{i=1}^\infty$ be a sequence of locally $\R^N$-invariant $\alpha$-Ricci flow solutions defined on the time interval $t \in (0,\infty)$. We say that
\[ \lim_{i\to\infty} (M_i, g_i(t),\alpha_i(t), p_i)= (M_\infty, g_\infty(t),\alpha_\infty(t), p_\infty)
\]
if there are:

\begin{enumerate}[label=(\roman*)]
    \item A sequence of open subsets $\{U_j\}_{j=1}^\infty$ of $B_\infty$, each containing $p_\infty$, such that any compact subset of $B_\infty$ eventually lies in all $U_j$, and
    
    \item Open subsets $V_{i,j} \subset B_i$ containing $p_i$ and isomorphisms
    \[
    \phi_{i,j} : \pi_\infty^{-1}(U_j) \to \pi_i^{-1}(V_{i,j})
    \]
    sending $\pi_\infty^{-1}(p_\infty)$ to $\pi_i^{-1}(p_i)$, such that
    
    \item For all $j$,
    \[
    \lim_{i \to \infty} \phi_{i,j}^* g_i(t) = g_\infty(t),\lim_{i \to \infty} \phi_{i,j}^* \alpha_i(t) = \alpha_\infty(t)
    \]
    smoothly on $\pi_\infty^{-1}(U_j) \times \left[ j^{-1},\,  j\right]$.
\end{enumerate}

If $B_\infty$ is compact, then we can remove the reference to basepoints.

\end{definition}

\begin{corollary}\label{alpha.expander}
    Let $(\bar g(t),\bar \alpha(t))$ be a locally $\R^N$-invariant $\alpha$-Ricci flow with a compact base $B$ defined on the time interval $(0,\infty)$ and $\lim_{i\to\infty} (\bar g_{s_i}(t),\bar\alpha_{s_i} (t))=(\hat g(t),\hat \alpha(t))$ for some sequence $s_i\to\infty$ or $s_i\to0$. If $\bar\alpha(t)$ has no fiber component and $(\hat g(t),\hat \alpha(t))$ is locally $\R^N$-invariant with a compact base $\hat B$ defined for $(0,\infty)$, then, writing $(\hat g(t),\hat \alpha(t))=(\hat G_{ij}( t),\hat A_\mu^i( t),\hat g_{\mu\nu}( t),\hat a_\mu( t),0)$, we have 
	
	(i) $\hat F_{\mu\nu}^i=0$,
	
	(ii) $\det(\hat G_{ij})$ is constant in spacetime,
    
    (iii) $\hat G_{ij}$, $\hat A_\mu^i$, $\hat a_\mu$ are constant in time and $\hat g_{\mu\nu}(t)=t\hat g_{\mu\nu}(1)$,
	
	(iv) The following equations hold
    \begin{align}
		\hat g^{\mu\nu}\hat G_{ij;\mu\nu}-\hat g^{\mu\nu}\hat G^{kl}\hat G_{ik,\mu}\hat G_{jl,\nu}=&0,\label{alpha.expander.G.ij}\\
		\hat R_{\mu\nu}-\frac{1}{4}\hat a_\mu \hat a_\nu-\frac{1}{4}\hat G^{ij}\hat G_{jk,\mu}\hat G^{kl}\hat G_{il,\nu}+\frac{1}{2t}\hat g_{\mu\nu}=&0.\label{alpha.expander.R.mu.nu}
	\end{align}
    
    (v) The $1$-form $\hat a$ is harmonic on $\hat B$:
\[
d\hat a=0,\qquad d_{\hat g}^*\hat a=0.
\]
\end{corollary}

\begin{proof}
By Corollary~\ref{c:conditional.monotonicity}, it suffices to construct a solution to \eqref{e:f_equation} along $(\hat g(t), \hat \alpha(t))$ on $(0,\infty)$ for which $\W_+$ remains constant. We follow similar ideas in \cite[Proposition 4.39]{lott2010dimensional}.

We write $u=(4\pi t)^{-n/2}e^{-f}$. Then \eqref{e:f_equation} is equivalent to 
\begin{equation}\label{e:u_equation}
    \left(\frac{\partial }{\partial t}+\Delta-R+\frac{1}{4}g^{\mu\nu}G^{ij}G_{jk,\mu}G^{kl}G_{il,\nu}+\frac{1}{2}g^{\mu\nu}g^{\gamma\delta}G_{ij}F_{\mu\gamma}^iF_{\nu\delta}^j+\frac{1}{4}g^{\mu\nu}a_\mu a_\nu\right)u=0.
\end{equation}
Note that if $u$ is a solution to \eqref{e:u_equation}, then $\int_B u  \,d\vol_B$ is constant in $t$. Choose $\delta_j\downarrow0$ and $t_j\uparrow\infty$ with $\delta_j<t_j$, so that $[\delta_j,t_j]$ exhausts $(0,\infty)$. Let $\tilde{u}_j(t)$ be a solution to \eqref{e:u_equation} on the interval $[\delta_j,t_j]$ with terminal condition
\[
\tilde{u}_j(t_j) = \frac{1}{\operatorname{vol}(B, g_{\mu\nu}(t_j))}.
\]
For any $0<\delta<T$, we claim that a subsequence of the $\tilde{u}_j$'s converges smoothly on the time interval $[\delta,T]$.

To see this, at time $T+1$ we know that if $\delta_j<\delta$ and $t_j\ge T+1$ then $\tilde{u}_j(T+1) \ge 0$ and
\[
\int_B \tilde{u}_j(T+1)\, d\vol_B = 1.
\]
Solving the conjugate heat equation with terminal data at time $T+1$, and restricting the solution to the time interval $[\delta,T]$, gives a smoothing operator from the space of terminal data
\[
\left\{ \tilde{u} \in L^1(B) : \tilde{u} \ge 0,\ \int_B \tilde{u}\, d\vol_B(T+1) = 1 \right\}
\]
to $C^\infty([\delta,T]\times B)$ by the heat kernel estimate \cite{chau2011pseudolocality}. Thus we have the derivative bounds needed to extract a subsequence of the $\tilde{u}_j$'s that converges smoothly on $[\delta,T]$. By a diagonal argument over $[1/k,k]$, $k\ge2$, we can extract a subsequence of the $\tilde{u}_j$'s that converges smoothly on compact subsets of $(0,\infty)$ to a nonzero solution $\bar{u}(t)$ of \eqref{e:u_equation}, defined for $t \in (0,\infty)$.

One can show, as in \cite[Proof of Proposition 7.5]{kleiner2008notes}, that $\bar{u}(t) > 0$. If $\bar f(t)$ is given by
\[
\bar u(t) = (4\pi t)^{-n/2} e^{-\bar{f}(t)},
\]
then $\W_+(\bar G_{ij},\bar A^i_\mu,\bar g_{\mu\nu},\bar a_\mu,0,\bar f,t)$ is non-decreasing in $t$. We write
\[
\W_{+,\infty} = \lim_{t \to \infty} \W_+(\bar G_{ij}(t), \bar A^i_{\mu}(t), \bar g_{\mu\nu}(t),\bar a_\mu(t),0, \bar f(t),t),
\]
if we are in the case $s_i\to\infty$ and 
\[
\W_{+,\infty} = \lim_{t \to0} \W_+(\bar G_{ij}(t), \bar A^i_{\mu}(t), \bar g_{\mu\nu}(t),\bar a_\mu(t),0, \bar f(t),t),
\]
if we are in the case $s_i\to0$.
Note that $\W_{+,\infty}$ is possibly $\pm\infty$ for the moment.

Next, we put 
$(\bar g_{s_i}(t),\bar \alpha_{s_i}(t))=(s_i^{-1}\bar G_{ij}(s_i t),\bar A_\mu^i(s_i t),s_i^{-1}\bar g_{\mu\nu}(s_i t),\bar a_\mu(s_i t),0)$ and $u_i(t) = s_i^{n/2}\bar{u}(s_it)$. By assumption, $\lim_{i \to \infty} (\bar g_{s_i}(t),\bar\alpha_{s_i}(t)) = (\hat g(t),\hat\alpha(t))$ in the sense of Definition~\ref{d:convergence.invariant.alpha.RF}. Then by the same smoothing argument as above, there is a subsequence of $\{u_i(t)\}_{i=1}^\infty$ that converges smoothly on compact subsets of $(0,\infty)$ to a solution $\hat u(t)$ on $\hat B$. Define $\hat f(t)$ by
\[
\hat u(t) = (4\pi t)^{-n/2}e^{-\hat f(t)}.
\]
Then after passing to a subsequence of $s_i$,
\begin{equation}
\begin{aligned}
&\W_+(\hat G_{ij}(t), \hat A^i_{\mu}(t), \hat g_{\mu\nu}(t),\hat a_\mu(t),0, \hat f(t),t)\\
=& \lim_{i \to \infty} \W_+(s_i^{-1}\bar G_{ij}(s_i t),\bar A_\mu^i(s_i t),s_i^{-1}\bar g_{\mu\nu}(s_i t),\bar a_\mu(s_i t),0,\bar f(s_i t),t) \\
=& \lim_{i \to \infty} \W_+(\bar G_{ij}(s_i t),\bar A_\mu^i(s_i t),\bar g_{\mu\nu}(s_i t),\bar a_\mu(s_i t),0,\bar f(s_i t),s_it) \\
=& \W_{+,\infty}
\end{aligned}
\end{equation}

This shows that $\W_{+,\infty}< \infty$ and that
\[
\W_+(\hat G_{ij}(t), \hat A^i_{\mu}(t), \hat g_{\mu\nu}(t),\hat a_\mu(t),0, \hat f(t),t)
\]
is constant in $t$. This finishes the proof.
\end{proof}

When the base is $S^1$, the argument of Corollary~\ref{alpha.expander} applies under the following weaker convergence assumption.
\begin{corollary}\label{1d.alpha.expander}
Let $(\bar g(t),\bar\alpha(t))$ be a locally $\R^N$-invariant $\alpha$-Ricci flow on $S^1$, defined for $t\in(0,\infty)$, with trivial linear holonomy and no fiber component. Let $s_i\to0$ or $s_i\to\infty$. After time-independent identifications of the bases and constant linear changes of the fiber coordinates, write
\[
\begin{aligned}
g_i(t):=s_i^{-1}\bar g(s_it)
&=u_i(t)(dx^0)^2+(G_i)_{jk}(t)(dx^j+A_i^j(t))(dx^k+A_i^k(t)),\\
\alpha_i(t):=\bar\alpha(s_it)&=a_i(t)\,dx^0.
\end{aligned}
\]
Suppose that $\hat G>0$, $\hat u>0$, and
\[
(G_i,u_i,a_i)\longrightarrow(\hat G,\hat u,\hat a)
\quad\text{in }C^\infty(S^1\times[\delta,T]),
\qquad 0<\delta<T<\infty.
\]
Then, choosing $x^0$ to be arclength at time one, a constant linear change of the fiber coordinates, and an orientation of $S^1$, we have
\[
\hat G_{ij}=\delta_{ij},\qquad \hat u(t)=t,\qquad \hat a=\sqrt2.
\]
In particular, for any time-independent flat connection $\hat A$, the metric and $1$-form determined by $(\hat G,\hat A,\hat u,\hat a,0)$ lift to a flow isometric to $(g^*(t),\alpha^*(t))$ on the universal cover.
\end{corollary}
\begin{proof}
Since $\dim B=1$ and $b=0$, we have $F=0$, and the equations for $G,u,a$, equation~\eqref{e:f_equation}, and $\W_+$ are independent of $A$. Moreover, $\W_+$ is unchanged by constant linear changes of the fiber coordinates. Thus the conjugate heat construction and passage to the limit in the proof of Corollary~\ref{alpha.expander} apply to the assumed convergence of $(G_i,u_i,a_i)$: the limiting $\W_+$ is constant, and we obtain~\eqref{alpha.expander.G.ij},~\eqref{alpha.expander.R.mu.nu}, $\partial_t\hat G=\partial_t\hat a=0$, and $\hat u(t)=t\hat u(1)$. Choose $x^0$ to be arclength at time one, so that $\hat u(t)=t$. By triviality of the linear holonomy, $\hat G_{ij}$ and $\hat a$ are periodic functions of $x^0$.
    Then equation (\ref{alpha.expander.G.ij}) becomes
    \[\hat G_{ij}''-\hat G^{kl}\hat G_{ik}'\hat G_{jl}'=0\] where $'=\partial_{0}$.
    In other words, $(\hat G^{-1}\hat G')'=0$. Thus $\hat G^{-1}\hat G'=X$ for some constant matrix $X$. Integrating $\hat G'=\hat GX$, we get $\hat G(x^0)=\hat G(0)\exp(x^0X)$. We can choose coordinates such that $\hat G(0)=I$. Then $\hat G'(0)=X$ is symmetric and in particular, $X$ is diagonalizable and has real eigenvalues. Since $\hat G$ is periodic in $x^0$, we must have $X=0$ and hence $\hat G=I$.
    Equation (\ref{alpha.expander.R.mu.nu}) now reads
	\[
		-\frac{1}{4}\hat a^2+\frac{1}{2}=0.
	\]
	Hence adjusting the orientation of $S^1$ if necessary, we have $\hat a=\sqrt 2$.
Finally, on the universal cover, write $\hat A^i=d\psi^i$; the time-independent change $y^i=x^i+\psi^i$ removes $\hat A$ and gives the asserted isometry.
\end{proof}

\begin{remark}\label{r:connection.convergence}
Corollary~\ref{1d.alpha.expander} does not require $A_i\to\hat A$. Even though $F_i=0$, normalization of the fiber coordinates may make $A_i$ unbounded, so convergence of $(G_i,u_i,a_i)$ need not imply convergence in the sense of Definition~\ref{d:convergence.invariant.alpha.RF}. In the groupoid setting, different flat connections with the same $G$ and base metric give equivalent Riemannian groupoids by~\cite[Lemma~5.11]{lott2010dimensional}, although the global bundle metrics need not be isometric. To avoid groupoid language, we use convergence of $(G_i,u_i,a_i)$ directly; since $F=b=0$, these are precisely the coefficients needed in the entropy argument.
\end{remark}
Now, we show that if a blow-down limit has no fiber component, then it must be the expander given in Example~\ref{example.noncompact invariant expander}. We follow the argument of Lott \cite[Propositions 6.8 and 6.11]{lott2010dimensional}. One important ingredient is the stability of $(g^*(t),\alpha^*(t))$ among locally $\R^N$-invariant $\alpha$-Ricci flows on $S^1\times\R^N$ with trivial linear holonomy and no fiber component which is proved in Appendix~\ref{s:dynamic stability}; see Theorem~\ref{stability} and Corollary~\ref{neighborhood.around.expander}. We first introduce the definitions needed to apply these results.

Let $\bar g=u(dx^0)^2+G_{ij}(dx^i+A^i)(dx^j+A^j)$ and $\bar\alpha= adx^0$ be a locally $\R^N$-invariant metric and closed 1-form over $S^1$. We fix orientation of $S^1$ such that $\int_{S^1}\bar\alpha=\sqrt2\ell$ , where $\ell>0$. Write $S^1=\R/\ell\Z$. Consider the tuple \[\mathbf{\bar g}=( G, u,\eta)\text{ where }\bar\alpha=(\sqrt2+\partial_0\eta)\,dx^0.\] Here $\bar\alpha$ determines $\eta$ up to an additive constant.  The tuple corresponding to $(g^*(t),\alpha^*(t))$ at time one is $\mathbf g^*=(I,1,0)$.

\begin{definition}\label{tuple.equivalence}
The equivalence class $[\mathbf{\bar g}]$ consists of the tuples
\[
\left(
S^T( G\circ\phi)S,\,
( u\circ\phi)(\phi')^2,\,
\eta\circ\phi+\sqrt2(\phi-x^0)-c
\right),
\]
where $S\in\operatorname{GL}(N,\R)$, $c\in\R$, and $\phi$ is a smooth orientation-preserving diffeomorphism of $S^1$. In the expression $\phi-x^0$, we use a degree-one lift of $\phi$ to $\R$, denoted by the same symbol, satisfying $\phi(x^0+\ell)=\phi(x^0)+\ell$. Thus $\phi-x^0$ is $\ell$-periodic and defines a real-valued function on $S^1$. Changing the lift adds an integer multiple of $\ell$, which is absorbed into the constant $c$. The change in $\eta$ is obtained by
\[
\phi^*\bar\alpha
=\sqrt 2d\phi+d(\eta\circ\phi)=\left(\sqrt2+\partial_0\bigl(\eta\circ\phi
+\sqrt2(\phi-x^0)-c\bigr)\right)dx^0.
\]
For a finite-dimensional vector space $V$ and nonintegral $r>0$, let $h^r(S^1;V)$ be the closure of $C^\infty(S^1;V)$ in $C^r(S^1;V)$, with the inherited norm. Set $\mathbb V=\operatorname{Sym}_N\times\R\times\R$, where $\operatorname{Sym}_N$ is the space of real symmetric $N\times N$ matrices. For nonintegral $r>1$, define
\begin{equation}\label{eq:distance.to.expander.class}
 d_r([\mathbf{\bar g}],[\mathbf g^*])
 =\inf_{\mathbf v\in[\mathbf{\bar g}]}
 \|\mathbf v-\mathbf g^*\|_{h^r(S^1;\mathbb V)}.
\end{equation}
Fix $0<\rho<\theta<1$. For $\delta,\epsilon>0$, define
\[
\begin{aligned}
\mathcal V_\delta
&=\{[\mathbf{\bar g}]:d_{1+\theta}([\mathbf{\bar g}],[\mathbf g^*])<\delta\},\\
\mathcal U_\epsilon
&=\{[\mathbf{\bar g}]:d_{1+\rho}([\mathbf{\bar g}],[\mathbf g^*])<\epsilon\}.
\end{aligned}
\]
We have
\[
d_r([\mathbf{\bar g}],[\mathbf g^*])=0
\quad\Longleftrightarrow\quad
[\mathbf{\bar g}]=[\mathbf g^*].
\]
Indeed, suppose that representatives corresponding to $(S_j,\phi_j,c_j)$ converge to $(I,1,0)$ in $h^r$. Since $G>0$ and $u>0$, this convergence implies that $S_j$ and $S_j^{-1}$ are bounded and $0<c\le\phi_j'\le C$ for constants independent of $j$. After normalizing the degree-one lifts by $\phi_j(0)\in[0,\ell]$ and passing to a subsequence, $S_j\to S\in\operatorname{GL}(N,\R)$ and $\phi_j\to\phi$ in $C^1$, where $\phi$ is a smooth orientation-preserving circle diffeomorphism. Since each $\phi_j$ is onto, convergence of the first component gives $S^T G S=I$ at every point. Write $ a=\sqrt2+\partial_0\eta$. Convergence of the second and third components gives $ a>0$ and
\[
\left(\frac{ a^2}{ u}\right)\circ\phi=2,
\qquad\text{hence}\qquad
 u=\frac12 a^2.
\]
Thus $y=x^0+\eta(x^0)/\sqrt2$ defines an orientation-preserving diffeomorphism of $S^1$, in which $\bar\alpha=\sqrt2\,dy$ and the base metric is $dy^2$. Together with the constant fiber change $S$, this gives the tuple $(I,1,0)$ after choosing the additive constant. The converse follows directly from~\eqref{eq:distance.to.expander.class}. In particular,
\[
\bigcap_{\epsilon>0}\mathcal U_\epsilon=\{[\mathbf g^*]\}.
\]
\end{definition}
    \begin{proposition}
        \label{p:base.S^1.implies.universal.cover.expander}
	Let $(g(t),\alpha(t))$ be a solution to the $\alpha$-Ricci flow on $M=\mathbb{T}^2$. Let $s_i\to\infty$  be a sequence such that $(M,g_{s_i}(1))\xrightarrow{GH} (S^1,d)$ and $(\tilde g_{s_i}(t),\tilde \alpha_{s_i}(t))\xrightarrow{C^\infty}(\tilde g_\infty(t),\tilde\alpha_\infty(t))$ with no fiber component. Then $(\tilde g_\infty(t),\tilde\alpha_\infty(t))=(g^*(t),\alpha^*(t))$.
    \end{proposition}
\begin{proof}
By Corollary~\ref{c:convergence_CFG}, the limit is $\R\times\Z$-invariant and descends to an $\R$-invariant flow $(\bar g(t),\bar\alpha(t))$ over $B=S^1$. Write its coefficients as $G,u,a$, with base metric $u(t)d\theta^2$ and $\bar\alpha(t)=a(t)d\theta$. By Theorems~\ref{diameter.bound},~\ref{alpha.norm.estimate}, and~\ref{type.iii.estimate},
\begin{equation}\label{type.3.estimates.blow.down}
\diam(B,u(t)d\theta^2)\le C\sqrt t,\qquad
|R(\bar g(t))|\le \frac Ct,\qquad
|\bar\alpha(t)|^2\le \frac2t.
\end{equation}
As in the proof of Corollary~\ref{alpha.expander}, construct a normalized solution of~\eqref{e:f_equation} on compact positive-time intervals and pass to a diagonal limit on $(0,\infty)$. Lemma~\ref{volume.ratio.lower.bound.by.W_+} and Corollary~\ref{c:conditional.monotonicity} give
\[
\frac12\log(4\pi)-\log\bigl(t^{-1/2}\vol(B,u(t)d\theta^2)\bigr)
\le\W_+(t)\le\W_+(1),\qquad 0<t\le1.
\]
Consequently, by~\eqref{type.3.estimates.blow.down},
\begin{equation}\label{eq:base.length.backward.bounds}
0<c\le t^{-1/2}\vol(B,u(t)d\theta^2)
=2t^{-1/2}\diam(B,u(t)d\theta^2)\le C,
\qquad 0<t\le1.
\end{equation}
In particular, $\W_+$ is bounded from below on $(0,1]$.

Let $t_i\to0$ and write $(g_i(t),\alpha_i(t))=(t_i^{-1}\bar g(t_it),\bar\alpha(t_it))$. By \eqref{eq:base.length.backward.bounds} we choose base arclength $\theta\in\R/\ell_i\Z$ and normalize the fiber coordinate so that
\[
g_i(1)=d\theta^2+G_i(\theta)(dv+A_i(\theta)d\theta)^2,
\qquad \alpha_i(1)=a_i(\theta)d\theta,
\qquad G_i(0)=1,\qquad c\le\ell_i\le C.
\]
These base identifications and fiber normalizations are fixed in time; in the following time-one estimates, we write $G_i=G_i(1)$ and $a_i=a_i(1)$.
We have
\begin{equation}\label{scalar.in.terms.of.Gi}
    R_i:=R(g_i(1))=
-(\log G_i)''-\frac12((\log G_i)')^2.
\end{equation}
By~\eqref{type.3.estimates.blow.down}, $|R(g_i(1))|\le C$. Evaluating \eqref{scalar.in.terms.of.Gi} at minimum and maximum of the periodic function $(\log G_i)'$ and integrating from $\theta=0$ gives
\[
| (\log G_i)' |\le2\sqrt C,
\qquad C_1^{-1}\le G_i\le C_1.
\]
The smoothing estimates in~\cite{guo2025alpha}, used in Corollary~\ref{c:convergence_CFG}, give bounds on all covariant derivatives of $R_i$ and $\alpha_i$ on bounded lifted balls. Locally choose the fiber coordinate so that $A_i=0$. Since $\theta$ is base arclength, $\partial_\theta$ is then a unit vector field with $\nabla_{\partial_\theta}\partial_\theta=0$. By \eqref{eq:base.length.backward.bounds}, we have
\[
|\partial_\theta^k R_i|\le C_k,\qquad
|a_i|=|\alpha_i|\le\sqrt2,\qquad
\partial_\theta^k a_i
=(\nabla^k\alpha_i)(\partial_\theta,\ldots,\partial_\theta),
\qquad |\partial_\theta^k a_i|\le C_k.
\]
The identity
\[
G_i''=\frac{(G_i')^2}{2G_i}-R_iG_i
\]
then bounds every derivative of $G_i$ inductively. After parametrizing the bases with a common period $\ell$, the base coefficient at time one is $u_i(1)=(\ell_i/\ell)^2$. Together with these time-one bounds, the metric distortion bounds from~\eqref{type.3.estimates.blow.down}, equations~\eqref{e:invariant.alpha.ricci.flow}, and the positive-time smoothing estimates in~\cite{guo2025alpha} give, after passing to a diagonal subsequence with $\ell_i\to\ell_\infty>0$,
\[
(G_i(t),u_i(t),a_i(t))\longrightarrow(G_\infty(t),u_\infty(t),a_\infty(t))
\quad\text{in }C^\infty(S^1\times[\delta,T]),
\qquad 0<\delta<T<\infty.
\]

By Corollary~\ref{1d.alpha.expander}, we can choose coordinates so that $G_\infty(t)=1$, $u_\infty(t)=t$ and $a_\infty(t)=\sqrt2$. At time one, $(G_i(1),u_i(1),a_i(1))\longrightarrow(1,1,\sqrt2)$ in $C^\infty(S^1)$. Since $\bar\alpha$ evolves by an exact form, and $\alpha_i(1)=\bar\alpha(t_i)$, we can write $\eta_i'=a_i(1)-\sqrt2$ where $\eta_i\in C^\infty(S^1)$ satisfies $\int_0^\ell\eta_i\,d\theta=0$. Then we have $(G_i(1),u_i(1),\eta_i)\longrightarrow(1,1,0)$ in $h^{1+\theta}(S^1)$.
For $s>0$, write the blow-down tuple as
\[
\mathbf{\bar g}_s(t)
=\bigl(s^{-1}G(st),s^{-1}u(st),\eta(st)\bigr),
\qquad a=\sqrt2+\partial_0\eta.
\]
The constant fiber change $S=\sqrt s\,I$ gives
\[
[\mathbf{\bar g}_s(1)]=[(G(s),s^{-1}u(s),\eta(s))].
\] Given $\epsilon>0$, choose $\delta=\delta(\epsilon)>0$ from Corollary~\ref{neighborhood.around.expander}. The preceding convergence implies $[\mathbf{\bar g}_{t_i}(1)]\in\mathcal V_\delta$ for all sufficiently large $i$. Fix $T>0$. Since $t_i\to0$, we have $T/t_i\ge1$ for all sufficiently large $i$. The blow-down by $t_i$, followed by the blow-down by $T/t_i$, is the blow-down by $T$. Therefore Corollary~\ref{neighborhood.around.expander} gives
\[
[\mathbf{\bar g}_T(1)]
=[(\mathbf{\bar g}_{t_i})_{T/t_i}(1)]
\in\mathcal U_\epsilon.
\]
Letting $\epsilon\downarrow0$ gives $[\mathbf{\bar g}_T(1)]=[\mathbf g^*]$ for every $T>0$.
Hence, in the original coordinates,
\[
G'=0,\qquad u(t)=\frac t2a(t)^2.
\]
Substituting into~\eqref{gauged.invariant.alpha.ricci.flow} gives
\[
\partial_t G=\partial_t a=\partial_t A=0.
\]
Thus, on the universal cover, $(\bar g(t),\bar\alpha(t))$ is isometric to $(g^*(t),\alpha^*(t))$.
\end{proof}

\section{No fiber component when \texorpdfstring{$[\alpha]$}{[alpha]} is rational}\label{s:no.fiber}
In view of Corollary~\ref{alpha.expander}, the monotonicity of $\mathcal{W}_+$ holds when the fiber component of the blow-down $\bar \alpha$ vanishes. In this section, we show that when $[\alpha]$ is rational, the fiber component of any blow-down vanishes, and thus we have the monotonicity formula.

In the following sections, we always assume $(g(t), \alpha(t))$ is a solution to the $\alpha$-Ricci flow on $M=\mathbb{T}^2$. We first use Cheeger-Fukaya-Gromov theory to get generators of $\pi_1(M)$ whose lengths are bounded.

\begin{lemma}\label{l:collapsing.generators}
    Suppose there exists a sequence $s_i\to\infty$ such that $(M,g_{s_i}(1))\xrightarrow{GH} (S^1,d)$. Then there are closed curves $\gamma_i,\sigma_i$ such that $[\gamma_i],[\sigma_i]$ generate $\pi_1(M)$ and $L(\gamma_i,g_{s_i}(1))\to0$ and $L(\sigma_i,g_{s_i}(1))\to L(S^1,d)$.
\end{lemma}
\begin{proof}
    By \cite[Theorem 2.6]{cheeger1992nilpotent}, for some sequence $\epsilon_i\to0$, there exists an $\epsilon_i$-almost Riemannian submersion $F_i:(M,g_{s_i}(1))\to (S^1,d)$ such that $\diam F_i^{-1}(\theta)\le \epsilon_i$ for any $\theta\in S^1$.  Let $\gamma_i=F_i^{-1}(\theta)\cong S^1$ be one of the fibers of $F_i$. Then $L(\gamma_i,g_{s_i}(1))\to0$. Let $\sigma_{1,i}$ be the integral curve of the horizontal distribution of $F_i:M\to S^1$ such that $\sigma_{1,i}(0)=x\in M$ and $\sigma_{1,i}(1)\in F_i^{-1}(F_i(x))$. Let $\sigma_{2,i}$ be a shortest curve in $F_i^{-1}(F_i(x))$ connecting $\sigma_{1,i}(1)$ and $x$. Let $\sigma_i=\sigma_{1,i}\cup \sigma_{2,i}$. Then by \cite[Theorem 2.6]{cheeger1992nilpotent}, we have $L(\sigma_i,g_{s_i}(1))\le (1+\epsilon_i)L(S^1,d)+\epsilon_i$.  Since $M=\mathbb{T}^2$ is orientable, $F_i$ is a trivial fibration. Since $F_{i*}[\sigma_i]=[S^1]$, we have $[\gamma_i],[\sigma_i]$ generate $\pi_1(M)$.
\end{proof}

Next, we show that when $[\alpha]$ is rational, the fiber of the fibration given by collapsing has zero pairing with $[\alpha]$.
\begin{lemma}\label{l:zero.pairing.with.fiber}
    Suppose the cohomology class $[\alpha]$ is rational and there exists a sequence $s_i\to\infty$ such that $(M,g_{s_i}(1))\xrightarrow{GH} (S^1,d)$. Let $\gamma_i,\sigma_i$ be as in Lemma~\ref{l:collapsing.generators}. Then for $i$ sufficiently large, we have $[\alpha](\gamma_i)=0$.
\end{lemma}
\begin{proof} 
    Suppose for contradiction that there exists a sequence of indices $i \to \infty$ such that $[\alpha](\gamma_i)\neq0$. Then Theorem~\ref{growth.of.length.of.shortest.curve} applied to the flow $(g(t),\alpha(t))$ gives $L_{[\gamma_i]}(t)^2\ge\frac{1}{2}[\alpha](\gamma_i)^2 t+L_{[\gamma_i]}(0)^2$. Hence $L_{[\gamma_i]}(s_i)\ge \frac{1}{\sqrt 2}|[\alpha](\gamma_i)|\sqrt {s_i}$. In particular, $L(F_{i}^{-1}(\theta),g_{s_i}(1))\ge s_i^{-1/2}L_{[\gamma_i]}(s_i)\ge \frac{1}{\sqrt 2}|[\alpha](\gamma_i)|>0$. By Lemma~\ref{l:collapsing.generators}, we have $L(F_{i}^{-1}(\theta),g_{s_i}(1))\to0$, forcing $|[\alpha](\gamma_i)| \to 0$. By Lemma~\ref{action.on.level.sets}, since $[\alpha]$ is rational, if $|[\alpha](\gamma_i)|<\ell([\alpha])$, then $[\alpha](\gamma_i)=0$. This contradicts $[\alpha](\gamma_i)\neq0$.
\end{proof}

\begin{theorem}[No fiber component]\label{no_fiber_component} 
Suppose $[\alpha]$ is rational and there is a sequence $s_i\to\infty$ such that $(M,g_{s_i}(1))\allowbreak\xrightarrow{GH} (S^1,d)$ and $(\tilde{g}_{s_i}(t), \tilde{\alpha}_{s_i}(t))\xrightarrow{C^\infty}(\tilde g_\infty(t), \tilde \alpha_\infty(t))$. Then $(\tilde g_\infty(t), \tilde \alpha_\infty(t))$ has no fiber component.
\end{theorem}
\begin{proof}
Suppose for the sake of contradiction that the limit has a nonzero fiber component. Then, by Corollary~\ref{fiber.constant.in.spacetime}, $\tilde{\alpha}_\infty(\partial_v)(\theta, v, t) = b \neq 0$ for any $(\theta, v, t) \in \R^2 \times (0,\infty)$.

By \cite[Theorem 2.6]{cheeger1992nilpotent}, for $i$ sufficiently large, we have a smooth fibration map $F_i: (M, g_{s_i}(1)) \to (S^1, d)$ which is an $\epsilon_i$-almost Riemannian submersion. Since $M = \mathbb{T}^2$ is orientable and $F_i: M \to S^1$ is a fibration with fiber $S^1$, the fibration $F_i$ must be trivial.

Choose $0<\delta < \operatorname{inj}(S^1,d)$ sufficiently small. By the controlled-contraction argument in the proof of Corollary~\ref{c:convergence_CFG}, globally trivial short loops contract in the $2\delta$-ball, so the local fundamental group $\pi_1(M,x_i;\delta)$ identifies with restrictions of the deck action of $\pi_1(M)$ on $\tilde M$. Let $\gamma_i$ be the fiber of $F_i$ as in Lemma~\ref{l:collapsing.generators} and \ref{l:zero.pairing.with.fiber}. Then, by the same argument as in \cite[Proposition 11.5]{fukaya1990hausdorff}, there exists $n_i \in \N$ such that $0 < \lim_{i\to\infty} d\big([\gamma_i]^{n_i}(p), p\big) \le \delta$ for $p \in \tilde{M}$.

After passing to a subsequence, we may assume that $[\gamma_i]^{n_i} \to \gamma_\delta \in G_\infty$. Consequently, we have $0 < d(\gamma_\delta(p), p) \le \delta$. Since $\delta < \operatorname{inj}(S^1,d)$, $\gamma_\delta$ must be a translation purely in the $\partial_v$ direction.

Normalize and orient $v$ so that the metric on the fiber through $p$ at time one is $dv^2$ and $\gamma_\delta$ translates in the positive direction. Let $h>0$ be the intrinsic distance on this fiber from $p$ to $\gamma_\delta(p)$. Let $\eta_1$ be the curve along the $\R$-orbit from $p$ to $\gamma_\delta(p)$, and let $\eta_2$ be the shortest geodesic with respect to $\tilde{g}_\infty(1)$ from $\gamma_\delta(p)$ to $[\gamma_i]^{n_i}(p)$. Defining $\eta=\eta_1\cup\eta_2$, we have
\begin{align*}
        n_i[\alpha](\gamma_i) &= \int_{\pi(\eta)} \alpha_{s_i} = \int_{\eta} \tilde{\alpha}_{s_i} = \int_{\eta} \tilde{\alpha}_\infty + o(1) \\
        &= \int_{\eta_1} \tilde{\alpha}_\infty + \int_{\eta_2} \tilde{\alpha}_\infty + o(1) \\
        &= bh + \mathcal{E}_i,
    \end{align*}
where the error term $\mathcal{E}_i$ is bounded by
    \[
        |\mathcal{E}_i| \le \sup_{\R^2}|\tilde{\alpha}_\infty|_{\tilde{g}_\infty} d\big(\gamma_\delta(p), [\gamma_i]^{n_i}(p)\big) + o(1).
    \]Since $d\big(\gamma_\delta(p), [\gamma_i]^{n_i}(p)\big) \to 0$ as $i \to \infty$, the error term $\mathcal{E}_i$ vanishes in the limit. Thus, for $i$ sufficiently large, $[\alpha](\gamma_i) \neq 0$. This contradicts Lemma~\ref{l:zero.pairing.with.fiber}.
\end{proof}

\section{Gromov--Hausdorff limit and rationality of \texorpdfstring{$[\alpha]$}{[alpha]}}\label{s:GHlimit_vs_rationality}
In this section we characterize the rationality of $[\alpha]$ in terms of the Gromov--Hausdorff limit of $(M, g_{s_i}(1))$, under the assumption that the universal covers converge to the $\alpha$-expander $(g^*(t),\alpha^*(t))$ with no fiber component. 

The following lemma establishes that if convergence on the universal cover to $(g^*(t), \alpha^*(t))$ holds along the full sequence, then any element $\gamma \in \pi_1(M)$ that is bounded at some sufficiently large time remains bounded for all later times.
\begin{lemma}\label{l:bounded.generator}
Suppose
\[
(\tilde g_s(t),\tilde\alpha_s(t))
\xrightarrow{C^\infty}(g^*(t),\alpha^*(t))
\quad\text{as }s\to\infty.
\]
Identify $\pi_1(M)$ with the group of deck transformations of the
universal cover and fix a basepoint $0\in\tilde M$. For every $L>0$,
there exist $s_0=s_0(L)>0$ and $L'=\sqrt 8 L>0$ such that for every
$s_1\ge s_0$ and $\gamma\in\pi_1(M)$ satisfying
\[
d_{\tilde g_{s_1}(1)}\bigl(\gamma(0),0\bigr)\leq L
\]
we have
\[
\sup_{s\geq s_1}d_{\tilde g_s(1)}\bigl(\gamma(0),0\bigr)\leq L'.
\]
\end{lemma}
\begin{proof}
Let $\varepsilon\in (0,\frac{1}{3})$ and $R>1$, to be chosen below. By smooth convergence, there exists $s_0=s_0(\varepsilon,R)>0$ such that $(\tilde g_s(t),\tilde\alpha_s(t))$ are $\varepsilon$-close to $(g^*(t),\alpha^*(t))$ for any
$s\geq s_0$ on $B_R(g^*(1))\times [R^{-1},R]$. We use the
coordinates $(\theta,v)$ from \eqref{eq:flat_expander_model}.
Write
\[
D_s(\gamma)=d_{\tilde g_s(1)}(\gamma(0),0).
\]For any $\gamma$ satisfying $D_{s_1}(\gamma)\le L$, Theorem \ref{alpha.norm.estimate} gives
\[
|[\alpha](\gamma)|
\le \|\alpha_{s_1}(1)\|_{L^\infty(g_{s_1}(1))}L
\le \sqrt{2}L.
\]

Let $s\geq s_1$ be such that $D_s(\gamma)^2<(1-\varepsilon)R^2$.  We  
write $\gamma(0)=(\theta_0,v_0)$. Then $(v_0^2+\theta_0^2)\le (1-\varepsilon)^{-1}D_s(\gamma)^2< R^2$ and thus $(\theta_0,v_0)\in B_R(g^*(1))$.
Metric comparison
at times $1$ and $1\le r\le R$ gives
\begin{align}
d_{\tilde g_s(1)}\bigl(\gamma(0),0\bigr)^2
&\geq (1-\varepsilon)(v_0^2+\theta_0^2),
\label{eq:estimate.g_s(1)}\\
d_{\tilde g_s(r)}\bigl(\gamma(0),0\bigr)^2
&\leq (1+\varepsilon)(v_0^2+r\theta_0^2).
\label{eq:estimate.g_s(R)}
\end{align}

Because $\tilde\alpha_s(t)$ and $\alpha^*$ are closed, their integrals
depend only on the homotopy class. Integrating along the straight segment
from $0$ to $\gamma(0)$ therefore yields
\[
\left|[\alpha](\gamma)-\sqrt{2}\,\theta_0\right|
\leq \varepsilon\sqrt{v_0^2+\theta_0^2},
\]
where $[\alpha](\gamma)=\int_\gamma\tilde\alpha_s(t)$ is independent of
$s$ and $t$. Consequently,
\begin{equation}\label{eq:theta}
\theta_0^2
\leq [\alpha](\gamma)^2
+\varepsilon(v_0^2+\theta_0^2)\le 2L^2+\varepsilon(v_0^2+\theta_0^2).
\end{equation}
The scaling identity $\tilde g_{sR}(1)=R^{-1}\tilde g_s(R)$, together
with \eqref{eq:estimate.g_s(1)}--\eqref{eq:theta}, gives
\begin{equation}\label{eq:dg_sR(1)}
\begin{aligned}
    d_{\tilde g_{sR}(1)}\bigl(\gamma(0),0\bigr)^2
\le& (1+\varepsilon)(\frac{v_0^2}{R}+\theta_0^2)\\
\le&(1+\varepsilon) (\frac{d_{\tilde g_s(1)}\bigl(\gamma(0),0\bigr)^2}{(1-\varepsilon)R}+2L^2+\varepsilon(v_0^2+\theta_0^2))\\
\le&(1+\varepsilon)2L^2+\frac{1+\varepsilon}{1-\varepsilon} (R^{-1}+\varepsilon)d_{\tilde g_s(1)}\bigl(\gamma(0),0\bigr)^2\\
\end{aligned}
\end{equation}
Choosing $\varepsilon<\frac{1}{8}$ and $R^2=R(L)^2>\max \{64,(1-\varepsilon)^{-1}8L^2\}$, we obtain
\begin{equation}
D_{sR}(\gamma)^2
\leq 4L^2
+\frac12D_s(\gamma)^2.
\end{equation}
 and thus since $8L^2\le (1-\varepsilon)R^2$, we have
\begin{equation}\label{eq:consequence.of.iteration}
    D_s(\gamma)^2\le 8L^2\Longrightarrow D_{sR}(\gamma)^2\le 8L^2.
\end{equation}

 By \eqref{eq:estimate.g_s(1)} and \eqref{eq:estimate.g_s(R)}, we have for $1\le r\le R$
\[
D_{s_1r}(\gamma)^2\le (1+\varepsilon)(\frac{v_0^2}{r}+\theta_0^2)\le (1+\varepsilon)(v_0^2+\theta_0^2)\le \frac{1+\varepsilon}{1-\varepsilon}D_{s_1}(\gamma)^2\le 2L^2.
\]
Thus $\sup_{s\in [s_1,s_1R]} D_{s}(\gamma)^2\le 2L^2<8L^2$. Applying \eqref{eq:consequence.of.iteration} iteratively gives for any $k\in\N$
\[\sup_{s\in [s_1R^k,s_1R^{k+1}]}D_{s}(\gamma)^2\le 8L^2,\] that is
\[\sup_{s\ge s_1}D_{s}(\gamma)^2\le 8L^2.\]This proves the claim.
\end{proof}

We now establish the characterization of the rationality of $[\alpha]$ in terms of the Gromov--Hausdorff limit. We recall that $d([\alpha])$ denotes the smooth metric on $S^1$ whose
total length is $\ell([\alpha])/\sqrt 2$ where $\ell([\alpha])$ is defined in Definition~\ref{d:prmitive length}.

\begin{theorem}\label{t:GH_limit_vs_rationality}
Let $s_i\to\infty$ be a sequence such that $(M,g_{s_i}(1))\xrightarrow{GH}(X,d)$, and suppose that, on the universal cover,
\[
(\tilde g_{s_i}(t),\tilde\alpha_{s_i}(t))
\xrightarrow{C^\infty}
\bigl(g^*(t),\alpha^*(t)\bigr).
\]
Then the following hold:
\begin{enumerate}[label=(\roman*)]
    \item If $(X,d)=\{\mathrm{pt}\}$, then $[\alpha]$ is irrational.
    \item If $[\alpha]$ is rational, then $(X,d)=(S^1,d([\alpha]))$.
    \item If, in addition, $(\tilde g_s(t),\tilde\alpha_s(t))\xrightarrow{C^\infty}(g^*(t),\alpha^*(t))$ as $s\to\infty$ with no fiber component, then $(X,d)=(S^1,d)$ implies that $[\alpha]$ is rational. Consequently, by (ii), $(X,d)=(S^1,d([\alpha]))$.
\end{enumerate}
\end{theorem}

\begin{proof}

\textbf{Step 1}: Setup.

As in Example~\ref{example.noncompact invariant expander}, in local coordinates of $\R^2$, we have $g^*(t) = dv^2 + td\theta^2$ and  $\alpha^*(t)=\sqrt 2 d\theta$. We define $u^*=\sqrt 2\theta\in C^\infty(\mathbb{R}^2)$ so that $\alpha^* = du^*$. We have the following equivariant convergence:
\begin{center}
    \begin{tikzcd}
  & (\tilde{M},\tilde{g}_{s_i}(1), \tilde\alpha_{s_i}(1), G_{s_i}, p) \arrow[r,  "pt\ eq\ C^\infty"] \arrow[d, "\pi"]
   & (\R^2,g^*(1), \alpha^*(1), G_\infty, p_\infty) \arrow[d,  "\pi_\infty"]\\
  & (M, g_{s_i}(1)) \arrow[r, "GH"]
   & (X,d),
\end{tikzcd}
\end{center}
where $(X, d)$ is either $(S^1, d)$ or a point and $G_\infty = \mathbb{R}\times\mathbb{Z}$ or $\mathbb{R}\times\mathbb{R}$, respectively, by Corollary~\ref{c:convergence_CFG}. For each $i$, we write $ g_i:= g_{s_i}(1)$, $\tilde g_i:=\tilde g_{s_i}(1)$, $\alpha_i:=\alpha_{s_i}(1)$ and $\tilde\alpha_i:=\tilde\alpha_{s_i}(1)$. Note, for each $i$, there exists $u_i \in C^\infty(\mathbb{R}^2)$, such that $du_{i} = \tilde\alpha_i$. We normalize these functions such that $u_{i}(0) = 0$. By Lemma~\ref{action.on.level.sets}, we know that $G_i$ acts on the collection of level sets $\{u_i^{-1}(r):r\in \R\}$.

\textbf{Step 2}: Passing the action of $G_i$ on level sets to the limit.

Smooth pointed convergence and the normalization $u_i(0)=u^*(0)=0$ give $u_i\to u^*$ in $C^\infty_{\mathrm{loc}}$. If $\gamma_i\to\gamma\in G_\infty$, then $[\alpha](\gamma_i)=u_i(\gamma_i(x))-u_i(x)\to u^*(\gamma(x))-u^*(x)$. The limit is independent of $x$, so $G_\infty$ acts on the level sets of $u^*$.

\textbf{Step 3}: We prove (i). If $(X,d)=\{\mathrm{pt}\}$, then $[\alpha]$ is irrational.

By Corollary~\ref{c:convergence_CFG}, $G_\infty=\R^2$ which acts transitively on $\R^2$. Thus, for any $\lambda\in \R$, there is $\gamma\in G_\infty$ such that $u^*(\gamma(x))-u^*(x)=\lambda$. Let $\gamma_i\in G_i$ be such that $\gamma_i\to \gamma$. Then $u_i(\gamma_i(x))-u_i(x)=[\alpha](\gamma_i)$ converges to $\lambda$ as $i\to\infty$. Since the action of $G_i$ on $\{u_i^{-1}(r):r\in \R\}$ is independent of $i$, the action of $G_i$ on $\{u_i^{-1}(r):r\in \R\}$ is not discrete. By Lemma~\ref{action.on.level.sets}, $[\alpha]$ is irrational. In particular, we know that if $[\alpha]$ is rational, then $(X,d)=(S^1,d)$.

\textbf{Step 4}: Identifying the $\R$ component of $G_\infty$ when $[\alpha]$ is rational.

By Step 3 and Theorem~\ref{no_fiber_component}, $(X,d)=(S^1,d)$ and the limit $(g^*(t),\alpha^*(t))$ has no fiber component. By Corollary~\ref{c:convergence_CFG}, $G_\infty \cong \R\times\Z$. Since $\alpha^* = \sqrt{2} \, d\theta$, the fiber direction cannot contain a $\partial_\theta$ component, which implies that $\partial_v$ must correspond to the fiber direction. Let $H_\infty = \big\{(\theta,v) \mapsto (\theta,v+v_1) : v_1 \in \R \big\} \le G_\infty$ be the $\R$-subgroup generated by $\partial_v$. We then observe that for any point $x \in \R^2$, if $u^*(x) = r$, the entire orbit $H_\infty \cdot x$ coincides precisely with the level set $(u^*)^{-1}(r)$.

\textbf{Step 5}: We prove (ii). If $[\alpha]$ is rational then $(X,d)=(S^1,d([\alpha]))$.

By Lemmas~\ref{l:collapsing.generators} and \ref{l:zero.pairing.with.fiber}, there exist closed curves $\gamma_i$ and $\sigma_i$ such that $[\gamma_i], [\sigma_i]$ generate $\pi_1(M)$, $L(\gamma_i, g_i) \to 0$, $L(\sigma_i, g_i) \to L(S^1, d)$, and $[\alpha](\gamma_i) = 0$. Since $[\gamma_i], [\sigma_i]$ generate $\pi_1(M)$, Lemma~\ref{action.on.level.sets} implies that $u_i([\sigma_i](x))-u_i(x)=[\alpha](\sigma_i)=\ell([\alpha])$. By Lemma~\ref{action.on.level.sets}, the action of $G_i$ on $\{u_i^{-1}(r):r\in\R\}$ is generated by $r\mapsto r+\ell([\alpha])$ which is independent of $i$. Thus by Step 2, the action of $G_\infty$ on $\{(u^*)^{-1}(r):r\in \R\}$ is a subgroup of $\langle r\mapsto r+\ell([\alpha])\rangle$. 
We now show that the action is the full group. Since $d_{g_i}([\sigma_i](x),x)\le L(\sigma_i,g_i)$ is uniformly bounded, passing to a subsequence yields $[\sigma_i]\to \sigma_\infty\in G_\infty$. Thus we have $u^*(\sigma_\infty(x))-u^*(x)=\ell([\alpha])$. Therefore, the action of $G_\infty$ on $\{(u^*)^{-1}(r):r\in \R\}$ is generated by $ r\mapsto r+\ell([\alpha])$. By Step 4, the orbit $H_\infty \cdot x$ coincides with the level set $(u^*)^{-1}(r)$. Thus $S^1$ has length $\ell([\alpha])/\sqrt 2$.

\textbf{Step 6}: We prove (iii).

    Suppose $L=L(S^1,d)>0$. By Lemma~\ref{l:collapsing.generators}, there exist generators $[\gamma],[\sigma]\in \pi_1(M)$ such that for some $s_1>s_0(2L)$ given in Lemma~\ref{l:bounded.generator}, we have $d_{g_{s_1}(1)}([\gamma](0),0)\le 2L$, $d_{g_{s_1}(1)}([\sigma](0),0)\le 2L$. By Lemma~\ref{l:bounded.generator}, we have $d_{g_s(1)}([\gamma](0),0)\le L'$ and $d_{g_{s}(1)}([\sigma](0),0)\le L'$ for any $s\ge s_1$. Then $G_i\ni[\gamma]\to \gamma_\infty\in G_\infty$ and $G_i\ni[\sigma]\to \sigma_\infty\in G_\infty$. Since $[\gamma],[\sigma]$ are generators, the action of $G_i$ on $\{u_i^{-1}(r):r\in \R\}$ then fully passes to the limit action of $G_\infty$ on $\{(u^*)^{-1}(r):r\in \R\}$. Thus if $(g^*(t),\alpha^*(t))$ has no fiber component and $(X,d)=(S^1,d)$, then by the same argument as Step 4, the action of $G_\infty$ on $\{(u^*)^{-1}(r):r\in \R\}$ is discrete and so is the action of $G_i$ on $\{u_i^{-1}(r):r\in \R\}$. Thus $[\alpha]$ is rational by Lemma \ref{action.on.level.sets}.
\end{proof}

\begin{remark}
Theorem~\ref{t:GH_limit_vs_rationality} is an analog of \cite[Proposition~6.2]{lott2010dimensional}. In the 3D Type III Ricci flow with diameter bounds studied by Lott \cite{lott2010dimensional}, there is a one-to-one correspondence between the expander on the universal cover and the Gromov--Hausdorff limit of the blow-down. In our setting, however, this one-to-one correspondence fails because the same expander $(g^*(t), \alpha^*(t))$ can yield different Gromov--Hausdorff limits. Indeed, as shown in Example~\ref{example.noncompact invariant expander}, $(g^*(t), \alpha^*(t))$ can be quotiented by either an $\mathbb{R}\times\mathbb{Z}$ or an $\mathbb{R} \times \mathbb{R}$ group action to obtain the corresponding base limit. The information determining which quotient is realized is precisely encoded in the cohomology class $[\alpha]$, as manifested in Lemma~\ref{action.on.level.sets} and the proof above.
\end{remark}
\begin{remark}\label{r:S1.not.imply.rational}
       One cannot conclude that $[\alpha]$ is rational merely from the fact that, along a subsequence, the Gromov--Hausdorff limit is $(S^1, d)$ and the limit on the universal cover $(g^*(t+c), \alpha^*(t+c))$ has no fiber component. Indeed, the action of $\pi_1(M)$ on $\{u_i^{-1}(r):r\in\mathbb{R}\}$ may be lost in the limit, allowing a sequence of dense actions to converge to a discrete one. For example, in $(\theta,v)$ coordinates, consider the lattices generated by $(\frac{\sqrt 2}{i^2},\frac{1}{i})$ and $(1,0)$. The corresponding action on $\{(u^*)^{-1}(r):r\in\mathbb{R}\}$ is dense for every $i$, but the lattices converge to the group of integer $\theta$-translations and arbitrary $v$-translations, isomorphic to $\mathbb{R}\times\mathbb{Z}$, which gives a discrete action in the limit. Thanks to Lemma~\ref{l:bounded.generator}, when full-sequence convergence on the universal cover holds, such pathologies cannot occur.
       \end{remark}

\section{Proof of Theorem~\ref{t:main.GH_limit} (i) and Corollary~\ref{c:klein.bottle}}\label{s:proof.of.main.theorem.(i)}
We now complete the proof of Theorem~\ref{t:main.GH_limit} (i) and Corollary~\ref{c:klein.bottle}. First, we rule out the trivial flat model for the rational case.

\begin{lemma}\label{rational.class.alpha.nonzero}
    Suppose $[\alpha]$ is rational. Furthermore, suppose that there exists a sequence $s_i\to\infty$ such that $(\tilde{g}_{s_i}(t),\tilde{\alpha}_{s_i}(t))\xrightarrow{C^\infty}(\tilde{g}_\infty(t),\tilde{\alpha}_\infty(t))$. Then $(\tilde{g}_\infty(t),\tilde{\alpha}_\infty(t))$ cannot be $(g_{\mathrm{Euc}},0)$.
\end{lemma}
\begin{proof}
    Suppose for contradiction that $(\tilde g_\infty(t),\tilde\alpha_\infty(t))=(g_{\mathrm{Euc}},0)$. For each $i$, we write $ g_i:= g_{s_i}(1)$, $\tilde g_i:=\tilde g_{s_i}(1)$, $\alpha_i:=\alpha_{s_i}(1)$ and $\tilde\alpha_i:=\tilde\alpha_{s_i}(1)$.
    On $\tilde M$, there exists a function $u_i\in C^\infty(\tilde M)$ such that $\tilde \alpha_i=du_i$. Since $[\alpha]$ is nonzero, the function $u_i:\R^2\to \R$ is surjective. 
    
    Set $\ell=\ell([\alpha])$ and choose $p\in u_i^{-1}(\ell/2)$ and $q\in u_i^{-1}(0)$. By Lemma~\ref{action.on.level.sets}, $u_i([\gamma](q))\in\ell\Z$ for every deck transformation $[\gamma]$. Integrating $du_i=\tilde\alpha_i$ along a curve $c$ from $p$ to $[\gamma](q)$ gives $\ell/2\le L(c)\|\tilde\alpha_i\|_{L^\infty(\tilde g_i)}$. Taking the infimum over these curves and deck transformations yields
    \begin{equation}\label{lower.bound.of.diameter}
        \operatorname{diam}(M,g_i)\ge d_{g_i}(\pi(p),\pi(q))
        \ge \frac{\ell}{2\|\tilde\alpha_i\|_{L^\infty(\tilde g_i)}}.
    \end{equation}
    As $i\to\infty$, smooth pointed convergence to
$(g_{\mathrm{Euc}},0)$ implies that
$|\tilde\alpha_i|_{\tilde g_i}^2\to0$ uniformly on bounded balls.
By Theorem \ref{diameter.bound}, every point of $M$ has a lift
within a fixed distance of the chosen basepoint in $\tilde M$.
Since $\tilde\alpha_i$ and $\tilde g_i$ are deck-invariant,
it follows that
$
\sup_{\tilde M}|\tilde\alpha_i|_{\tilde g_i}^2\to0
$. By \eqref{lower.bound.of.diameter}, this forces $\operatorname{diam}(M, g_i) \to \infty$, directly contradicting the uniform upper bound $\operatorname{diam}(M, g_i) \le C$ from Theorem~\ref{diameter.bound}.
\end{proof}
\begin{remark}
    The argument in Lemma~\ref{rational.class.alpha.nonzero} does not extend to the case where $[\alpha]$ is irrational. In that setting, the action of $\pi_1(M)$ on $\{u_i^{-1}(r):r\in\R\}$ has dense orbits so that there is no contradiction to the diameter bound Theorem~\ref{diameter.bound}.
\end{remark}
Now, we can prove Theorem~\ref{t:main.GH_limit} when $[\alpha]$ is rational.
\begin{proof}[Proof of Theorem~\ref{t:main.GH_limit} (i)]
    Suppose for contradiction that there exists a sequence $s_i\to\infty$ such that $(M,g_{s_i}(1))$ converges in Gromov--Hausdorff topology to a point. Then by Corollary~\ref{c:convergence_CFG} $(\tilde g_{s_i}(t),\tilde \alpha_{s_i}(t))\xrightarrow{C^\infty}(\tilde g_{\infty}(t),\tilde\alpha_\infty(t))$ where $(\tilde g_{\infty}(t),\tilde\alpha_\infty(t))$ is $\R^2$-invariant. By Lemma~\ref{l:R2.invariant.solution}, $(\tilde g_{\infty}(t),\tilde\alpha_\infty(t))$ is either $(g^*(t+c),\alpha^*(t+c))$ for some $c\ge 0$ or $(g_{\mathrm{Euc}},0)$. Then Lemma~\ref{rational.class.alpha.nonzero} shows that $(\tilde g_{s_i}(t),\tilde \alpha_{s_i}(t))\xrightarrow{C^\infty}(g^*(t+c),\alpha^*(t+c))$. We consider a further blow-down $(\tilde g_{\infty,s}(t),\tilde \alpha_{\infty,s}(t))$ and it converges to $(g^*(t),\alpha^*(t))$ as $s\to\infty$ by Proposition~\ref{p:canonical_alpha_expander_attractor}. For each fixed $r\ge1$, we have $r^{-1}g^*(r+c)=r^{-1}dv^2+(1+c/r)d\theta^2\le g^*(1+c)$, so the diameter of the limiting quotient cannot increase and the base limit remains a point. Equivariant convergence and a diagonal choice as $r\to\infty$ therefore preserve this point limit. Thus there exists $t_i\to\infty$ such that $(\tilde g_{t_i}(t),\tilde\alpha_{t_i}(t))\xrightarrow{C^\infty}(g^*(t),\alpha^*(t))$ and $(M,g_{t_i}(1))\xrightarrow{GH}\{\mathrm{pt}\}$. By Theorem~\ref{t:GH_limit_vs_rationality}, $[\alpha]$ is irrational. This is a contradiction. Therefore, for every sequence $s_i \to \infty$, $(M, g_{s_i}(1))$ converges in the Gromov--Hausdorff sense to $(S^1, d)$ for some metric $d$ on $S^1$. Hence by Proposition~\ref{p:base.S^1.implies.universal.cover.expander} and Theorem~\ref{no_fiber_component}, it follows that $(\tilde g_s(t),\tilde \alpha_s(t))\xrightarrow{C^\infty}(g^*(t),\alpha^*(t))$. By Theorem~\ref{t:GH_limit_vs_rationality} again, we conclude that $(M,g_{s}(1))\xrightarrow{GH} (S^1,d([\alpha]))$.
\end{proof}

The Klein bottle case can be proved using the same idea.
\begin{proof}[Proof of Corollary~\ref{c:klein.bottle}]
    Let $\pi:\mathbb{T}^2\to K$ be the two-sheeted cover of the Klein bottle by the 2-torus $\mathbb{T}^2$. Then $(\pi^*g(t),\pi^*\alpha(t))$ is an $\alpha$-Ricci flow on $\mathbb{T}^2$.

    If $[\alpha]=0$, Lott and Sesum's result~\cite{lott2014ricci} gives $(\mathbb T^2,t^{-1}\pi^*g(t))\allowbreak\xrightarrow{GH}\{\mathrm{pt}\}$ as well as $(\tilde g_s(t),\tilde\alpha_s(t))\allowbreak\xrightarrow{C^\infty}(g_{\mathrm{Euc}},0)$. Since $\diam(K,t^{-1}g(t))\le\diam(\mathbb T^2,t^{-1}\pi^*g(t))$, the Klein bottle also collapses to a point, proving (ii).

    Now suppose that $[\alpha]\neq0$. Then $[\pi^*\alpha]\neq0$, and since $K$ and $\mathbb{T}^2$ share the same universal cover, applying Theorem~\ref{t:main.GH_limit} to $(\pi^*g(t),\pi^*\alpha(t))$ gives $(\tilde g_s(t),\tilde\alpha_s(t))\xrightarrow{C^\infty}(g^*(t),\alpha^*(t))$.
    
    We now show that the Gromov--Hausdorff limit of $(K,g_s(1))$ is $(S^1,d([\alpha]))$. Let $[\gamma_1],[\gamma_2]$ be the generators of $\pi_1(K)$ such that $[\alpha]([\gamma_1])=\ell([\alpha])$ and $[\alpha]([\gamma_2])=0$. Here we used the fact that $[\gamma_2]$ is torsion in $H_1(K,\Z)$ and hence vanishes in $H_1(K,\R)$. The generators of $\pi_1(\mathbb{T}^2)$ are $[\gamma_1]^2$, $[\gamma_2]$. We have $[\pi^*\alpha]([\gamma_1]^2)=[\alpha]([\gamma_1]^2)=2\ell([\alpha])$, $[\pi^*\alpha]([\gamma_2])=[\alpha]([\gamma_2])=0$. Then $[\pi^*\alpha]$ is rational and we have $\ell([\pi^*\alpha])=2\ell([\alpha])$. By Theorem~\ref{t:main.GH_limit}, $G_i=\pi_1(\mathbb{T}^2)$ converges to $G_\infty=\R\times\Z$ which is generated by $\partial_v$ and $\theta\mapsto \theta+ \ell([\pi^*\alpha])/\sqrt {2}$.

    The bounded diameter of the torus cover allows us to choose $\tau_i$ in the left coset $[\gamma_1]G_i$ with uniformly bounded displacement, so a subsequence converges to an isometry $\tau$. The Klein-bottle relation gives $\tau_i[\gamma_2]^n\tau_i^{-1}=[\gamma_2]^{-n}$. By equivariant convergence and discreteness of the period subgroup, every fiber translation $T_h(\theta,v)=(\theta,v+h)$ is a limit of powers of $[\gamma_2]$; hence $\tau T_h\tau^{-1}=T_{-h}$. Since $\tau$ preserves $\alpha^*=\sqrt2\,d\theta$, it has the form $\tau(\theta,v)=(\theta+a,-v+r)$. The periods of $\tau_i$ are odd multiples of $\ell([\alpha])$, so $a$ is an odd multiple of $\ell([\alpha])/\sqrt2$. Composing with elements of $G_\infty$, we may take $a=\ell([\alpha])/\sqrt2$ and $r=0$. Thus the full limiting group is $\R\rtimes\Z$, generated by the fiber translations and this glide, and its quotient is the circle of length $\ell([\alpha])/\sqrt2$. This gives $(K,g_s(1))\xrightarrow{GH}(S^1,d([\alpha]))$, as claimed.
\end{proof}

\section{Proof of Theorem~\ref{t:main.GH_limit} (ii)}\label{s:proof.of.main.theorem.(ii)}
\begin{proof}[Proof of Theorem~\ref{t:main.GH_limit} (ii)]
    By Corollary~\ref{c:convergence_CFG}, after passing to a subsequence if necessary, at least one of the following three alternatives occurs:
\begin{enumerate}[label=(\arabic*)]
    \item $(M,g_s(1))\xrightarrow{GH} \{\mathrm{pt}\}$ as $s\to\infty$;

    \item There exists a sequence $s_i\to\infty$ such that $(M,g_{s_i}(1))\xrightarrow{GH}(S^1,d)$, and on the universal cover $(\tilde{g}_{s_i}(t),\tilde{\alpha}_{s_i}(t))\xrightarrow{C^\infty}(\tilde{g}_\infty(t),\tilde{\alpha}_\infty(t))$ with a nonzero fiber component;

    \item There exists a sequence $s_i\to\infty$ such that $(M,g_{s_i}(1))\xrightarrow{GH}(S^1,d)$, and on the universal cover $(\tilde{g}_{s_i}(t),\tilde{\alpha}_{s_i}(t))\xrightarrow{C^\infty}(\tilde{g}_\infty(t),\tilde{\alpha}_\infty(t))$ with no fiber component.
\end{enumerate}
    We will show that in each of the cases above, there exists a sequence $s_i' \to \infty$ such that $(M,g_{s_i'}(1))\xrightarrow{GH} \{\mathrm{pt}\}$ and $(\tilde{g}_{s_i'}(t), \tilde{\alpha}_{s_i'}(t)) \xrightarrow{C^\infty} (g^*(t), \alpha^*(t))$.

\noindent\textbf{Case 1}

    Suppose for contradiction that for any sequence  $s_i\to\infty$, $(\tilde g_{s_i}(t),\tilde\alpha_{s_i}(t))$ stays uniformly bounded away from $(g^*(t),\alpha^*(t))$. By Corollary~\ref{c:convergence_CFG} and that $(M,g_s(1))\xrightarrow{GH} \{\mathrm{pt}\}$, there is a convergent subsequence converging to $(\tilde g_\infty(t),\tilde \alpha_\infty(t))$ which is $\R^2$-invariant. By Lemma~\ref{l:R2.invariant.solution}, $(\tilde g_\infty(t),\tilde \alpha_\infty(t))$ is either $(g^*(t+c),\alpha^*(t+c))$ for some $c\ge0$ or $(g_{\mathrm {Euc}},0)$. 
    
    If $(\tilde g_\infty(t),\tilde \alpha_\infty(t))=(g^*(t+c),\alpha^*(t+c))$, then we consider a further blow-down sequence $(\tilde g_{\infty,s}(t),\tilde \alpha_{\infty,s}(t))$ and it converges to $(g^*(t),\alpha^*(t))$ as $s\to\infty$ by Proposition~\ref{p:canonical_alpha_expander_attractor}. As in the proof of Theorem~\ref{t:main.GH_limit}(i), there exists $t_i\to\infty$ such that $(\tilde g_{t_i}(t),\tilde\alpha_{t_i}(t))$ converges to $(g^*(t),\alpha^*(t))$. This contradicts our assumption. Thus  $(\tilde g_\infty(t),\tilde \alpha_\infty(t))=(g_{\mathrm {Euc}},0)$.
    
    By the above argument, for every sequence $s_i \to \infty$, there exists a subsequence such that $(\tilde{g}_{s_i}(t), \tilde{\alpha}_{s_i}(t)) \xrightarrow{C^\infty} (g_{\mathrm{Euc}}, 0)$. Consequently, the full family $(\tilde{g}_s(t), \tilde{\alpha}_s(t))$ converges smoothly to $(g_{\mathrm{Euc}}, 0)$ as $s \to \infty$. In particular, $|R(\tilde{g}_s(t))| \to 0$ and $|\tilde{\alpha}_s(t)|^2_{\tilde{g}_s(t)} \to 0$ in $C^\infty_{\mathrm{loc}}$ as $s \to \infty$. Since $\operatorname{diam}(M, g_s(1)) \to 0$ as $s \to \infty$, it follows that $s |R(g(s))| \to 0$ and $s |\alpha(s)|_{g(s)}^2 \to 0$ uniformly on $M$.

Therefore, given any $\epsilon > 0$, there exists $T > 0$ such that for all $t \ge T$,
\[
    |R(g(t))| \le \epsilon t^{-1} \quad \text{and} \quad |\alpha(t)|_{g(t)}^2 \le \epsilon t^{-1}.
\]
By Theorem~\ref{growth.of.length.of.shortest.curve}, for any class $\sigma \in H_1(M, \mathbb{Z})$, the length $L_\sigma(t)$ of the shortest geodesic representative satisfies
\[
    D^+ L_\sigma(t) \le \int_0^{L_\sigma(t)} \left( -\frac{1}{2} R + \frac{1}{4} (\alpha(\gamma'))^2 \right) ds \le \epsilon t^{-1} L_\sigma(t).
\]
Integrating this differential inequality from $T$ to $t$, we obtain
\begin{equation}\label{L.sigma<=t^epsilon}
    L_\sigma(t) \le \left( \frac{t}{T} \right)^\epsilon L_\sigma(T).
\end{equation}
On the other hand, equation \eqref{growth.L_sigma} in Theorem~\ref{growth.of.length.of.shortest.curve} asserts that
\begin{equation}\label{L.sigma>=t^1/2}
    L_\sigma(t)^2 \ge \frac{1}{2} [\alpha](\sigma)^2 (t-T) + L_\sigma(T)^2.
\end{equation}
Since $[\alpha] \neq 0$, we can choose $\sigma \in H_1(M, \mathbb{Z})$ such that $[\alpha](\sigma) \neq 0$. Choosing $\epsilon<\frac{1}{2}$, \eqref{L.sigma<=t^epsilon} contradicts \eqref{L.sigma>=t^1/2} for $t$ large.

\noindent\textbf{Case 2}

    We know that $(\tilde g_\infty(t), \tilde \alpha_\infty(t))$ reduces to an $\R$-invariant $\alpha$-Ricci flow on $S^1$. Writing $(\tilde g_\infty(t), \tilde \alpha_\infty(t))$ as in \eqref{invariant.flow}, the flow satisfies system \eqref{gauged.invariant.alpha.ricci.flow}, which takes the form
    \begin{align}
	\frac{\partial G}{\partial t}=&\Delta_uG- G^{-1}|\nabla_uG|^2+\frac{1}{2}b^2\label{eqn.G}\\
	\frac{\partial A}{\partial t}=&\frac{1}{2}a bG^{-1}\label{eqn.A}\\
	\frac{\partial u}{\partial t}=&\frac{1}{2}G^{-2}(G')^2+\frac{1}{2}a^2\label{eqn.u}\\
	\frac{\partial a}{\partial t}=&[u^{-1/2}(au^{-1/2})']'-\frac{1}{2}a b^2G^{-1}\label{eqn.a}.
\end{align}
where $f'=\frac{df}{d\theta}$, $\nabla_u,\Delta_u$ are the gradient and Laplacian with respect to the metric $ud\theta^2$ on $S^1$ and $b=\tilde\alpha_\infty(\partial_v)\neq0$ is a spacetime constant by assumption and Corollary~\ref{fiber.constant.in.spacetime}.

Throughout this argument, we assume $t\ge 1$ and $C$ is some constant depending on time-one data.
Applying the maximum principle on \eqref{eqn.G} we have $\frac{1}{2}b^2t+ m\le G\le \frac{1}{2}b^2t+ M$ where $m=\min_{ S^1}G(1)-\frac{1}{2}b^2$ and $M=\max_{ S^1}G(1)-\frac{1}{2}b^2$. Setting $w=\log G$, we compute
\[\frac{\partial w}{\partial t}=\Delta_uw+\frac{1}{2}b^2G^{-1}.\]
Setting $Z=|\nabla w|^2=u^{-1}w'^2=u^{-1}\frac{G'^2}{G^2}$, we obtain
\[\frac{\partial Z}{\partial t}=\Delta_u Z-2|\nabla^2w|^2-\frac{1}{2}Z^2-\frac{1}{2}u^{-1}a^2Z-b^2G^{-1}Z\le\Delta_u Z-\frac{b^2}{\frac{1}{2}b^2t+M}Z.\]
Then by the maximum principle, we have \[u^{-1}\frac{G'^2}{G^2}\le \max_{S^1} Z(t)\le\max_{S^1} Z(1)\left(\frac{1+2M/b^2}{t+2M/b^2}\right)^2.\]
Let $Q=|ad\theta|^2=a^2/u$. By \eqref{eqn.a} and \eqref{eqn.u}, we have
\[
	\frac{\partial Q}{\partial t}=\Delta_uQ-2|\nabla (ad\theta)|^2- b^2G^{-1}Q-\frac{\partial_t u}{u}Q\le \Delta_u Q-\frac{b^2}{\frac{1}{2}b^2t+M}Q.\]
Applying the maximum principle, we have 
\begin{equation}\label{e:estimate.Q}
    u^{-1}a^2\le \max_{S^1} Q(t)\le\max_{S^1} Q(1)\left(\frac{1+2M/b^2}{t+2M/b^2}\right)^2.
\end{equation}
Thus we have
\[\frac{\partial u}{\partial t}=\frac{1}{2}G^{-2}(G')^2+\frac{1}{2}a^2\le \frac{Cu}{(t+2M/b^2)^2}.\]
Solving the differential inequality, we have \[\max_{S^1} u(t)\le \max_{S^1}u(1)\exp [C(\frac{1}{1+2M/b^2}-\frac{1}{t+2M/b^2})].\]
Plugging this into \eqref{e:estimate.Q} we have $|a(t)|\le C (t+2M/b^2)^{-1}$. Putting the estimates for $a$ and $G$ back to \eqref{eqn.A}, we have
\[|\frac{\partial A}{\partial t}|\le \frac{C}{(t+2M/b^2)^2}.\]
This gives
\[|A(t)|\le |A(1)|+C(\frac{1}{1+2M/b^2}-\frac{1}{t+2M/b^2}).\] 

To summarize, we have 
\begin{equation}\label{estimates.component.fiber.component}
    G(t)\sim \frac{1}{2}b^2t,\quad A(t)=O(1),\quad u(t)=O(1),\quad a(t)=O(\frac{1
}{t}).
\end{equation}
Since the fibers are $G_\infty$-orbits, the quotient metric at time $t$
is $u(\theta,t)\,d\theta^2$.
For each fixed $T>0$, equivariant convergence gives $(M,(s_iT)^{-1}g(s_iT))\xrightarrow{GH}(S^1,T^{-1}u(\theta,T)\,d\theta^2)$ as $i\to\infty$.  By
\eqref{estimates.component.fiber.component}, $\diam(S^1,T^{-1}u(\theta,T)\,d\theta^2)\to0$ as $T\to\infty$. By a diagonal choice, select $T_j\to\infty$ and increasing indices $i_j$ such that $s_j'=s_{i_j}T_j$ satisfies $\operatorname{diam}(M,g_{s_j'}(1))\to0$ and
\[
\left|
|\tilde\alpha_{s_j'}(1)|_{\tilde g_{s_j'}(1)}^2(\tilde p)
-T_j|\tilde\alpha_\infty(T_j)|_{\tilde g_\infty(T_j)}^2(\tilde p)
\right|\longrightarrow0.
\]
Since the component estimates and the positive lower bound for $u$ give
\[
t|\tilde\alpha_\infty(t)|_{\tilde g_\infty(t)}^2
=
t\left(\frac{a(\theta,t)^2}{u(\theta,t)}
+\frac{b^2}{G(\theta,t)}\right)
\longrightarrow2
\]
uniformly in $\theta$, this diagonal choice also ensures $|\tilde\alpha_{s_j'}(1)|_{\tilde g_{s_j'}(1)}^2(\tilde p)\to2$.

By Corollary~\ref{c:convergence_CFG}, after passing to a subsequence,
the flows on the universal cover converge to an
$\mathbb R^2$-invariant solution. By
Lemma~\ref{l:R2.invariant.solution}, the limit is either
$(g_{\mathrm{Euc}},0)$ or $(g^*(t+c),\alpha^*)$ for some $c\ge0$.
The time-one norm limit excludes the former and, since
$
|\alpha^*|_{g^*(1+c)}^2=\frac{2}{1+c},
$
forces $c=0$ in the latter. Thus $(\tilde g_{s_j'}(t),\tilde\alpha_{s_j'}(t))
\xrightarrow{C^\infty}(g^*(t),\alpha^*(t))$, while $(M,g_{s_j'}(1))\xrightarrow{GH} \{\mathrm{pt}\}$.

\noindent\textbf{Case 3}

If there is a sequence $s_i\to\infty$ satisfying Case 2, then we are done.

Otherwise, we may assume that whenever $s_i\to\infty$ satisfies $(M,g_{s_i}(1))\xrightarrow{GH}(S^1,d)$ and $(\tilde g_{s_i}(t),\tilde\alpha_{s_i}(t))\xrightarrow{C^\infty}(\tilde g_\infty(t),\tilde\alpha_\infty(t))$, the limit has no fiber component. We are going to deduce a contradiction.

We claim that there exists $s_i\to\infty$ such that $(M,g_{s_i}(1))\xrightarrow{GH}\{\mathrm{pt}\}$. Suppose for contradiction that any $s_i\to\infty$ such that $(M,g_{s_i}(1))\xrightarrow{GH}(X,d)$, we have $X=S^1$. By Corollary \ref{c:convergence_CFG}, for any $s_i\to\infty$, there is a subsequence such that $(M,g_{s_i}(1))$ converges. Then by our assumption and Proposition~\ref{p:base.S^1.implies.universal.cover.expander} $(\tilde g_{s_i}(t),\tilde\alpha_{s_i}(t))\xrightarrow{C^\infty}(g^*(t),\alpha^*(t))$ with no fiber component. Then every subsequential limit on the universal cover is $(g^*(t),\alpha^*(t))$, so we get the full sequence convergence $(\tilde g_{s}(t),\tilde\alpha_{s}(t))\xrightarrow{C^\infty}(g^*(t),\alpha^*(t))$ with no fiber component. By Theorem~\ref{t:GH_limit_vs_rationality}(iii), $[\alpha]$ is rational, contradicting our assumption that $[\alpha]$ is irrational.

Thus we get a sequence $s_i\to\infty$ such that $(M,g_{s_i}(1))\xrightarrow{GH}\{\mathrm{pt}\}$. By Corollary~\ref{c:convergence_CFG} and Lemma~\ref{l:R2.invariant.solution}, after passing to a subsequence, 
\[
(\tilde g_{s_i}(t),\tilde\alpha_{s_i}(t))\xrightarrow{C^\infty}(\tilde g_\infty(t),\tilde\alpha_\infty(t))
=(g^*(t+c),\alpha^*(t+c))\quad\text{or}\quad(g_{\mathrm{Euc}},0).
\]
In the first case, Proposition~\ref{p:canonical_alpha_expander_attractor} and a further blow-down give the desired sequence for Theorem \ref{t:main.GH_limit}(ii), with the Gromov-Hausdorff limit being a point as argued in the proof of Theorem~\ref{t:main.GH_limit}(i). It remains to rule out the possibility that every sequence whose Gromov--Hausdorff limit is a point has $(g_{\mathrm{Euc}},0)$ as its limit on the universal cover.

Under this assumption, the only possible subsequential blow-down limits on the universal cover are $(g^*(t),\alpha^*(t))$ and $(g_{\mathrm{Euc}},0)$: the first occurs when the Gromov--Hausdorff limit downstairs is a circle, and the second when it is a point. We claim that these two limits cannot both occur. Indeed, by continuity of $(\tilde g_s(t),\tilde \alpha_s(t))$ in $s$, if both occurred, one could choose a sequence of $s_i\to\infty$ for which $(\tilde g_{s_i}(t),\tilde \alpha_{s_i}(t))$ stays  a definite distance away from both limits. By Corollary \ref{c:convergence_CFG}, there is a convergent subsequence $(\tilde g_{s_i}(t),\tilde \alpha_{s_i}(t))$ whose limit is neither $(g^*(t),\alpha^*(t))$ nor $(g_{\mathrm{Euc}},0)$. This is a contradiction. Thus $(\tilde g_s(t),\tilde \alpha_s(t))$ converges to one of $(g^*(t),\alpha^*(t))$ or $(g_{\mathrm{Euc}},0)$. Since there is already a subsequence whose Gromov-Hausdorff limit is a circle, the full family must converge to $(g^*(t),\alpha^*(t))$. Now Theorem~\ref{t:GH_limit_vs_rationality}(iii) again implies that $[\alpha]$ is rational, a contradiction.
\end{proof}

\appendix
\section{Stability of \texorpdfstring{$\R^N$}{RN}-invariant \texorpdfstring{$\alpha$}{alpha}-Ricci flow over \texorpdfstring{$S^1$}{S1}}\label{s:dynamic stability}

We consider stability of locally $\R^N$-invariant $\alpha$-Ricci flow with no fiber components around $(g^*(t),\alpha^*(t))$ on a trivial $\R^N$-bundle over $S^1$ with trivial linear holonomy, among perturbations fixing the cohomology class of $\bar\alpha$. This includes the invariant flows used in Proposition~\ref{p:base.S^1.implies.universal.cover.expander}. The argument follows the stability method of \cite{knopf2009convergence}.
One can also obtain similar result for bundles with nontrivial linear holonomy as in \cite[Section~5]{knopf2009convergence}.
Write
\[
\bar g=u(dx^0)^2+G_{ij}(dx^i+A^i dx^0)(dx^j+A^j dx^0),\qquad
\bar\alpha=a\,dx^0.
\]
All coefficients are periodic in $x^0\in\R/\ell\Z$. Choose the orientation so that $\ell:=\frac{1}{\sqrt2}\int_{S^1}\bar\alpha>0$. Since $F=b=0$, equation~\eqref{gauged.invariant.alpha.ricci.flow} gives $A(t)=A(1)$. We consider $\tau=\log t$, $\tilde G(\tau)=G(e^\tau)$, $\tilde u(\tau)=e^{-\tau}u(e^\tau)$, $\tilde a(\tau)=a(e^\tau)$ and writing $\tilde a=\sqrt2+\partial_0\tilde\eta$. Then after adding a time dependent constant to $\tilde\eta$, \eqref{gauged.invariant.alpha.ricci.flow} gives
\begin{equation}\label{rescaled.equation}
\begin{aligned}
\partial_\tau\tilde G_{ij}
 &=\tilde u^{-1}\left(\partial_0^2\tilde G_{ij}
 -\frac{\partial_0\tilde u}{2\tilde u}\partial_0\tilde G_{ij}
 -\tilde G^{kl}\partial_0\tilde G_{ik}\partial_0\tilde G_{jl}\right),\\
\partial_\tau\tilde u
 &=\frac12|\partial_0\tilde G|_{\tilde G}^2
 +\frac12(\sqrt2+\partial_0\tilde\eta)^2-\tilde u,\\
\partial_\tau\tilde\eta
 &=\tilde u^{-1}\partial_0^2\tilde\eta
 -\frac12\tilde u^{-2}\partial_0\tilde u(\sqrt2+\partial_0\tilde\eta).
\end{aligned}
\end{equation}
where $|\partial_0G|_G^2=G^{ik}G^{jl}\partial_0G_{ij}\partial_0G_{kl}$. We write $\mathbf{\tilde g}(\tau)=(\tilde G(\tau),\tilde u(\tau),\tilde\eta(\tau))$.

Adding the Lie derivatives with respect to the DeTurck vector field $X=X^0\partial_0$, where $X^0=3\partial_0\hat u/(2\hat u)$ gives
\begin{equation}\label{rescaled.gauged.equation}
\begin{aligned}
\partial_\tau\hat G_{ij}
 &=\hat u^{-1}\left(\partial_0^2\hat G_{ij}
 -\frac{\partial_0\hat u}{2\hat u}\partial_0\hat G_{ij}
 -\hat G^{kl}\partial_0\hat G_{ik}\partial_0\hat G_{jl}\right)
 +X^0\partial_0\hat G_{ij},\\
\partial_\tau\hat u
 &=\frac12|\partial_0\hat G|_{\hat G}^2
 +\frac12(\sqrt2+\partial_0\hat\eta)^2-\hat u
 +X^0\partial_0\hat u+2\hat u\partial_0X^0,\\
\partial_\tau\hat\eta
 &=\hat u^{-1}\partial_0^2\hat\eta
 -\frac12\hat u^{-2}\partial_0\hat u(\sqrt2+\partial_0\hat\eta)
 +X^0(\sqrt2+\partial_0\hat\eta).
\end{aligned}
\end{equation}
We write $\mathbf{\hat g}=(\hat G,\hat u,\hat \eta)$ and \eqref{rescaled.gauged.equation} as $\partial_\tau\mathbf{\hat g}=Q(\mathbf{\hat g})$.

If $\partial_\tau\varphi_\tau=-X\circ\varphi_\tau$, $\varphi_0=\mathrm{id}$ and $(\hat G,\hat u,\hat \eta)$ solves \eqref{rescaled.gauged.equation}, then
\begin{equation}\label{eq:base.pullback.stability}
\tilde G=\hat G\circ\varphi_\tau,\qquad
\tilde u=(\hat u\circ\varphi_\tau)(\varphi_\tau')^2,\qquad
\tilde\eta=\hat\eta\circ\varphi_\tau+\sqrt2(\tilde\varphi_\tau-x^0)
\end{equation}
solves~\eqref{rescaled.equation}. In the last formula, $\tilde\varphi_\tau:\R\to\R$  is a lift of $\varphi_\tau$ satisfying $\partial_\tau\tilde\varphi_\tau=-\tilde X\circ\tilde\varphi_\tau$ and $\tilde \varphi_0(x^0)=x^0$ and that $\tilde\varphi_\tau-x^0$ descends to a function on $S^1$.

\begin{proposition}\label{linearization}
The linearization of $Q(\mathbf{\hat g})$ at $\mathbf g=(G_*,1,c)$, where $c$ is a constant, is given by
\[
DQ(\mathbf {g})(h,v,\xi)=L_{\mathbf g}(h,v,\xi)
=(h'',\,3v''-v+\sqrt2\xi',\,\xi''+\sqrt2v').
\]
Its kernel is $\{(H,0,c):H^T=H\text{ is constant},\ c\text{ is constant}\}$.
Zero is a semisimple eigenvalue and all other eigenvalues have strictly negative real part.
\end{proposition}
\begin{proof}
At the fixed point, $X=0$ and $\delta X^0=\frac32v'$. Direct differentiation gives the displayed operator. With
\[
\langle(h_1,v_1,\xi_1),(h_2,v_2,\xi_2)\rangle
=\int_0^\ell\left(\operatorname{tr}(G_*^{-1}h_1G_*^{-1}h_2)+v_1v_2+\xi_1\xi_2\right)dx^0,
\]
the diagonal part is self-adjoint and the first-order coupling is anti-self-adjoint. Therefore
\[
\langle L_{\mathbf g}(h,v,\xi),(h,v,\xi)\rangle
=-\|h'\|^2-3\|v'\|^2-\|v\|^2-\|\xi'\|^2.
\]
This gives the kernel. The subspace where $h$ and $\xi$ have zero mean, with $v$ arbitrary, is invariant, and the Poincar\'e inequality gives a strictly negative bound there, excluding generalized zero eigenvectors on this subspace. Its invariant complement consists of $(H,0,c)$ with $H$ and $c$ constant, on which $L_{\mathbf g}=0$. Thus zero is a semisimple eigenvalue. Ellipticity on the circle gives compact resolvent and a spectral gap.
\end{proof}

\begin{definition}[Little-H\"older spaces]
With $h^r$ and $\mathbb V$ as in Definition~\ref{tuple.equivalence}, set
\[
\mathbb X_0=h^\rho(S^1;\mathbb V),\qquad
\mathbb X_1=h^{2+\rho}(S^1;\mathbb V).
\]
For $\frac12<\beta<1-\frac\rho2$, continuous interpolation gives
$\mathbb X_\beta=(\mathbb X_0,\mathbb X_1)_\beta
=h^{1+\theta}(S^1;\mathbb V)$, where $1+\theta=\rho+2\beta$ and $\rho<\theta<1$.
\end{definition}

\begin{theorem}[Stability]\label{stability}
Fix $\mathbf g=(G_*,1,0)$ and $0<\rho<\theta<1$ as above. There is a neighborhood $\mathcal U$ of $\mathbf g$ in $h^{1+\theta}(S^1;\mathbb V)$ such that every initial datum $\mathbf{\tilde g}(0)\in\mathcal U$ gives a global solution of~\eqref{rescaled.equation}. The corresponding solution of~\eqref{rescaled.gauged.equation} converges exponentially in $C^{2+\rho}$, for $\tau\ge1$, to an equilibrium
\[
\mathbf g_\infty=(G_\infty,1,c_\infty),\qquad G_\infty\text{ constant positive definite}.
\]
In the original normalized coordinates, $\mathbf{\tilde g}(\tau)$ converges exponentially in $C^{1+\rho}$ to
\[
\bigl(G_\infty,(\varphi_\infty')^2,
 c_\infty+\sqrt2(\varphi_\infty-x^0)\bigr),
\]
where $\varphi_\infty$ is an orientation-preserving $C^{2+\rho}$ base diffeomorphism. Moreover, $\tilde\alpha=(\sqrt2+~\tilde\eta')dx^0$ converges exponentially in $C^{1+\rho}$ for $\tau\ge1$ to $\sqrt2\,d\varphi_\infty$.
If the initial data are smooth, these convergences hold in every $C^k$ norm, with constants depending on higher initial norms.
\end{theorem}
\begin{proof}
Our proof mainly follows the strategy of \cite[Theorem 3]{knopf2009convergence}.  Fix $0<\sigma<\rho$ and set $\mathbb E_0=h^\sigma(S^1;\mathbb V)$ and $\mathbb E_1=h^{2+\sigma}(S^1;\mathbb V)$. Then $\mathbb X_1\subset\mathbb E_1\subset\mathbb X_0\subset\mathbb E_0$ continuously and densely. For $\frac12<\beta<1-\frac\rho2$, both $\mathbb X_\beta$ and $\mathbb E_\beta=(\mathbb E_0,\mathbb E_1)_\beta$ embed into $C^1$.

\noindent\textbf{Step 1.} By Proposition~\ref{linearization}, $L_{\mathbf g}=DQ(\mathbf g)=L_{\mathrm{sa}}+F$, where $L_{\mathrm{sa}}(h,v,\xi)=(h'',3v''-v,\xi'')$ and $F(h,v,\xi)=(0,\sqrt2\xi',\sqrt2v')$. The complexification $L_{\mathrm{sa}}^{\mathbb C}$ is strictly elliptic and self-adjoint with non-positive spectrum. By Schauder estimates, it is sectorial (see \cite[Definition 1]{knopf2009convergence}). Since $F^{\mathbb C}$ is bounded from $\mathbb X_\beta$ to $\mathbb X_0$ and from $\mathbb E_\beta$ to $\mathbb E_0$, \cite[Proposition 2.4.1]{lunardi2012analytic} implies that $L_{\mathbf g}^{\mathbb C}$ is sectorial on $\mathbb X_0$ and $\mathbb E_0$, with domains $\mathbb X_1$ and $\mathbb E_1$, respectively.

The principal part of $Q$ is diagonal with coefficients $u^{-1},3,u^{-1}$; the remaining terms depend smoothly on the variables and their first derivatives. Near $\mathbf g$, positivity and uniform parabolicity hold. With $\nu=(\rho-\sigma)/2$, we have $\mathbb X_0=(\mathbb E_0,\mathbb E_1)_\nu$ and $\mathbb X_\beta=(\mathbb E_0,\mathbb E_1)_{\beta+\nu}$, where $\beta+\nu<1$. Thus interpolation, the little-H\"older product estimates, and Schauder graph-norm equivalence verify the hypotheses of Simonett's theorem as in \cite[Lemmas 1 and 2]{knopf2009convergence}.

By Proposition~\ref{linearization}, zero is a semisimple eigenvalue and the remaining spectrum lies in the set $\{z:\operatorname{Re}z\le-\omega_0\}$ for some $\omega_0>0$. Simonett's theorem \cite{simonett1995center}, in the form of \cite[Theorem 4]{knopf2009convergence}, gives the spectral decomposition $\mathbb X_a=\mathbb X_a^{\mathrm s}\oplus\mathbb X_a^{\mathrm c}$ for $a\in[0,1]$, where $\mathbb X_a^{\mathrm c}=\ker L_{\mathbf g}$ has dimension $N(N+1)/2+1$. Write $\pi^{\mathrm s}:\mathbb X_a\to \mathbb X_a^{\mathrm s},\pi^{\mathrm c}:\mathbb X_a\to \mathbb X_a^{\mathrm c}$ for the corresponding projections. For each $r\in\mathbb N$, there is a locally invariant $C^r$ center manifold $\Gamma_{\mathrm{loc}}^r=\{z+\gamma_d^r(z):z\in B(\mathbb X_1^{\mathrm c},\pi^{\mathrm c}\mathbf g,d)\}$, with $\gamma_d^r:\mathbb X_a^{\mathrm c}\to \mathbb X_a^{\mathrm s}$ satisfying $\gamma_d^r(\pi^{\mathrm c}\mathbf g)=\pi^{\mathrm s}\mathbf g$ and $D\gamma_d^r(\pi^{\mathrm c}\mathbf g)=0$. For $0<\omega<\omega_0$, after decreasing $d$, its attraction estimate is
\begin{equation}\label{exponential.converge.to.center.manifold}
\begin{aligned}
&\|\pi^{\mathrm s}\mathbf{\hat g}(\tau)-\gamma_d^r(\pi^{\mathrm c}\mathbf{\hat g}(\tau))\|_{\mathbb X_1}\\
&\qquad\le C\tau^{\beta-1}e^{-\omega\tau}
\|\pi^{\mathrm s}\mathbf{\hat g}(0)-\gamma_d^r(\pi^{\mathrm c}\mathbf{\hat g}(0))\|_{\mathbb X_\beta},
\end{aligned}
\end{equation}
for $\tau>0$ as long as the solution remains in $B(\mathbb X_\beta,\mathbf g,d)$.

\noindent\textbf{Step 2.} We show that the center manifold coincides locally with the moduli space $\mathcal M(\mathbf g,d)$ of fixed points. If $\mathbf g'$ is a sufficiently close fixed point, its stationary trajectory and~\eqref{exponential.converge.to.center.manifold} imply $\pi^{\mathrm s}\mathbf g'=\gamma_d^r(\pi^{\mathrm c}\mathbf g')$. Thus $\mathbf g'\in\Gamma_{\mathrm{loc}}^r$.

Conversely, every tuple $(H,1,c)$, with $H$ constant positive definite and $c\in\R$, is a fixed point. Its projection onto $\ker L_{\mathbf g}$ parametrizes a neighborhood of $\pi^{\mathrm c}\mathbf g$. By the preceding inclusion, these fixed points fill the center-manifold graph, so $\gamma_d^r\equiv\pi^{\mathrm s}\mathbf g$. Thus, after shrinking the neighborhood, $\mathcal M(\mathbf g,d)=\Gamma_{\mathrm{loc}}^r=:\Gamma$, independently of $r$. Denote its graphing function by $\gamma$.

\noindent\textbf{Step 3.} By \cite[(2.2)]{knopf2009convergence}, $\mathbb X_\beta=h^{1+\theta}$ with $1+\theta=\rho+2\beta$. Set $d_0=\|\mathbf{\tilde g}(0)-\mathbf g\|_{\mathbb X_\beta}$ and fix $0<\lambda<\omega$ and $0<\epsilon<\delta<d$. For $d_0<\epsilon$ sufficiently small, local well-posedness and smooth dependence, together with positive-time smoothing \cite[Theorem 3.1 and Remark 3.3]{simonett1995center}, give $\sup_{0\le\tau\le1}\|\mathbf{\hat g}(\tau)-\mathbf g\|_{\mathbb X_\beta}+\|\mathbf{\hat g}(1)-\mathbf g\|_{\mathbb X_1}\le C_0d_0$.

For $\tau\ge1$, write $z(\tau)=\pi^{\mathrm c}\mathbf{\hat g}(\tau)$. Since $Q(z(\tau)+\gamma(z(\tau)))=0$, local Lipschitz continuity of $Q:\mathbb X_1\to\mathbb X_0$, equivalence of norms on the center space, and~\eqref{exponential.converge.to.center.manifold} give
\begin{equation}\label{exponential.decay.time.derivative}
\begin{aligned}
\|\partial_\tau z(\tau)\|_{\mathbb X_1}
&=\|\pi^{\mathrm c}\partial_\tau\mathbf{\hat g}(\tau)\|_{\mathbb X_1}\\
&=\|\pi^{\mathrm c}Q(\mathbf{\hat g}(\tau))\|_{\mathbb X_1}\\
&=\left\|\pi^{\mathrm c}\bigl[Q(\mathbf{\hat g}(\tau))
-Q(z(\tau)+\gamma(z(\tau)))\bigr]\right\|_{\mathbb X_1}\\
&\le C\|\pi^{\mathrm s}\mathbf{\hat g}(\tau)-\gamma(z(\tau))\|_{\mathbb X_1}\\
&\le Ce^{-\lambda\tau}d_0.
\end{aligned}
\end{equation}
Choosing $\epsilon\ll\delta$, integrating this estimate from time one, and combining it with \eqref{exponential.converge.to.center.manifold}, we obtain
\begin{equation}\label{eq:initial.space.stability}
\|\mathbf{\hat g}(\tau)-\mathbf g\|_{\mathbb X_\beta}
\le C d_0(1+\lambda^{-1})<\delta<d
\end{equation}
on every interval on which the solution remains in $B(\mathbb X_\beta,\mathbf g,d)$. This bound, together with local well-posedness \cite[Theorem 3.1]{simonett1995center} and positive-time parabolic estimates, gives existence for all $\tau\ge0$. Integrating~\eqref{exponential.decay.time.derivative} from $\tau$ to infinity yields a limit $z_\infty$ in the center space. Combining its exponential convergence in $\mathbb X_1$ with \eqref{exponential.converge.to.center.manifold} gives
\[
\mathbf{\hat g}(\tau)
=\pi^{\mathrm c}\mathbf{\hat g}(\tau)+\pi^{\mathrm s}\mathbf{\hat g}(\tau)
\longrightarrow z_\infty+\gamma(z_\infty)
=:\mathbf g_\infty=(G_\infty,1,c_\infty)\in\Gamma
\]
exponentially in $\mathbb X_1$ for $\tau\ge1$. Parabolic Schauder estimates then give $\|\mathbf{\hat g}(\tau)-\mathbf g_\infty\|_{C^k}\le C_ke^{-\lambda_k\tau}d_0$ for every $k$ and $\tau\ge1$.

To return to the original coordinates, choose $\rho<\rho_1<\theta$. Parabolic smoothing and the definition of $X$ give
\[
\|X(\tau)\|_{C^{2+\rho_1}}
\le Cd_0\begin{cases}
\tau^{-1+(\theta-\rho_1)/2},&0<\tau\le1,\\
e^{-\lambda\tau},&\tau\ge1.
\end{cases}
\]
The short-time bound is integrable. Standard estimates, as in \cite[Lemma 3.5]{guenther2002stability}, therefore give an orientation-preserving $C^{2+\rho}$ diffeomorphism $\varphi_\infty$ with $\|\varphi_\tau-\varphi_\infty\|_{C^{2+\rho}}\le Ce^{-\lambda\tau}d_0$ for $\tau\ge1$. The pullback formulas~\eqref{eq:base.pullback.stability} yield the asserted exponential $C^{1+\rho}$ convergence of $\mathbf{\tilde g}$. Separately, $\tilde\alpha=\varphi_\tau^*((\sqrt2+\partial_0\hat\eta)dx^0)$ converges exponentially in $C^{1+\rho}$ to $\sqrt2\,d\varphi_\infty$, using the $C^{2+\rho}$ convergence of $\hat\eta$ and $\varphi_\tau$. For smooth initial data, $X$ is smooth up to $\tau=0$, and the same argument applies in every $C^k$ norm, with constants depending on higher initial norms.
\end{proof}

We have the following corollary. Recall Definition~\ref{tuple.equivalence}.
\begin{corollary}\label{neighborhood.around.expander}
For every $\epsilon>0$ there exists $\delta=\delta(\epsilon)>0$ such that a smooth locally $\R^N$-invariant $\alpha$-Ricci flow $(\bar g(t),\bar \alpha(t))$ with trivial linear holonomy, no fiber component, and $\int_{S^1}\bar\alpha=\sqrt2\ell$ satisfies
\[
[\mathbf{\bar g}(1)]\in\mathcal V_\delta
\quad\Longrightarrow\quad
[\mathbf{\bar g}_s(1)]\in\mathcal U_\epsilon
\qquad\text{for all }s\ge1.
\]
Here $\mathbf{\bar g}(t)=(G(t),u(t),\eta(t))$, and $\mathbf{\bar g}_s(1)=(s^{-1}G(s),s^{-1}u(s),\eta(s))$.
\end{corollary}
\begin{proof}
By~\eqref{eq:distance.to.expander.class}, choose a representative $\mathbf{\tilde g}_0$ of $[\mathbf{\bar g}(1)]$ such that
\[
d_0=\|\mathbf{\tilde g}_0-\mathbf g^*\|_{h^{1+\theta}(S^1;\mathbb V)}<\delta.
\]
Let $\mathbf{\tilde g}(\tau)$ and $\mathbf{\hat g}(\tau)$ solve~\eqref{rescaled.equation} and~\eqref{rescaled.gauged.equation}, respectively, with initial tuple $\mathbf{\tilde g}_0$ at logarithmic time $\tau=0$, corresponding to physical time $t=1$. By Theorem~\ref{stability},~\eqref{eq:initial.space.stability}, and~\eqref{eq:base.pullback.stability},
\[
\begin{aligned}
 d_{1+\rho}([\mathbf{\tilde g}(\tau)],[\mathbf g^*])
 &=d_{1+\rho}([\mathbf{\hat g}(\tau)],[\mathbf g^*])\\
 &\le\|\mathbf{\hat g}(\tau)-\mathbf g^*\|_{h^{1+\rho}}
 \le Cd_0<C\delta,\qquad \tau\ge0.
\end{aligned}
\]
The equality follows from the base pullback~\eqref{eq:base.pullback.stability}. The last bound follows from~\eqref{eq:initial.space.stability} and the inclusion $h^{1+\theta}\subset h^{1+\rho}$. Equation~\eqref{rescaled.equation} is invariant under the transformations defining the tuple classes, so $[\mathbf{\tilde g}(\tau)]=[(G(e^\tau),e^{-\tau}u(e^\tau),\eta(e^\tau))]$. For $s\ge1$, the constant fiber change $S=\sqrt s\,I$ gives
\[
[\mathbf{\bar g}_s(1)]
=[(G(s),s^{-1}u(s),\eta(s))]
=[\mathbf{\tilde g}(\log s)].
\]
Choose $\delta$ sufficiently small that Theorem~\ref{stability} applies and $C\delta<\epsilon$. The preceding estimate at $\tau=\log s$ then gives $[\mathbf{\bar g}_s(1)]\in\mathcal U_\epsilon$ for every $s\ge1$.
\end{proof}
\printbibliography

\end{document}